\documentclass[12pt,reqno]{amsart}
	\usepackage[centering,text={432pt,624pt}]{geometry}               
\usepackage{listings}
\usepackage{caption}
\usepackage{easytable}

\usepackage{diagbox}
\usepackage{graphicx}
\usepackage{tikz}
\usepackage{amssymb}
\usepackage{pdfsync}
\usepackage{hyperref}
\usepackage{verbatim} 
\usepackage{epstopdf}
\usepackage{esint} 
\usepackage{amssymb}
\usepackage{latexsym}
\usepackage{amsmath}
\usepackage{bbm}
\usepackage[mathscr]{euscript}
\usepackage{amsthm}

\newcommand{\pvector}[1]{
  \begin{pmatrix}
    #1
  \end{pmatrix}} 
\newcommand{\ddirac}[1]{
  \,\boldsymbol{\delta}\!\pvector{#1}\!} 
\renewcommand{\d}{\,{\rm d}} 

\newcommand{\sph}[1]{\mathbb{S}^{#1}}

\def\R{\mathbb{R}}

\def\N{\mathbbm{N}}
\def\Qr{\mathbb{Q}}
\def\Z{\mathbb{Z}}
\def\Co{\mathbb{C}}

\def\P{\mathbb{P}}

\def\F{\mathcal{F}}

\def\J{\mathcal{J}}

\newcommand{\te}{\theta}
\newcommand{\la}{\lambda}
\newcommand{\vphi}{\varphi}

\newcommand{\eps}{\varepsilon}

\providecommand{\ab}[1]{\vert #1\vert}
\providecommand{\abs}[1]{\Bigl\vert #1 \Bigr\vert}
\providecommand{\Abs}[1]{\biggl\vert #1 \biggr\vert}
\newcommand{\ds}{\displaystyle}

\newcommand{\cp}{\mathcal{C}}

\providecommand{\norma}[1]{\Vert #1 \Vert}

\def\Co{\mathbb{C}}

\renewcommand{\leq}{\leqslant}
\renewcommand{\geq}{\geqslant}
\newcommand{\one}{\mathbf{1}}

\renewcommand{\frak}{\mathfrak}

\theoremstyle{plain}
\newtheorem{theorem}{Theorem}[section]
\newtheorem{corollary}[theorem]{Corollary}

\newtheorem{proposition}[theorem]{Proposition}
\newtheorem{lemma}[theorem]{Lemma}

\newtheorem{conjecture}[theorem]{Conjecture}

\theoremstyle{definition}
\newtheorem{remark}[theorem]{Remark}

\numberwithin{equation}{section}

\title[Maximizers for spherical restriction]{Global maximizers for adjoint Fourier restriction inequalities on low dimensional spheres}

\author{Diogo Oliveira e Silva}
\address{
		Diogo Oliveira e Silva.
        School of Mathematics\\
        University of Birmingham\\
        Edgbaston, Birmingham\\
        B15 2TT, England.}
\email{d.oliveiraesilva@bham.ac.uk}

\author{Ren\'e Quilodr\'an}
\address{Ren\'e Quilodr\'an.}
\email{rquilodr@dim.uchile.cl}

\begin{document}

\begin{abstract}
We prove that constant functions are the unique real-valued maximizers for all $L^2-L^{2n}$ adjoint Fourier restriction inequalities on the unit sphere $\mathbb{S}^{d-1}\subset\R^d$, $d\in\{3,4,5,6,7\}$, where $n\geq 3$ is an integer. 
The proof uses tools from probability theory,  Lie theory, functional analysis, and the theory of special functions. 
It also relies on general solutions of the underlying Euler--Lagrange equation being smooth, a fact of independent interest which we establish in the companion paper \cite{OSQ19}.
We further show that complex-valued maximizers coincide with nonnegative maximizers multiplied by the character $e^{i\xi\cdot\omega}$, for some $\xi$,  thereby extending previous work of Christ \& Shao \cite{CS12b} to arbitrary dimensions $d\geq 2$ and general even exponents.
\end{abstract}

\maketitle
\setcounter{tocdepth}{1}
\tableofcontents

\section{Introduction}

The Fourier adjoint restriction operator, also known as the {\it extension operator},  of a complex-valued function $f:\sph{d-1}\to\Co$ on the unit sphere is defined at a point $x\in\R^d$ as
\begin{equation}\label{eq:NormalizationFT}
\widehat{f\sigma}_{d-1}(x)=\int_{\sph{d-1}} f(\omega) e^{-ix\cdot\omega} \,\textup{d}\sigma_{d-1}(\omega),
\end{equation}
where we denote the usual surface measure on $\sph{d-1}$ by $\sigma_{d-1}$.
The cornerstone Tomas--Stein inequality \cite{St93, To75} states that
\begin{equation}\label{eq:TS}
\|\widehat{f\sigma}_{d-1}\|_{L^q(\R^d)}\leq {\bf T}_{d,q} \|f\|_{L^2(\sph{d-1})},
\end{equation}
provided $d\geq 2$ and $q\geq q_d:=2\frac{d+1}{d-1}$.
Here  ${\bf T}_{d,q}$ denotes the optimal constant given by
\begin{equation}\label{eq:bestconstant}
{\bf T}_{d,q}^q=\sup_{\mathbf{0}\neq f\in L^2}\Phi_{d,q}(f),
\end{equation}
where the functional $\Phi_{d,q}:L^2(\sph{d-1})\to\R$ is defined via
\begin{equation}\label{eq:PhidqDef}
\Phi_{d,q}(f)=\frac{\|\widehat{f\sigma}_{d-1}\|^q_{L^q(\R^d)}}{\|f\|^q_{L^2(\sph{d-1})}}.
\end{equation}
By a {\it maximizer} of \eqref{eq:TS} we mean a nonzero, complex-valued function $f\in L^2(\sph{d-1})$ for which
$\|\widehat{f\sigma}_{d-1}\|_{L^q(\R^d)}= {\bf T}_{d,q} \|f\|_{L^2(\sph{d-1})}$.
In this case, we will also say that $f$ maximizes the functional $\Phi_{d,q}$.

There has been a surge in attention given in the recent literature to Sharp Fourier Restriction Theory.
In particular, the sharp form of inequality \eqref{eq:TS} has motivated a great deal of interesting work.
It has been shown in \cite{FVV11} that complex-valued maximizers of inequality \eqref{eq:TS} exist in the non-endpoint case $q>q_d$, in any dimension $d\geq 2$.
If $q$ is an even integer, then maximizers are known to exist among the class of non-negative functions.
Existence of maximizers in the endpoint case $q=q_d$ was first established in \cite{CS12a} for $d=3$, then in \cite{Sh16} for $d=2$ and in \cite{FLS16} for all $d\geq 2$, the cases $d\geq 4$ being conditional on a celebrated conjecture concerning Gaussian maximizers for the Fourier extension inequality on the paraboloid. 

{If $q=\infty$, then one easily checks that the unique real-valued maximizers of \eqref{eq:TS} are the constant functions.}
It has been proved \cite{COS15,Fo15} that constant functions are the unique real-valued maximizers of $\Phi_{d,4}$ when $d\in\{3,4,5,6,7\}$, while evidence has been provided in \cite{CFOST15, OSTZK18} to support the natural conjecture that constants maximize $\Phi_{2,6}$ as well. 

So far no explicit maximizer of \eqref{eq:TS} has been found  when $q\in (4,\infty)$, in any dimension $d\geq 2$. 
In this paper, we consider the case of even integers $q$ in low dimensions.
Our main result is the following.

\begin{theorem}\label{thm:MainThm}
	Let $d\in\{3,4,5,6,7\}$ and $q\geq 6$ be an even integer. 
	Then constant functions are the unique real-valued maximizers of the functional $\Phi_{d,q}$.
	The same conclusion holds for $d=2$ and even integers $q\geq 8$, provided that constants maximize the functional $\Phi_{2,6}$.
\end{theorem}

\noindent 
Our methods are tailored to handle even exponents $q\in 2\N$, and partly rely on the aforementioned works \cite{COS15, Fo15} which are only available in dimensions $d\leq 7$. 
It remains an interesting open problem to determine whether constant functions maximize the functional $\Phi_{d,q}$ in any dimension $d\geq 8$, for any exponent $q\in[q_d,\infty)$, as well as for the remaining exponents $q\in [q_d,\infty)\setminus 2\N$ in the lower dimensional cases $2\leq d\leq 7$. 

It is natural to ask about more general complex-valued maximizers.
Our second result shows that any complex-valued maximizer of \eqref{eq:TS} coincides with an even, nonnegative maximizer multiplied by the character $e^{i\omega\cdot\xi}$ for some $\xi$, provided $q$ is an even integer.

\begin{theorem}\label{thm:complex}
Let $d\geq 2$ and $q\geq 2\frac{d+1}{d-1}$ be an even integer.
Then each complex-valued maximizer of the functional $\Phi_{d,q}$ is of the form
\begin{equation}\label{eq:CvalMax}
c e^{i\xi\cdot\omega}F(\omega),
\end{equation}
for some $\xi\in\R^d$, some $c\in\Co\setminus\{0\}$, and some nonnegative maximizer $F$ of $\Phi_{d,q}$ satisfying $F(\omega)=F(-\omega)$, for every $\omega\in\sph{d-1}$.
\end{theorem}
\noindent Theorem \ref{thm:complex} extends \cite[Theorem 1.2]{CS12b} to arbitrary dimensions and general even exponents.
As an immediate consequence of Theorems \ref{thm:MainThm} and \ref{thm:complex}, we obtain the following partial extension of \cite[Theorem 1]{COS15}.

\begin{corollary}
	Let $d\in\{3,4,5,6,7\}$ and $q\geq 4$ be an even integer.
	Then  all complex-valued maximizers of the functional $\Phi_{d,q}$ are given by
	$$f(\omega)=ce^{i\xi\cdot\omega},$$
	for some $\xi\in\R^d$ and $c\in\Co\setminus\{0\}$.
	The same conclusion holds for $d=2$ and even integers $q\geq 8$, provided that constants maximize $\Phi_{2,6}$.
\end{corollary}

In the companion paper \cite{OSQ19}, we explored the fact that a maximizer of \eqref{eq:TS}  satisfies the Euler--Lagrange equation 
\begin{equation}\label{eq:ELnonconv}
\Bigl(|\widehat{f\sigma}_{d-1}|^{q-2} \widehat{f\sigma}_{d-1}\Bigr)^{\vee}\Bigl\vert_{\sph{d-1}}
=\lambda \|f\|_{L^2(\sph{d-1})}^{q-2} f,
\quad\sigma_{d-1}\text{-a.e. on }\sph{d-1},
\end{equation}
with $\lambda={\bf T}_{d,q}^q$.
Nonzero solutions of \eqref{eq:ELnonconv} corresponding to generic values $\lambda\in\Co$ are called {\it critical points} of the functional $\Phi_{d,q}$.
If $q=2n$ is an even integer, then the Tomas--Stein inequality \eqref{eq:TS} can be equivalently stated in convolution form via Plancherel's Theorem as 
\begin{equation}\label{eq:TSconv}
\|(f\sigma_{d-1})^{\ast n}\|^2_{L^2(\R^d)}\leq (2\pi)^{-d}{\bf T}^{2n}_{d,2n} \|f\|^{2n}_{L^2(\sph{d-1})},
\end{equation}
where the $n$-fold convolution measure $(f\sigma_{d-1})^{\ast n}$ is recursively defined for integral values of $n\geq 2$ via
\begin{equation}\label{eq:recsigmaast}
(f\sigma_{d-1})^{\ast 2}=f\sigma_{d-1}\ast f\sigma_{d-1}, \text{ and }
 (f\sigma_{d-1})^{\ast (n+1)}=(f\sigma_{d-1})^{\ast n}\ast f\sigma_{d-1}.
\end{equation}   
The functional $\Phi_{d,2n}$ can then be recast as
\begin{equation}\label{eq:PhiConvForm}
\Phi_{d,2n}(f)
=(2\pi)^d\frac{\|(f\sigma_{d-1})^{\ast n}\|_{L^2(\R^d)}^2}{\|f\|_{L^2(\sph{d-1})}^{2n}},
\end{equation}
and the  Euler--Lagrange equation \eqref{eq:ELnonconv} translates into 
\begin{equation}\label{eq:simpleEL}
\Bigl((f\sigma_{d-1})^{\ast n}\ast(f_\star\sigma_{d-1})^{\ast(n-1)}\Bigr) \Big\vert_{\mathbb S^{d-1}}=(2\pi)^{-d}\lambda \|f\|_{L^2(\sph{d-1})}^{2n-2} f,\quad\sigma_{d-1}\text{-a.e. on }\mathbb S^{d-1},
\end{equation}
where $f_\star$ denotes the {\it conjugate reflection} of $f$ around the origin, 
defined via 
$$f_\star(\omega)=\overline{f(-\omega)},\quad \text{ for all }\omega\in\sph{d-1}.$$
A function $f:\sph{d-1}\to\Co$ is said to be {\it antipodally symmetric} if $f=f_\star$,
in which case $\widehat{f\sigma}_{d-1}$ is easily seen to be real-valued.

\noindent In  \cite{OSQ19}, we considered a more general version of the Euler--Lagrange equation \eqref{eq:simpleEL}, and proved that all the corresponding solutions are $C^\infty$-smooth. 
A particular consequence which will be relevant for the purposes of the present paper is the following generalization of \cite[Theorem 1.1]{CS12b}. 

\begin{theorem}[\cite{OSQ19}]\label{thm:smoothnessTheorem}
Let $d\geq 2$ and $q\geq 2\frac{d+1}{d-1}$ be an even integer.
If $f\in L^2(\sph{d-1})$ is a critical point of the functional  $\Phi_{d,q}$, then $f\in C^\infty(\sph{d-1})$. 
In particular, maximizers of $\Phi_{d,q}$ are $C^\infty$-smooth.
\end{theorem}

\noindent It would be useful to extend Theorem \ref{thm:smoothnessTheorem}  to the case of general exponents $q\in[q_d,\infty)$, as it is one of the missing ingredients to turn our next result into an unconditional one. 
In fact, Theorem \ref{prop:conditionalMax} provides some sufficient conditions for constant functions to be the unique real-valued maximizers among the class of continuously differentiable functions, $C^1(\sph{d-1})$, and follows from the methods developed in the present paper.

\begin{theorem}\label{prop:conditionalMax}
	Let $d\geq 2$. Then there exists $q_\star=q_\star(d)\in[2\frac{d+1}{d-1},\infty)$ with the following property. If constant functions maximize $\Phi_{d,q}$ for some $q\geq q_\star$, then any real-valued, continuously differentiable maximizer of $\Phi_{d,q+2}$ is a constant function. 
\end{theorem}

\noindent As a further application of our methods, we show that, if $d\in\{3,4,5,6,7\}$, then the optimal constant $\mathbf{T}_{d,q}$ is continuous at $q=\infty$.
More precisely, the following result holds.
\begin{theorem}\label{prop:continuityTInfty}
	Let $d\in\{3,4,5,6,7\}$. Then
	\begin{equation}\label{eq:valueLimitT}
	\lim_{q\to\infty}\mathbf{T}_{d,q}=\mathbf{T}_{d,\infty}=\biggl(\frac{2\pi^{\frac{d}{2}}}{\Gamma(\frac{d}{2})}\biggr)^{\frac{1}{2}}.
	\end{equation}
\end{theorem}

\subsection{Outline}
In \S \ref{sec:symm}, we investigate monotonicity properties of the functional  $\Phi_{d,q}$.
In particular, we show that the value of $\Phi_{d,q}$ does not decrease under a certain antipodal symmetrization, 
via an argument from \cite{BOSQ18} that does {\it not} depend on the possible convolution structure of the functional $\Phi_{d,q}$, and can therefore be extended to handle general exponents $q\geq q_d$.
Under the additional assumption that $q\in 2\N$ is an even integer, we further observe as in \cite{CS12a} that $\Phi_{d,q}(f)\leq \Phi_{d,q}(|f|)$.
Both of these monotonicity results are completed with a  characterization of the cases of equality, and together lead to a proof of Theorem \ref{thm:complex}.

In \S \ref{sec:explicitConvolutions}, we study convolution measures on the sphere.
We rely on explicit formulae for the 2- and 3-fold convolution measures $\sigma_{d-1}^{\ast 2}$ and $\sigma_{d-1}^{\ast 3}$ which enable the exact computation of some of the corresponding $L^2$-norms.
We further discuss the case of higher order $k$-fold convolutions,
and highlight a fundamental distinction between even and odd dimensions.
The latter turns out to be related to the theory of uniform random walks in $d$-dimensional Euclidean space.
Given that this is a very classical field of research in probability theory, it is remarkable that some of the explicit formulae which we rely upon only seem to have appeared in the literature a few years ago; see \cite{BS, GP12}, and Appendix \ref{sec:AppendixGP} for further details.

In \S \ref{sec:FA}, we consider the Euler--Lagrange equation \eqref{eq:ELnonconv} from a different point of view,
by rewriting it as the eigenvalue problem for a certain integral operator, denoted $T_f$.
We prove that $T_f$ is self-adjoint, positive definite, and trace class. 
Some inspiration stems from the  seemingly unknown but very interesting work \cite{GRR69},
where a similar route was undertaken in order to investigate some extremal positive definite functions on the real line and on the periodic torus $\R/\Z$.

In \S \ref{sec:newold}, we show how certain symmetries of the functional $\Phi_{d,q}$ imply the existence of further eigenfunctions for the operator $T_f$ considered in \S\ref{sec:FA}. 
This step relies on some non-trivial information about the Lie theory underlying the special orthogonal group, SO$(d)$, and its Lie algebra, $\frak{so}(d)$.
In particular, we are led to the problem of determining the minimal codimension of the proper subalgebras of $\frak{so}(d)$.
This question has been addressed in the literature \cite{AFG12,Ho65}, and the answer reveals a curious difference that occurs in the four-dimensional case $d=4$. 

All of the aforementioned ingredients come together nicely in \S \ref{sec:BconstantMax} and \S \ref{sec:ProofMain}.
We use them to show that
{\it if constant functions are known to maximize the functional $\Phi_{d,q}$, for some $q\in2\N$,  
 then they necessarily are the unique real-valued maximizers for $\Phi_{d,q+2}$},
 provided that a certain inequality between the values $\Phi_{d,q}(\one)$ and $\Phi_{d,q+2}(\one)$ holds. 
In \S \ref{sec:BconstantMax}, we thus reduce the proof of Theorem \ref{thm:MainThm} to the verification of a single numerical inequality, which in turn is proved in \S \ref{sec:ProofMain}, following two steps. 
We first verify that such an inequality holds for ``small'' values of $q\in 2\N$, 
and here dimensional parity is again seen to play a role. 
If $d$ is odd, then the explicit convolution formulae from \S \ref{sec:explicitConvolutions} lead to an analytic proof of the inequality in question.
If $d$ is even, then we resort to a rigorous numerical verification; see Appendix \ref{sec:numerics} for details. 
The second step is an analytic proof of the inequality in question which applies to all $q\geq q_\star(d)$, for some $q_\star(d)\in [8,12]\cap 2\N$.
For the sake of clarity and to better illustrate the main ideas, we first deal with the case $d=3$, where matters reduce to the analysis of weighted integrals of powers of the sinc function, a topic which has a rich history in connection to the cube slicing problem; see  \cite{Ba86,NP00}.
For general $2\leq d\leq 7$, matters are less straightforward.
We are led to establish very fine tailored asymptotics for certain integrals involving powers of Bessel functions, and obtain precise estimates with absolute error smaller than what is needed for our purposes. 

Integrals involving products of Bessel functions play an important role in many areas of mathematics, and have made prominent appearances in the context of Sharp Fourier Restriction Theory \cite{COS15, CFOST15, COSS19, FOS17, OST17, OSTZK18}. 
Further examples include cube, polydisc and cylinder slicing \cite{Ba86,Br11,Br13,Di15,Di17,KK13,KK19,NP00,OP00}, Khintchine-type inequalities \cite{Ko14,KK01,Mo17,NP00,Sa85}, the aforementioned case of uniform random walks in Euclidean space (see \S \ref{sec:randomwalks} below), as well as more applied topics; see e.g.\@ the introduction in \cite{vDC06a,vDC06b}. 
In \S \ref{sec:fullAsymp}, we discuss the general asymptotic expansion for a certain class of Bessel integrals which contains those considered in \S \ref{sec:ProofMain}. As a consequence, we provide short proofs of Theorems \ref{prop:conditionalMax} and \ref{prop:continuityTInfty}.

\subsection{Notation}
We reserve the letter $d$ to denote the dimension of the ambient space $\R^d$, and
$q_d$ to denote the endpoint Tomas--Stein exponent, $q_d=2\frac{d+1}{d-1}$.
The set of natural numbers is $\N=\{1,2,3,\ldots\}$, and $\N_0=\N\cup\{0\}$.
The constant function is denoted ${\bf 1}:\sph{d-1}\to\R$, ${\bf 1}(\omega)\equiv 1$, and the zero function is denoted ${\bf 0}:\sph{d-1}\to\R$, ${\bf 0}(\omega)\equiv 0$.
We find it convenient to define
\begin{equation}\label{eq:omegad}
\omega_{d-1}:=\sigma_{d-1}(\sph{d-1})=\int_{\sph{d-1}}{\bf 1}(\omega)\,\textup{d}\sigma_{d-1}(\omega)=\frac{2\pi^{\frac d2}}{\Gamma(\frac d2)}.
\end{equation}
The indicator function of a set $E\subset \R^d$ is denoted by $\mathbbm{1}_E$ or $\mathbbm{1}(E)$. 
Real and imaginary parts of a complex number $z\in\Co$ will be denoted by $\Re(z)$ and $\Im(z)$, respectively.
We let $B_r\subset\R^d$ denote the closed ball of radius $r>0$ centered at the origin, and
will continue to denote by $(f\sigma_{d-1})^{\ast k}$ the $k$-fold convolution measure, recursively defined in \eqref{eq:recsigmaast}.
We emphasize that $\sigma_{d-1}^{\ast k}$ should not be confused with the usual $k$-fold product measure, denoted $\sigma_{d-1}^k$.

\section{Symmetrization and $\Co$-valued maximizers}\label{sec:symm}

In the search of maximizers for the functional $\Phi_{d,q}$, defined in \eqref{eq:PhidqDef} for general $q\geq q_d$, we would like to restrict attention to antipodally symmetric functions $f=f_\star$.
That this can be done is hinted by the previous works \cite{CFOST15, CS12a, Fo15, Sh16}, and was accomplished in \cite{BOSQ18}.
For the convenience of the reader, we present  a brief outline of the argument in \cite[Prop.\@ 6.7]{BOSQ18}. Interestingly enough and contrary to the aforementioned works, the  argument does not depend on the convolution form of the functional $\Phi_{d,q}$, and therefore extends to general exponents $q$.

\begin{proposition}\label{prop:symm}
Given $d\geq 2$ and $q\geq q_d$, let  $\mathbf{0}\neq f\in L^2(\sph{d-1})$. Then
\begin{equation}\label{eq:SymIncreases}
\Phi_{d,q}(f)
\leq
\sup_{\substack{\mathbf{0}\neq F\in L^2\\F=F_\star}}
\Phi_{d,q}(F).
\end{equation}
If equality holds in \eqref{eq:SymIncreases}, then there exist $\mathbf{0}\neq F\in L^2(\sph{d-1})$ and $\kappa\in\Co\setminus\{0\}$, satisfying $F=F_\star$ and $f=\kappa F$.  
\end{proposition}

\begin{proof}
We can assume $q\neq\infty$, otherwise the right-hand side of \eqref{eq:SymIncreases} equals $\Phi_{d,\infty}(\one)$ and equality holds if and only if $f(\omega)=\kappa e^{i\xi\cdot\omega}$, for some $\kappa\in\Co\setminus\{0\}$ and $\xi\in\R^d$.
Let  $\mathbf{0}\neq f:\sph{d-1}\to\Co$ be given. 
We may decompose 
 $f=F+iG$, where $F=\frac{1}{2}(f+f_\star)$ and $G=\frac{1}{2i}(f-f_\star)$. 
 The functions $F$ and $G$ are complex-valued and antipodally symmetric, 
and one easily checks that $\norma{f}_{L^2}^2=\norma{F}_{L^2}^2+\norma{G}_{L^2}^2$.
By linearity of the extension operator,
$\widehat{f\sigma}_{d-1}=\widehat{F\sigma}_{d-1}+i\widehat{G\sigma}_{d-1}.$
The antipodal symmetry of $F$ and $G$ implies that $\widehat{F\sigma}_{d-1}$ and $\widehat{G\sigma}_{d-1}$ are real-valued functions. Consequently,
$|\widehat{f\sigma}_{d-1}|^2=|\widehat{F\sigma}_{d-1}|^2+|\widehat{G\sigma}_{d-1}|^2,$
and Minkowski's inequality on $L^{q/2}(\R^d)$ then implies 
\begin{equation}\label{eq:Mink}
\big\||\widehat{f\sigma}_{d-1}|^2\big\|_{L^{q/2}}\leq\big\||\widehat{F\sigma}_{d-1}|^2\big\|_{L^{q/2}}
+\big\||\widehat{G\sigma}_{d-1}|^2\big\|_{L^{q/2}}.
\end{equation}
In turn, this can be rewritten as
\[
\big\|\widehat{f\sigma}_{d-1}\big\|_{L^q}^2\leq\big\|\widehat{F\sigma}_{d-1}\big\|_{L^q}^2
+\big\|\widehat{G\sigma}_{d-1}\big\|_{L^q}^2.
\]
We conclude that
\begin{equation}\label{eq:twoineqs}
\frac{\|\widehat{f\sigma}_{d-1}\|_{L^q}^2}{\|f\|_{L^2}^2}
\leq\frac{\|\widehat{F\sigma}_{d-1}\|_{L^q}^2+\|\widehat{G\sigma}_{d-1}\|_{L^q}^2}{\|F\|_{L^2}^2+\|G\|_{L^2}^2}
\leq \max\left\{\frac{\|\widehat{F\sigma}_{d-1}\|_{L^q}^2}{\|F\|_{L^2}^2},\frac{\|\widehat{G\sigma}_{d-1}\|_{L^q}^2}{\|G\|_{L^2}^2}\right\},
\end{equation}
where in the latter expression 
the ratio is set to zero if either $F$ or $G$ happen to vanish identically.
In order for equality to hold in \eqref{eq:SymIncreases}, both inequalities in \eqref{eq:twoineqs} must be equalities.
Then necessarily one of the following alternatives must hold:
\begin{itemize}
	\item $\norma{G}_{L^2}=0$, in which case $f=F$, and so $\kappa=1$ and $f$ is antipodally symmetric; or
	\item $\norma{F}_{L^2}=0$, in  which case  $f=iG$, and so $\kappa=i$ and $if$ is antipodally symmetric; or
	\item $\norma{F}_{L^2}\norma{G}_{L^2}\neq 0$  and ${\norma{\widehat{F\sigma}_{d-1}}_{L^q}}{\norma{F}_{L^2}^{-1}}={\norma{\widehat{G\sigma}_{d-1}}_{L^q}}{\norma{G}_{L^2}^{-1}}$.
\end{itemize}
In the latter case, equality must hold in the application of Minkowski's inequality \eqref{eq:Mink}. Since 
$q/2\in (1,\infty)$, we conclude the existence of a constant $\mu>0$ such that 
\begin{equation}\label{eq:equalityAfterMink}
\ab{\widehat{F\sigma}_{d-1}}=\mu\ab{\widehat{G\sigma}_{d-1}},\text{ a.e. in }\R^d.
\end{equation} 
The functions $\widehat{F\sigma}_{d-1}, \widehat{G\sigma}_{d-1}$ are real-valued and analytic.
The latter property follows from the fact that they are both the Fourier transform of compactly supported, finite measures. 
Since $\widehat{F\sigma}_{d-1}, \widehat{G\sigma}_{d-1}$ are continuous and not identically zero, there must exist $\xi_0\in\R^d$ and $r>0$, such that either 
\[\widehat{F\sigma}_{d-1}(\xi)=\mu \widehat{G\sigma}_{d-1}(\xi),\,\text{ for every }\ab{\xi-\xi_0}<r,\] 
or else 
\[\widehat{F\sigma}_{d-1}(\xi)=-\mu \widehat{G\sigma}_{d-1}(\xi),\text{ for every }\ab{\xi-\xi_0}<r.\] 
Analiticity then forces $F=\mu G$ in the first case, and $F=-\mu G$ in the second case.
It follows that 
$f=\kappa F$, where $\kappa$ equals either $1+ i/\mu$  or $1- i/\mu$, and this concludes the proof of the proposition.
\end{proof}

If $q=2n$ is an even integer, then  the following result shows that the value of $\Phi_{d,2n}(f)$ does not decrease if the function $f$ is replaced by its absolute value $|f|$.
\begin{lemma}\label{lem:posequal}
Given $d\geq 2$, let $n\in\N$ be such that $n\geq 3$ if $d=2$ and  $n\geq 2$ if $d\geq 3$.
Let $f\in L^2(\sph{d-1})$. Then
\begin{equation}\label{eq:positivity}
\|(f\sigma_{d-1})^{\ast n}\|_{L^2(\R^d)}\leq \|(|f|\sigma_{d-1})^{\ast n}\|_{L^2(\R^d)},
\end{equation}
with equality if and only if there exists a measurable function $h:{B_n}\to\Co$ such that 
\begin{equation}\label{eq:posequiv}
\prod_{j=1}^n f(\omega_j)=h\Big(\sum_{j=1}^n \omega_j\Big)\Big|\prod_{j=1}^n f(\omega_j)\Big|,
\end{equation}
for $\sigma_{d-1}^n$-a.e. $(\omega_1,\ldots,\omega_n)\in(\sph{d-1})^n$.
\end{lemma}

\begin{proof}
The proof follows similar lines to those of \cite[Lemma 8]{COS15}. 
We have that
\begin{equation}\label{eq:fsigma}
(f\sigma_{d-1})^{\ast n}(\xi) = \int_{(\sph{d-1})^n} \prod_{j=1}^n\,f(\omega_j) \ddirac{\xi- \sum_{j=1}^n\omega_j}\, \prod_{j=1}^n\d\sigma_{d-1}(\omega_j),
\end{equation}
where $\boldsymbol{\delta}$ denotes the $d$-dimensional Dirac delta distribution; see \cite[Appendix A]{FOS17} and the references therein for a discussion of the relevant delta-calculus.
Identity \eqref{eq:fsigma} readily implies
\begin{equation}\label{eq:pos2}
|(f\sigma_{d-1})^{\ast n}(\xi)|\leq(|f|\sigma_{d-1})^{\ast n}(\xi),\quad\text{for all } \xi\in\R^d,
\end{equation}
 which, upon integration, yields \eqref{eq:positivity}.
If equality holds in \eqref{eq:positivity}, then we must have equality in \eqref{eq:pos2}, for almost every $\xi\in\R^d$. 
For each such $\xi\in\R^d$, there exists $h(\xi)\in\Co$ satisfying 
\begin{equation}\label{eq:preint}
\prod_{j=1}^n f(\omega_j)=h(\xi)\Big|\prod_{j=1}^n f(\omega_j)\Big|,
\quad\text{for } \Sigma_\xi\text{-a.e. } (\omega_1,\ldots,\omega_n)\in(\sph{d-1})^n,
\end{equation}
where the singular measure $\Sigma_\xi$ on $(\sph{d-1})^n$ is given by 
$$\d\Sigma_\xi(\omega_1,\ldots,\omega_n)=\ddirac{\xi-\sum_{j=1}^n\omega_j}\prod_{j=1}^n\d\sigma_{d-1}(\omega_j).$$
Integrating \eqref{eq:preint} with respect to the measure $\Sigma_\xi$ on $(\sph{d-1})^n$, we find that
$$(f\sigma_{d-1})^{\ast n}(\xi)=h(\xi)(|f|\sigma_{d-1})^{\ast n}(\xi),$$
which in particular shows that $h$ is indeed measurable.
We may now reason similarly to the second part of \cite[\S 2.3]{COS15} and arrive at \eqref{eq:posequiv}.
We claim that the set
$$E:=\Big\{(\omega_1,\ldots,\omega_n)\in(\sph{d-1})^n:\,\prod_{j=1}^n f(\omega_j)\neq h\Big(\sum_{j=1}^n \omega_j\Big)\Big|\prod_{j=1}^n f(\omega_j)\Big|\Big\}$$
satisfies $\sigma_{d-1}^{ n}(E)=0$. Indeed, appealing to Fubini's Theorem,
\begin{align*}
\sigma_{d-1}^{ n}(E)
&=\int_{(\sph{d-1})^n} \mathbbm{1}_E(\omega_1,\ldots,\omega_n) \prod_{j=1}^n\d\sigma_{d-1}(\omega_j)\\
&=\int_{(\sph{d-1})^n} \int_{\R^d} \mathbbm{1}_E(\omega_1,\ldots,\omega_n) \ddirac{\xi-\sum_{j=1}^n\omega_j} \d\xi\prod_{j=1}^n\d\sigma_{d-1}(\omega_j)\\
&=\int_{\R^d} \Big(\int_{(\sph{d-1})^n}  \mathbbm{1}_E(\omega_1,\ldots,\omega_n) \ddirac{\xi-\sum_{j=1}^n\omega_j} \prod_{j=1}^n\d\sigma_{d-1}(\omega_j)\Big)\d\xi\\
&=0,
\end{align*}
where the last identity holds in view of \eqref{eq:preint}.
This implies \eqref{eq:posequiv}, as desired.

Conversely, if \eqref{eq:posequiv} holds, we may argue as in the previous paragraph to conclude that identity \eqref{eq:preint} holds, for almost every $\xi\in\R^d$.
Then equality holds in \eqref{eq:pos2}, and therefore in \eqref{eq:positivity} as well.
This concludes the proof of the lemma.
\end{proof}

We are now ready to prove Theorem \ref{thm:complex}.
\begin{proof}[Proof of Theorem \ref{thm:complex}]
Let $d\geq 2$ and $n=\frac q2\geq \frac{d+1}{d-1}$ be integers, 
and let $\mathbf{0}\neq f\in L^2(\sph{d-1})$ be a complex-valued maximizer of the corresponding inequality \eqref{eq:TS}. 
Since $q=2n$ is an even integer, inequality \eqref{eq:TS} can be rewritten in convolution form \eqref{eq:TSconv}.
In view of Lemma \ref{lem:posequal}, the function $|f|$ is likewise a maximizer of \eqref{eq:TS}. 
Write $f=e^{i\varphi}F$, where the function $\varphi$ is real-valued and measurable, and $F=|f|$ is a nonnegative maximizer of  \eqref{eq:TS}.
By Proposition \ref{prop:symm}, we necessarily have $F=F_\star$.

We claim that there exists $\delta>0$, such that 
\begin{equation}\label{eq:boundedfrombelow}
F(\omega)\geq\delta, \quad\text{ for a.e. } \omega\in\sph{d-1}.
\end{equation}
We postpone the proof of the claim for a moment, and proceed with the main argument.
As a consequence of \eqref{eq:boundedfrombelow}, we have that $(F\sigma_{d-1})^{\ast n}(\zeta)>0$, for almost every $\zeta\in B_n\subset \R^d$ (and naturally, $(F\sigma_{d-1})^{\ast n}(\zeta)=0$ if $|\zeta|>n$).
Since $|(f\sigma_{d-1})^{\ast n}(\zeta)|\leq |(F\sigma_{d-1})^{\ast n}(\zeta)|$ for almost every $\zeta\in\R^d$, and $f$ is a maximizer by assumption, we further have that
\[|(f\sigma_{d-1})^{\ast n}(\zeta)|= |(F\sigma_{d-1})^{\ast n}(\zeta)|, \quad\text{for a.e. } \zeta\in B_n.\]
By Lemma \ref{lem:posequal}, this implies the existence of a measurable function $\psi:B_n\to\Co$, such that
\begin{equation}\label{eq:characterEqual}
\prod_{j=1}^n e^{i\varphi(\omega_j)}=\exp\Big({i\psi(\sum_{j=1}^n \omega_j)}\Big),\quad\text{for } \sigma_{d-1}^n\text{-a.e. } (\omega_1,\ldots,\omega_n)\in(\sph{d-1})^n,
\end{equation}
where we used the fact that $F>0$ to cancel out the appropriate terms.
But then \cite[Theorem 4]{COS15} implies the existence of $c\in\Co\setminus\{0\}$ and $\nu\in\Co^d$, such that $e^{i\varphi(\omega)}=c e^{\nu\cdot\omega}$, for almost every $\omega\in\sph{d-1}$.
Since $|e^{i\varphi}|\equiv 1$, we must have $|c|=1$ and $\Re(\nu)=0$.
The conclusion is that $e^{i\varphi(\omega)}=c e^{i\xi\cdot\omega}$, for some $\xi\in\R^d$ and unimodular $c\in\Co$, and so the proof is complete modulo the verification of the claim.

In order to prove the claim \eqref{eq:boundedfrombelow}, let us start by recalling some facts about convolutions of regular functions supported on $\sph{d-1}$. With $F$ as above, from Theorem \ref{thm:smoothnessTheorem} it follows that $F\in C^\infty(\sph{d-1})$. The  convolution $F\sigma_{d-1}\ast F\sigma_{d-1}$ is supported on $B_2\subset\R^d$, and can be written explicitly in the following integral form:
\[ (F\sigma_{d-1}\ast F\sigma_{d-1})(x)=\int_{\Gamma_x}F(x-\nu)F(\nu)\d\bar\sigma_x(\nu)\,\bigl(\sigma_{d-1}\ast\sigma_{d-1}\bigr)(x),
\]
for all $x\in\R^d$ satisfying $0<\ab{x}<2$.
Here, $\Gamma_x=\sph{d-1}\cap(x+\sph{d-1})$ is a $(d-2)$-dimensional sphere, and $\bar\sigma_x$ denotes the normalized surface measure on $\Gamma_x$, satisfying $\bar\sigma_x(\Gamma_x)=1$; see \cite[\S 4.1]{OSQ19}. 
For $k\geq 3$, the $k$-fold convolution $(F\sigma_{d-1})^{\ast k}$ can be written recursively via
\begin{equation}\label{eq:kconvF}
(F\sigma_{d-1})^{\ast k}(x)=\int_{\sph{d-1}}\bigl(F\sigma_{d-1})^{\ast (k-1)}(x-\omega)\, F(\omega)\d\sigma_{d-1}(\omega),\text{ for all }x\in\R^d.
\end{equation}
 If $d=2$ and $k\geq 5$, or if $d,k\geq 3$, it then follows from \cite[Prop.\@ 3.1]{OSQ19} and \cite[Prop.\@ 4.3 and Remark 4.4]{OSQ19} that
  the function $(F\sigma_{d-1})^{\ast k}$ 
is H\"older continuous on $\R^d$ of some parameter $\alpha=\alpha(d,k)>0$. 
As a consequence, the Euler--Lagrange equation \eqref{eq:simpleEL} satisfied by $F$,
\begin{equation}\label{eq:ELforF}
(F\sigma_{d-1})^{\ast(2n-1)} \big\vert_{\mathbb S^{d-1}}=(2\pi)^{-d}{\bf T}_{d,2n}^{2n} \|F\|_{L^2(\sph{d-1})}^{2n-2} F,
\end{equation}
holds {\it everywhere} on $\sph{d-1}$, and not merely $\sigma_{d-1}$-almost everywhere.

Coming back to the proof of \eqref{eq:boundedfrombelow}, start by assuming that $n\geq 3$.
We follow some of the steps outlined in \cite[Lemma 4.1]{CS12b}. 
Since $F$ does not vanish identically, there exists $\omega_0\in\sph{d-1}$ for which $F(\omega_0)>0$.
Since $F(\omega_0)=F(-\omega_0)$, $F$ is continuous, and $F\geq 0$ everywhere, there exists a neighborhood of the origin on which the function $(F\sigma_{d-1})^{\ast(2n-2)}$ is uniformly bounded from below by some strictly positive number. 
To see why this is necessarily the case, consider a closed cap  $\cp\subseteq\sph{d-1}$ on which $F$ is bounded from below away from zero; in other words, $F(\omega)>c$, for some $c>0$ and every $\omega\in\cp$.
Let $E=(n-1)(\cp-\cp)$ denote the Minkowski sum of $n-1$ copies of the difference set $\cp-\cp$. 
Since the function  $\mathbbm{1}_\cp\sigma_{d-1}\ast\mathbbm{1}_{-\cp}\sigma_{d-1}$ is supported on $\cp-\cp$, and 
\[\norma{\mathbbm{1}_\cp\sigma_{d-1}\ast\mathbbm{1}_{-\cp}\sigma_{d-1}}_{L^1(\R^d)}=\sigma_{d-1}(\cp)^2>0,\]
 we  conclude that $\cp-\cp$ has positive Lebesgue measure. Since $\cp-\cp$ is also symmetric with respect to the origin, and $n\geq 3$, it then follows by a theorem of Steinhaus \cite{St20, St72} that the set $E$ contains a neighborhood of the origin.
But $E$ is contained in the support of the function $(F\sigma_{d-1})^{\ast(2n-2)}$, and
in fact it follows  from \eqref{eq:kconvF} that 
$(F\sigma_{d-1})^{\ast(2n-2)}>0$ on the interior of $E$. 
As a consequence, for any given $\mu>0$, we have that $\mu(F\sigma_{d-1})^{\ast (2n-1)}\geq F\sigma_{d-1}\ast K$, for some nonnegative function $K\in C^0(\R^d)$ which satisfies $K(0)>0$.
By the Euler--Lagrange equation \eqref{eq:ELforF}, we have that $\mu(F\sigma_{d-1})^{\ast(2n-1)}\vert_{\sph{d-1}}=F$, where $\mu$ is defined by the identity
\[\mu^{-1}=(2\pi)^{-d}{\bf T}_{d,2n}^{2n}\|F\|_{L^2(\sph{d-1})}^{2n-2}.\]
But the inequality $F\geq F\sigma_{d-1}\ast K$ forces $F>0$ everywhere.
Indeed, if $\omega\in\sph{d-1}$ is such that $F(\omega)>0$, and $r,\varepsilon>0$ are such that $F(\omega')>\varepsilon$, for all $\omega'$ contained in the closed cap centered at $\omega$ and of radius $r$, denoted $\cp(\omega,r)$, then
\[ F(\nu)\geq \int_{\sph{d-1}}K(\nu-\omega')F(\omega')\d\sigma_{d-1}(\omega')\geq \eps\int_{\cp(\omega,r)}K(\nu-\omega')\d\sigma_{d-1}(\omega')>0, \]
for all $\nu\in\sph{d-1}$ satisfying $\ab{\nu-\omega}\leq \rho$, where $\rho>0$ depends only on $K$.
We conclude that there exists $\rho>0$ such that, if $F(\omega)>0$, then $F(\omega')>0$, for all $\omega'$ satisfying $\ab{\omega-\omega'}\leq \rho$. 
This immediately implies that $F>0$ on $\sph{d-1}$, and inequality \eqref{eq:boundedfrombelow} is then a consequence of the compactness of $\sph{d-1}$.
This completes the proof of the claim
in the special case when $n\geq 3$. 

If $n=2\geq\frac{d+1}{d-1}$, then $d\geq 3$.
Moreover, the above argument fails since the set $E=\cp-\cp$ does not contain a neighborhood of the origin as long as $\cp$ is contained in a hemisphere: for instance, if $\omega_0\in\sph{d-1}$ denotes the center of $\cp$, then $\lambda\omega_0\notin E$, for any $\lambda\in\R\setminus\{0\}$ (see Figure \ref{fig:c-c}). Therefore, a nonnegative function $K\in C^0(\R^d)$ satisfying $K(0)>0$ and $(F\sigma_{d-1})^{\ast 3}\geq F\sigma_{d-1}\ast K$ does not exist in general.
We circumvent this difficulty by following the line of reasoning in \cite[Lemma 3.2]{Ch11}. 
Define $\mathcal{O}=\{\omega\in\sph{d-1}\colon F(\omega)>0\}$, which is a non-empty, open subset of the unit sphere satisfying $\mathcal{O}=-\mathcal{O}$ (since $F$ is continuous, even, and does not vanish identically). 
Let $\cp\subseteq\sph{d-1}$ be as before, i.e. $\cp$ is a closed cap on which $F$ is bounded from below away from zero. 
As remarked above, the function $(F\sigma_{d-1})^{\ast 3}$ as defined in \eqref{eq:kconvF},
\begin{equation}\label{eq:tripleF}
(F\sigma_{d-1})^{\ast 3}(x)=\int_{\sph{d-1}}\bigl(F\sigma_{d-1}\ast F\sigma_{d-1}\bigr)(x-\omega)\,F(\omega)\d\sigma_{d-1}(\omega),
\end{equation}
is H\"older continuous on $\R^d$ of some parameter $\alpha=\alpha(d)>0$.
Moreover, the equality in the Euler--Lagrange equation \eqref{eq:simpleEL} satisfied by $F$,
\begin{equation}\label{eq:ELforF3}
\left(F\sigma_{d-1}\ast F\sigma_{d-1}\ast F\sigma_{d-1}\right) \big\vert_{\mathbb S^{d-1}}=(2\pi)^{-d}{\bf T}_{d,4}^4 \|F\|_{L^2(\sph{d-1})}^2 F,
\end{equation}
holds everywhere on $\sph{d-1}$.
It follows from \eqref{eq:tripleF} that $(F\sigma_{d-1})^{\ast 3}>0$ on $(\mathcal{O}+\mathcal{O}+\mathcal{O})\cap\sph{d-1}$, and from \eqref{eq:ELforF3} that $(\mathcal O+\mathcal O+\mathcal O)\cap\sph{d-1}\subseteq \mathcal O$. 
Since $\cp\subseteq\mathcal O=-\mathcal O$, we then have that
\begin{equation}\label{eq:containment}
(\cp-\cp+\omega_0)\cap\sph{d-1}\subseteq \mathcal O, \,\,\,\text{ for every } \omega_0\in\mathcal O.
\end{equation}
Note that, since $d\geq 3$, we have that\footnote{This fails if $d=2$. Indeed, if $\cp\subseteq\sph{1}$ is a cap, then $(\cp-\cp+\omega)\cap\sph{1}=\{\omega\}$ if $\omega\in\sph{1}\setminus(\cp\cup-\cp)$.} $\sigma_{d-1}(\cp-\cp+\omega)>0$, for every $\omega\in\sph{d-1}$.
We want to show that\footnote{Here, $\textup{cl}(X)$ and $\textup{int}(X)$ respectively denote the closure and the interior of a subset $X\subset\sph{d-1}$ with respect to the induced topology on the unit sphere $\sph{d-1}\subset\R^d$.} $\mathcal O=\textup{cl}(\mathcal O)$.
With this purpose in mind, let $\nu_0\in\textup{cl}(\mathcal O)$.
For all $\omega_0\in\mathcal O$ sufficiently close to $\nu_0$, we have that
\[\textup{int}[(\cp-\cp+\omega_0)\cap\sph{d-1}]\cap\textup{int}[(\cp-\cp+\nu_0)\cap\sph{d-1}]\neq\emptyset.\]
By \eqref{eq:containment}, it then follows that $\mathcal O\cap\textup{int}[(\cp-\cp+\nu_0)\cap\sph{d-1}]\neq\emptyset$, and so there exist $\widetilde{\omega}_0\in\mathcal O$ and $\zeta_0,\zeta'_0\in\cp$, such that $\widetilde{\omega}_0=\zeta_0-\zeta'_0+\nu_0$. 
Rearranging, we see that $\nu_0=\zeta'_0-\zeta_0+\widetilde{\omega}_0\in (\cp-\cp+\widetilde{\omega}_0)\cap\sph{d-1}\subseteq\mathcal O$, where the last containment is again a consequence of \eqref{eq:containment}.
This shows that $\nu_0\in\mathcal{O}$. 
Since $\nu_0\in \textup{cl}(\mathcal O)$ was arbitrary, it follows that $\mathcal O=\textup{cl}(\mathcal O)$. 
But $\mathcal O\subseteq\sph{d-1}$ is also open and non-empty, and $\sph{d-1}$ is connected, therefore $\mathcal O=\sph{d-1}$, as desired.

The proof of the theorem is now complete.
\end{proof}

\begin{figure}
\centering
\includegraphics[width=.5\linewidth]{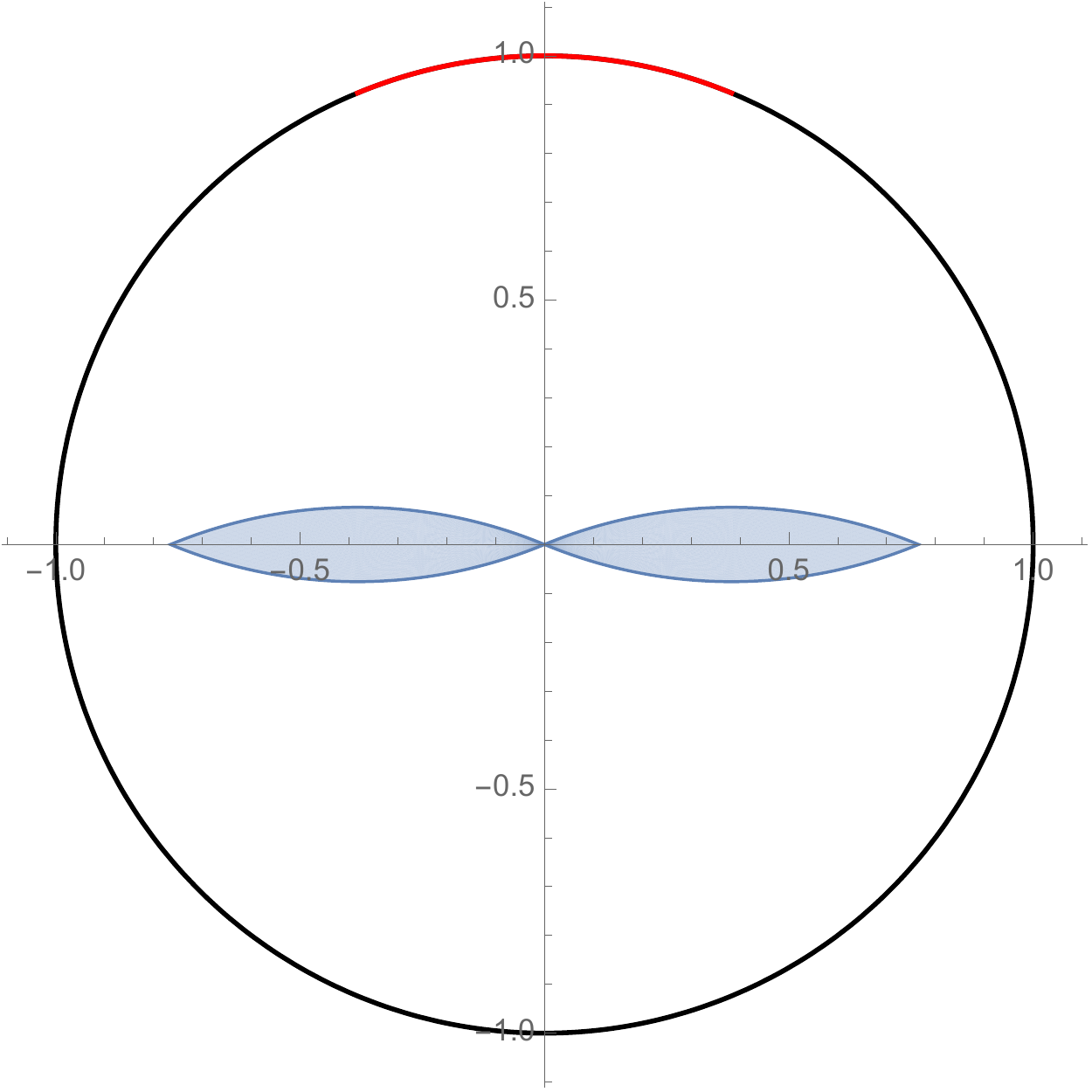}
\caption{The set $\cp-\cp$ (blue), where $\cp\subset\sph{1}$ is the (red) cap centered at the north pole with half-angle $\frac{\pi}8$.}
\label{fig:c-c}
\end{figure}

\section{Spherical convolutions}\label{sec:explicitConvolutions}

This section will focus on convolutions of the surface measure $\sigma_{d-1}$ on $\sph{d-1}$. 
One of our goals is to obtain, insofar as possible, explicit formulae for the $L^2$-norms $\norma{\sigma_{d-1}^{\ast n}}_{L^2}$, for integers $n\geq 2$.
These are related to the value $\Phi_{d,2n}(\one)$ via the identity 
\begin{equation}\label{eq:Phi1vsNorms}
 \Phi_{d,2n}(\one)
=(2\pi)^d\omega_{d-1}^{-n}\norma{\sigma_{d-1}^{\ast n}}_{L^2(\R^d)}^2, 
\end{equation}
where $\omega_{d-1}=\sigma_{d-1}(\sph{d-1})$  as usual.

Let us start with the simplest case $n=2$.
It is known, see \cite[Lemma 5]{COS15}, that the 2-fold convolution $\sigma_{d-1}^{\ast 2}$ defines a measure supported on the ball $B_2\subset\R^d$, absolutely continuous with respect to Lebesgue measure on $B_2$, whose Radon--Nikodym derivative equals
\begin{equation}\label{eq:2fold}
\sigma_{d-1}^{\ast 2}(\xi) 
=\int_{(\sph{d-1})^2} \ddirac{\xi-\omega-\nu}\,\textup{d}\sigma_{d-1}(\omega)\,\textup{d}\sigma_{d-1}(\nu)
=\frac{\omega_{d-2}}{2^{d-3}}\frac{1}{|\xi|} (4-|\xi|^2)_+^{\frac{d-3}2}.
\end{equation}
Here $x_+:=\max\{0,x\}$ for $x\in\R$, and $(4-\ab{\xi}^2)_+^{\frac{d-3}{2}}:=\bigl((4-\ab{\xi}^2)_+\bigr)^{\frac{d-3}{2}}$.
The corresponding $L^2$-norms can be easily computed in terms of the Gamma function, as the next result shows.

\begin{lemma}\label{lem:2fold}
	Let $d\geq 3$. Then:
	\begin{equation}\label{eq:2foldL2}
	\|\sigma_{d-1}^{\ast 2}\|_{L^2(\R^d)}^2
	=2^d\pi^{\frac{3d}2-1}\frac{\Gamma(\tfrac d2-1)\Gamma(d-2)}{\Gamma(\frac{d-1}2)^2\Gamma(\frac d2)\Gamma(\tfrac {3d}2-3)}.
	\end{equation}
\end{lemma}
\noindent Note that the right-hand side of \eqref{eq:2foldL2} diverges to $\infty$ as $d\to2^+$, reflecting the well-known fact that $\widehat{\sigma}_1\notin L^4(\R^2)$.

\begin{proof}
	Let $c_d=2^{-d+3}\omega_{d-2}$. 
	Identity \eqref{eq:2fold} and a computation in polar coordinates  yield:
	\begin{align*}
	\|\sigma_{d-1}^{\ast 2}\|_{L^2(\R^d)}^2
	&=\int_{B_2} \biggl(\frac{c_d}{|\xi|}\biggr)^2 (4-|\xi|^2)^{d-3}\,\textup{d}\xi
	=c_d^2\omega_{d-1}\int_0^2 (4-r^2)^{d-3} r^{d-3}\,\textup{d}r\\
	&=c_d^2\omega_{d-1}\int_0^1 (4-4t)^{d-3} (4t)^{\frac d2-2}\,2\,\textup{d}t
	=c_d^2\omega_{d-1} 8^{d-3}B(\tfrac d2-1,d-2)\\
	&=c_d^2\omega_{d-1} 8^{d-3}\frac{\Gamma(\tfrac d2-1)\Gamma(d-2)}{\Gamma(\tfrac {3d}2-3)}.
	\end{align*}
	From the first to the second line, we changed variables $r^2=4t$, and then expressed the Beta function in terms of the Gamma function,
	$B(a,b)=\frac{\Gamma(a)\Gamma(b)}{\Gamma(a+b)}$.
	The result then follows from invoking  \eqref{eq:omegad}.
\end{proof}

We now address the case $n=3$. 
The following result is the first instance in which a distinction between even and odd dimensions arises.

\begin{lemma}\label{lem:tripleConvos}
	Let $d\geq 2$. Then the following integral expression holds for the 3-fold convolution:
	\begin{equation}\label{eq:formulaTripleConvos}
	\begin{split}
	\sigma_{d-1}^{\ast 3}(\xi)=&\frac{\omega_{d-2}^2}{2^{2(d-3)}\ab{\xi}^{d-2}}\\
	\quad&\times
	\begin{cases}
	\int_{1-\ab{\xi}}^{1+\ab{\xi}}  (4-t^2)^{\frac{d-3}{2}}\Bigl(4\ab{\xi}^2-(1+\ab{\xi}^2-t^2)^2\Bigr)^{\frac{d-3}{2}}\d t, \,\text{ if }\ab{\xi}\leq 1,\\
	\int_{\ab{\xi}-1}^{2}  (4-t^2)^{\frac{d-3}{2}}\Bigl(4\ab{\xi}^2-(1+\ab{\xi}^2-t^2)^2\Bigr)^{\frac{d-3}{2}}\d t,\,\text{ if } 1\leq\ab{\xi}\leq 3.
	\end{cases}
	\end{split}
	\end{equation}
	In particular, if $d\geq 3$ is an odd integer, then there exist an even polynomial $P_{0,3}$ and polynomials $P_{1,3},\,Q_{1,3}$, all with rational coefficients, such that
	\[ \sigma_{d-1}^{\ast 3}(\xi)=\frac{\omega_{d-2}^2}{2^{2(d-3)}}\times
	\begin{cases}
	P_{0,3}(\ab{\xi}), &\text{ if }\ab{\xi}\leq 1,\\
	P_{1,3}(\ab{\xi})+Q_{1,3}(\ab{\xi}^{-1})
	,&\text{ if } 1\leq\ab{\xi}\leq 3.
	\end{cases} \]
\end{lemma}
\begin{proof}
	Let $d\geq 2$, and consider $\xi\neq 0$. Using spherical coordinates with axis parallel to $\xi$ together with identity \eqref{eq:2fold}, we compute
	\begin{align*}
	\sigma_{d-1}^{\ast 3}(\xi)&=\int_{\mathbb S^{d-1}}\sigma_{d-1}^{\ast2}(\xi-\eta)\d\sigma_{d-1}(\eta)\\
	&=\frac{\omega_{d-2}}{2^{d-3}}\int_{\mathbb S^{d-1}}\frac{(3-\ab{\xi}^2+2\xi\cdot\eta)^{\frac{d-3}{2}}}{(1+\ab{\xi}^2-2\xi\cdot\eta)^{\frac12}}\mathbbm{1}((1+\ab{\xi}^2-2\xi\cdot\eta)^{\frac12}\leq 2)\d\sigma_{d-1}(\eta)\\
	&=\frac{\omega_{d-2}^2}{2^{d-3}}\int_0^\pi \frac{(3-\ab{\xi}^2+2\ab{\xi}\cos\phi)^{\frac{d-3}{2}}}{(1+\ab{\xi}^2-2\ab{\xi}\cos\phi)^{\frac12}}\mathbbm{1}((1+\ab{\xi}^2-2\ab{\xi}\cos\phi)^{\frac12}\leq 2)\sin^{d-2}\phi\d\phi\\
	&=\frac{\omega_{d-2}^2}{2^{d-3}}\int_{-1}^{1} \frac{(3-\ab{\xi}^2+2\ab{\xi}u)^{\frac{d-3}{2}}}{(1+\ab{\xi}^2-2\ab{\xi}u)^{\frac12}}\mathbbm{1}((1+\ab{\xi}^2-2\ab{\xi}u)^{\frac12}\leq 2)(1-u^2)^{\frac{d-3}{2}}\d u,
	\end{align*}
	where in the last line we changed variables $\cos\phi=u$. Another change of variables 
	yields
	\begin{align*}
	\sigma_{d-1}^{\ast 3}(\xi)
	=\frac{\omega_{d-2}^2}{2^{2(d-3)}\ab{\xi}^{d-2}}\int_{\ab{1-\ab{\xi}}}^{1+\ab{\xi}}  (4-t^2)^{\frac{d-3}{2}}\Bigl(4\ab{\xi}^2-(1+\ab{\xi}^2-t^2)^2\Bigr)^{\frac{d-3}{2}}\mathbbm{1}(t\leq 2)\d t.
	\end{align*}
	For $\ab{\xi}\leq 1$, we obtain
	\begin{equation}\label{eq:ConvXiless1}
	\sigma_{d-1}^{\ast 3}(\xi)=\frac{\omega_{d-2}^2}{2^{2(d-3)}\ab{\xi}^{d-2}}\int_{1-\ab{\xi}}^{1+\ab{\xi}}  (4-t^2)^{\frac{d-3}{2}}\Bigl(4\ab{\xi}^2-(1+\ab{\xi}^2-t^2)^2\Bigr)^{\frac{d-3}{2}}\d t,
	\end{equation}
	whereas, for $1\leq \ab{\xi}\leq 3$,
	\begin{equation}\label{eq:ConvXionetothree}
	\sigma_{d-1}^{\ast 3}(\xi)=\frac{\omega_{d-2}^2}{2^{2(d-3)}\ab{\xi}^{d-2}}\int_{\ab{\xi}-1}^{2}  (4-t^2)^{\frac{d-3}{2}}\Bigl(4\ab{\xi}^2-(1+\ab{\xi}^2-t^2)^2\Bigr)^{\frac{d-3}{2}}\d t.
	\end{equation}
	This concludes the verification of \eqref{eq:formulaTripleConvos}.
	
	Now let $d\geq 3$ be an odd integer, so that  $\frac{d-3}{2}$ is an integer as well. In this case, it is clear that both integrals \eqref{eq:ConvXiless1} and \eqref{eq:ConvXionetothree} can be computed explicitly, yielding a polynomial with rational coefficients in the variable $\ab{\xi}$. Moreover, the right-hand side of \eqref{eq:ConvXionetothree} admits a representation of the form
	\[ \sigma_{d-1}^{\ast 3}(\xi)=\frac{\omega_{d-2}^2}{2^{2(d-3)}}(P_{1,3}(\ab{\xi})+Q_{1,3}(\ab{\xi}^{-1})), \quad 1\leq|\xi|\leq 3,\]
	for some polynomials $P_{1,3},\,Q_{1,3}\in \Qr[x]$ with rational coefficients, the degree of $Q_{1,3}$ being at most $d-2$. 
	
	On the other hand, if $\ab{\xi}\leq 1$, a change of variables in \eqref{eq:ConvXiless1} via $\ab{\xi}^{-1}(t-1)=s$, $\d t=\ab{\xi}\d s$, yields
	\begin{align*}
	\sigma_{d-1}^{\ast 3}(\xi)=\frac{\omega_{d-2}^2}{2^{2(d-3)}}\int_{-1}^{1}  (4-(1+\ab{\xi}s)^2)^{\frac{d-3}{2}}\Bigl(4-(\ab{\xi}(1-s^2)-2s)^2\Bigr)^{\frac{d-3}{2}}\d s,
	\end{align*}
	for every $|\xi|\leq 1$.
In particular, $\sigma_{d-1}^{\ast 3}(\xi)$ is then seen to be  a polynomial in the variable $\ab{\xi}$. 
Moreover, the substitution $s\rightsquigarrow -s$ reveals that the function
 \[r\mapsto \int_{-1}^{1}  (4-(1+rs)^2)^{\frac{d-3}{2}}\Bigl(4-(r(1-s^2)-2s)^2\Bigr)^{\frac{d-3}{2}}\d s\]
	defines an even polynomial in $r$. We conclude that 
	\[\sigma_{d-1}^{\ast 3}(\xi)=\omega_{d-2}^2 2^{-2(d-3)}P_{0,3}(\ab{\xi}),\quad |\xi|\leq1,\] 
	for some even polynomial $P_{0,3}\in \Qr[x]$ of degree $2(d-3)$.
	This completes the proof of the lemma.
	\end{proof}

Lemmata \ref{lem:2fold} and  \ref{lem:tripleConvos} pave the way towards the following exact calculations.

\begin{lemma}\label{lem:explicitL2norms}
	The following identities hold:
	\begin{align}
	\norma{\sigma_2^{\ast 2}}_{L^2(\R^3)}^2&=32\pi^3,& \norma{\sigma_2^{\ast 3}}_{L^2(\R^3)}^2&=256\pi^5,\label{eq:d=3}\\
	\norma{\sigma_4^{\ast 2}}_{L^2(\R^5)}^2&=\frac{2048\pi^6}{315},& \norma{\sigma_4^{\ast 3}}_{L^2(\R^5)}^2&=\frac{12402688\pi^{10}}{675675},\label{eq:d=5}\\
	\norma{\sigma_6^{\ast 2}}_{L^2(\R^7)}^2&=\frac{65536\pi^9}{225225},& \norma{\sigma_6^{\ast 3}}_{L^2(\R^7)}^2&=\frac{1718574645248\pi^{15}}{18446991688125},\label{eq:d=7}\\
	\norma{\sigma_8^{\ast 2}}_{L^2(\R^9)}^2&=\frac{4194304\pi^{12}}{916620705},& \norma{\sigma_8^{\ast 3}}_{L^2(\R^9)}^2&=\frac{1240907332251025408\pi^{20}}{14790982571478431443125}.\label{eq:d=9}
	\end{align}
\end{lemma}
\begin{proof}
	The $L^2$-norm of the 2-fold convolution $\sigma_{d-1}^{\ast 2}$ was already computed in Lemma \ref{lem:2fold}.
	The values on the left-hand side of \eqref{eq:d=3}--\eqref{eq:d=9} follow from specializing identity \eqref{eq:2foldL2} to the cases $d\in\{3,5,7,9\}$.
	
	For the 3-fold convolution $\sigma_{d-1}^{\ast 3}$, we consider the integral expression \eqref{eq:formulaTripleConvos} in each dimension separately. All formulae are easily programmable in a computer algebra system, which calculates all expressions {\it exactly} whenever $d$ is an odd integer. We proceed to list the explicit formulae for the 3-fold convolutions which were obtained in this way.\\
	
	\noindent\textbf{Case $d=3$.} Here $\frac{d-3}{2}=0$ and $\omega_{d-2}=\omega_1=2\pi$, so that computation of the corresponding integrals in  \eqref{eq:formulaTripleConvos} yields
	\begin{equation}\label{eq:tripledequal3}
	\sigma_2^{\ast 3}(\xi)=\begin{cases}
	8\pi^2,&\text{ if }\ab{\xi}\leq 1,\\
	\ds 4\pi^2\Bigl(-1+\frac{3}{\ab{\xi}}\Bigr),&\text{ if }1\leq \ab{\xi}\leq 3.
	\end{cases}
	\end{equation}
	
	\noindent\textbf{Case $d=5$.} Here $\frac{d-3}{2}=1$ and $\omega_{d-2}=\omega_3=\frac{2\pi^2}{\Gamma(2)}=2\pi^2$, so that computation of the corresponding integrals in  \eqref{eq:formulaTripleConvos} yields
	\begin{equation}\label{eq:tripledequal5}
	\sigma_4^{\ast 3}(\xi)=\begin{cases}
	\ds\frac{\pi^4}{105}(4\ab{\xi}^4-168\ab{\xi}^2+420),&\text{ if }\ab{\xi}\leq 1,\\
	\begin{aligned}
	&-\frac{\pi^4}{105\ab{\xi}^3}
	(2\ab{\xi}^7-84\ab{\xi}^5+210\ab{\xi}^4\\
	&\qquad+210\ab{\xi}^3-756\ab{\xi}^2+162)
	\end{aligned},&\text{ if }1\leq \ab{\xi}\leq 3.
	\end{cases}
	\end{equation}
	
	\noindent\textbf{Case $d=7$.} Here $\frac{d-3}{2}=2$ and $\omega_{d-2}=\omega_5=\frac{2\pi^{3}}{\Gamma(3)}=\pi^3$, so that  computation of the corresponding integrals in  \eqref{eq:formulaTripleConvos} yields
	\begin{equation}\label{eq:tripledequal7}
	\sigma_6^{\ast 3}(\xi)=\begin{cases}
	\displaystyle\frac{\pi^6}{15015}(\ab{\xi}^8-52\ab{\xi}^6+1430\ab{\xi}^4-6292\ab{\xi}^2+9009),&\text{ if }\ab{\xi}\leq 1,\\
	\begin{aligned}
	&-\frac{\pi^6}{30030\ab{\xi}^5}
	(\ab{\xi}^{13}-52\ab{\xi}^{11}+1430\ab{\xi}^9-3003\ab{\xi}^8\\
	&\qquad-6292\ab{\xi}^7+18876\ab{\xi}^6
	+9009\ab{\xi}^5-38610\ab{\xi}^4\\
	&\qquad+12636\ab{\xi}^2-2187)
	\end{aligned},&\text{ if }1\leq \ab{\xi}\leq 3.
	\end{cases}
	\end{equation}
	\noindent\textbf{Case $d=9$.} Here $\frac{d-3}{2}=3$ and $\omega_{d-2}=\omega_7=\frac{2\pi^{4}}{\Gamma(4)}=\frac{\pi^4}{3}$, so that  computation of the corresponding integrals in  \eqref{eq:formulaTripleConvos} yields 
	\begin{equation}\label{eq:tripledequal9}
	\sigma_8^{\ast 3}(\xi)=\begin{cases}
	\begin{aligned}
	&\frac{\pi^8}{87297210}(5\ab{\xi}^{12}-342\ab{\xi}^{10}+10659\ab{\xi}^8-236436\ab{\xi}^6\\
	&\quad+1398267\ab{\xi}^4-3602742\ab{\xi}^2+3741309)
	\end{aligned},&\text{ if }\ab{\xi}\leq 1,\\
	\begin{aligned}
	&-\frac{\pi^8}{174594420\ab{\xi}^7}
	(5\ab{\xi}^{19}
	-342\ab{\xi}^{17}+10659\ab{\xi}^{15}\\
	&\qquad-236436\ab{\xi}^{13}+415701\ab{\xi}^{12}+1398267\ab{\xi}^{11}\\
	&\qquad -3602742\ab{\xi}^{10}-3602742\ab{\xi}^9+12584403\ab{\xi}^8\\
	&\qquad+3741309\ab{\xi}^7-19151316\ab{\xi}^6+7770411\ab{\xi}^4\\
	&\qquad-2243862\ab{\xi}^2+295245)
	\end{aligned},&\text{ if }1\leq \ab{\xi}\leq 3.
	\end{cases}
	\end{equation}	
	The corresponding $L^2$-norms can be computed directly via the use of polar coordinates, as exemplified in the course of the proof Lemma \ref{lem:2fold} for the case of 2-fold convolutions.
	This concludes the proof of the lemma.
\end{proof}

\begin{remark} 
Formula \eqref{eq:tripledequal3} and \cite[Prop.\@ 3.1]{OSQ19} together imply that the convolution $\sigma_{d-1}^{\ast m}$ defines a continuous function on $\R^d$, provided  $d=2$ and $m\geq 5$, or $d\geq 3$ and $m\geq 3$. 
 As a consequence, if $d=2$ and $m\geq 3$, or $d\geq 3$ and $m\geq 2$, then\footnote{By a slight abuse of notation, we are using the fact that the spherical convolution defines a radial function on $\R^d$ and, given $r\geq 0$, denote by $\sigma^{\ast k}(r)$ the value attained on the sphere $|\xi|=r$.}
\begin{equation}\label{eq:equivalenceL2norm}
\norma{\sigma_{d-1}^{\ast m}}_{L^2(\R^d)}^2=\sigma_{d-1}^{\ast (2m)}(0)=\omega_{d-1}\sigma_{d-1}^{\ast (2m-1)}(1).
\end{equation}
To verify the first identity in \eqref{eq:equivalenceL2norm}, write $\sigma_{d-1}^{\ast (2m)}(0)=(\sigma_{d-1}^{\ast m}\ast\sigma_{d-1}^{\ast m})(0)$, and then use the integral expression for the convolution. The second identity is obtained by writing $\sigma_{d-1}^{\ast (2m)}(0)=(\sigma_{d-1}^{\ast (2m-1)}\ast \sigma_{d-1})(0)$, and then appealing to the radiality of $\sigma_{d-1}^{\ast (2m-1)}$.
\end{remark} 

By following similar steps to those in the proof of Lemma \ref{lem:tripleConvos}, it is possible to obtain integral formulae for the higher order $n$-fold convolutions $\sigma_{d-1}^{\ast n}$ in terms of $\sigma_{d-1}^{\ast(n-1)}$, thereby establishing a  recurrence relation. If $d\geq 3$ is an odd integer, 
then the recurrence can be resolved up to any given $N$ by integrating one value of $n\in\{3,4,\dots,N\}$ at a time. We leave the details to the interested reader since such recurrence relations will not be needed in our main argument.
Indeed, explicit expressions for $\sigma_{d-1}^{\ast n}$ already exist in the literature, and we proceed to describe them.

The following theorem (which has been translated into our notation) appears in recent work of Garc\'ia-Pelayo \cite[\S V, \S VI]{GP12}, in connection to the study of uniform random walks in Euclidean space.
We further comment on the link between convolution measures on the sphere and uniform random walks in \S \ref{sec:randomwalks} below.

\begin{theorem}[\cite{GP12}]\label{thm:formulaConvoGP}
	Let $d,n$ be integers, with $d\geq 3$ odd, and $n\geq 2$. Then,\footnote{We report a typographical error in \cite[Formula (32)]{GP12}, where the power $2^{(d-1)/2}$ in the denominator of the factor on the right-hand side should be replaced by $2$; see Appendix \ref{sec:AppendixGP} for details.}
	for every $r\geq 0$,
	\begin{equation}\label{eq:GPcompactFormula}
	\sigma_{d-1}^{\ast n}(r)=\frac{2^n\pi^{\frac{n(d-1)}{2}}}{\Gamma(\frac{d-1}{2})^n}\Bigl(-\frac{1}{2\pi r}\frac{\d}{\d r}\Bigr)^{\frac{d-1}{2}}\Lambda_d^{\ast n}(r),
	\end{equation}
	where $\Lambda_d(s):=(1-s^2)_+^{\frac{d-3}{2}}$.
	In particular, if $n\geq 2$, then, 	for every $r\geq 0$,
	\begin{equation}\label{eq:GPtripleConvo}
	\sigma_2^{\ast n}(r)=-\frac{(4\pi)^n}{8\pi r(n-2)!}\sum_{j=0}^{\lfloor\frac{r+n}{2}\rfloor}\binom{n}{j}(-1)^j\Bigl(\frac{r}{2}+\frac{n}{2}-j\Bigr)^{n-2}.
	\end{equation}
	Here $\lfloor\cdot\rfloor$ denotes the floor function.
\end{theorem}

	\noindent In formula \eqref{eq:GPcompactFormula}, the function $\Lambda_d$ is  supported on the interval $[-1,1]$. The convolution on the right-hand side of \eqref{eq:GPcompactFormula} is understood in the usual sense of convolutions of functions on the real line, whereas that on the left-hand side of \eqref{eq:GPcompactFormula} denotes the $n$-fold convolution of $\sigma_{d-1}$ in $\R^d$, as defined in \eqref{eq:fsigma} above. 
	 In Appendix A, we revisit the argument of \cite{GP12}, and summarize the proof of \eqref{eq:GPcompactFormula}.
	If $d=3$, then 
	\[(1-r^2)_+^{\frac{d-3}{2}}=\mathbbm{1}_{[-1,1]}(r),\] 
	in which case the explicit evaluation of the $n$-fold convolution $(\mathbbm{1}_{[-1,1]})^{\ast n}$ is more accessible, whereas the differential operator acting on it reduces to $-\frac{1}{2\pi r}\frac{\d}{\d r}$. These two observations eventually lead to identity \eqref{eq:GPtripleConvo}.

Borwein \& Sinnamon \cite{BS}, taking \eqref{eq:GPcompactFormula} as a starting point, derived a general formula for $\sigma_{d-1}^{\ast n}$ similar to \eqref{eq:GPtripleConvo}, for all odd integers $d\geq 3$. Prior to stating the relevant result from \cite{BS}, we introduce some notation. Let $H(x)$ denote the Heaviside step function,
\[ H(x)=\begin{cases}
0,& \text{ if }x<0,\\
\frac{1}{2},& \text{ if }x=0,\\
1,& \text{ if }x>0.
\end{cases} \]
If $Q$ is a polynomial in the variable $r$, then $[r^j]Q$ will denote the coefficient of the monomial $r^j$ in $Q$. For $m\geq 1$, we further set
\[ C_m(r):=\sum_{k=0}^{m-1}\frac{(m-1+k)!}{2^k\,k!(m-1-k)!}r^k. \]
Translated into our notation,\footnote{See Remark \ref{rem:BS} in Appendix \ref{sec:AppendixGP}.}  \cite[Cor.\@ 6]{BS} states the following. 
\begin{theorem}[\cite{BS}]\label{thm:formulaConvoBS}
	Let $d,n$ be integers, with $d\geq 3$ odd, and $n\geq 2$. Set $m=\frac{d-1}2$. Then, for every $r\geq 0$,
	\begin{multline}\label{eq:BSnthConvo}
	\sigma_{d-1}^{\ast n}(r)=\frac{1}{r^{d-1}}\biggl(\frac{2\pi^{m+\frac{1}{2}}}{\Gamma(m+\frac{1}{2})}\biggr)^{n-1}\biggl(\frac{\Gamma(2m)}{2^m\Gamma(m)}\biggr)^n\;\sum_{\ell=0}^{\lfloor\frac{n-1}{2}\rfloor}\binom{n}{\ell}(-1)^{m\ell}H(n-2\ell-r)\\
	\times\sum_{k=1}^{m}2^k\binom{m-1}{k-1}\frac{(2m-1-k)!}{(2m-1)!}r^k
	\sum_{j=0}^{(m-1)n}\frac{(n-2\ell-r)^{mn-1+j-k}}{(mn-1+j-k)!}[r^j](C_m(r)^\ell C_m(-r)^{n-\ell}).
	\end{multline}
\end{theorem}
\noindent Theorem \ref{thm:formulaConvoBS} will play a key role in the proof of Proposition \ref{lem:verifiedIneq} below.

\begin{remark} 
At first sight, formula \eqref{eq:BSnthConvo} may look complicated.
However, it is only the closed-form expression of a piecewise polynomial in the variable $r$, divided by the normalizing power $r^{d-1}$. It defines a continuous function on the half-line $r\in(0,\infty)$, except when $(d,n)=(3,2)$, in which case it defines a  continuous function on the interval $r\in(0,2)$. 
To see why this is indeed the case, note that the term 
\[H(n-2\ell-r)(n-2\ell-r)^{mn-1+j-k}\] 
defines a continuous function of $r$ on the whole real line, for any $\ell\in\{0,\dots,\lfloor\frac{n-1}{2}\rfloor\}$, $k\in\{1,\dots,m\}$ and $j\in\{0,\dots,(m-1)n\}$, since $mn-1+j-k\geq m(n-1)-1>0$ for $m\geq 1$ and $n\geq 3$, and for $m\geq 2$ and $n\geq 2$.
On the other hand, if $(m,n)=(1,2)$, then \eqref{eq:BSnthConvo} reduces to the case $d=3$ of \eqref{eq:2fold}. 
Even though the continuity at $r=0$ for $d\geq 3,n\geq 3$ may not seem evident from formula \eqref{eq:BSnthConvo},  it follows from \cite[Prop.\@ 3.1 and Remark 3.2]{OSQ19}; while the case $(d,n)=(3,3)$ is not covered by these results,  it follows easily from the fact that \eqref{eq:tripledequal3} defines a continuous function of $\xi$. 
\end{remark}

\subsection{Connection with the theory of uniform random walks}\label{sec:randomwalks}
Consider $n$ independent, identically distributed random variables $\{X_j\}_{j=1}^n$ taking values on the unit sphere $\sph{d-1}$ with uniform distribution. In other words, for any Borel subset $\Sigma\subseteq\sph{d-1}$, we require that
\[ \P(X_j\in\Sigma)=\omega_{d-1}^{-1}\sigma_{d-1}(\Sigma). \]
In this case, the random variable $Y_n=\sum_{j=1}^n X_j$ corresponds to the so-called {\it uniform $n$-step random walk} in $\R^d$, and is distributed according to the $n$-fold convolution of the normalized surface measure on the sphere,  $\bar\sigma_{d-1}:=\omega_{d-1}^{-1}\sigma_{d-1}$.
In other words, for any Borel subset $\Omega\subseteq \R^d$,
\[ \P(Y_n\in\Omega)=\int_{\Omega}\bar\sigma_{d-1}^{\ast n}(\xi)\d \xi.\]
Following Borwein et al.\@ \cite{BS,BSV16,BSWZ12}, we set $m=(d-1)/2$, and let $p_n(m-1/2;r)$ denote the probability density associated to the random variable $\ab{Y_n}$.
For any Lebesgue-measurable subset $E\subseteq (0,\infty)$, we thus have that
\[ \P(\ab{Y_n}\in E)=\int_E p_n(m-1/2;r)\d r. \]
A direct computation in polar coordinates establishes the following identities:
\begin{equation}\label{eq:relationSigmaDensity}
\bar\sigma_{d-1}^{\ast n}(r)=\omega_{d-1}^{-1}\frac{p_n(m-1/2;r)}{r^{d-1}},\quad \sigma_{d-1}^{\ast n}(r)=\omega_{d-1}^{n-1}\frac{p_n(m-1/2;r)}{r^{d-1}}. 
\end{equation}
Since random walks have been the object of intense investigation for more than a century, it is surprising that the explicit formula \eqref{eq:BSnthConvo} for the odd dimensional case of $\bar\sigma_{d-1}^{\ast n}$, and therefore for $p_n(m-1/2;r)$, has been obtained only very recently.
As far as we can tell, the even dimensional case remains a fascinating and largely open problem. 
Some interesting identities for the planar case $d=2$ appear in \cite{BSWZ12}, including modular and hypergeometric representations of the densities $p_3(0;r)$ and $p_4(0;r)$ for 3- and 4-step random walks, respectively, as well as some asymptotic expansions near $r=0$; see also \cite{Zh17a, Zh17b}.
For instance, \cite[Ex.\@ 4.3]{BSWZ12} reveals that 
\begin{equation}\label{eq:sigma1zeropre}
p_4(0;r)= -\frac{3}{2\pi^2}r\log r+\frac{9}{2\pi^2}\log(2)r+O(r^3), \text{ as } r\to 0^+.
 \end{equation}
 In view of \eqref{eq:relationSigmaDensity}, it follows from \eqref{eq:sigma1zeropre}  that the 4-fold convolution measure $\sigma_1^{\ast 4}$ has a logarithmic singularity at the origin, quantified as follows:
\begin{equation}\label{eq:sigma1zero}
 \sigma_1^{\ast 4}(r)= -12\pi\log r+36\pi\log(2)+O(r^2),\text{ as }r\to 0^+. 
 \end{equation}
 
\section{Tools from functional analysis}\label{sec:FA}

In this section, we rephrase
the Euler--Lagrange equation \eqref{eq:ELnonconv} as an eigenvalue problem for a certain operator, whose functional analytic properties will be of importance to the forthcoming analysis.
 
Given a function $f\in L^2(\sph{d-1})$,
define the integral operator $T_f:L^2(\sph{d-1})\to L^2(\sph{d-1})$, 
\begin{equation}\label{eq:DefTf}
T_f(g)(\omega)= (g \ast K_f)(\omega)= \int_{\sph{d-1}} g(\nu) K_f(\omega-\nu) \,\textup{d}\sigma_{d-1}(\nu),
\end{equation}
acting on functions $g\in L^2(\sph{d-1})$ by convolution with the  kernel $K_f$,
\begin{equation}\label{eq:defKf}
K_f(\xi)=(|\widehat{f\sigma}_{d-1}|^{q-2})^\vee(\xi)=\int_{\R^d} e^{iz\cdot \xi} |\widehat{f\sigma}_{d-1}(z)|^{q-2}\,\d z.
\end{equation}
If $q\geq q_d$, then the stated mapping property, $T_f:L^2(\sph{d-1})\to L^2(\sph{d-1})$, follows from Fourier inversion, H\"older's inequality, and \eqref{eq:TS}, since together they imply 
\begin{align}
|\langle T_f(g),h\rangle_{L^2(\sph{d-1})}|\notag
&=\left|\int_{\sph{d-1}} (g\ast K_f)(\omega)\overline{h(\omega)}\,\d\sigma_{d-1}(\omega)\right|\notag\\
&=\left|\int_{\R^d} |\widehat{f\sigma}_{d-1}(x)|^{q-2} \widehat{g\sigma}_{d-1}(x)\overline{\widehat{h\sigma}_{d-1}(x)}\,\d x\right|\label{eq:ToBeSpecialized}\\
&\leq \|\widehat{f\sigma}_{d-1}\|_{L^q(\R^d)}^{q-2} \|\widehat{g\sigma}_{d-1}\|_{L^q(\R^d)}\|\widehat{h\sigma}_{d-1}\|_{L^q(\R^d)}\notag\\
&\lesssim \|f\|_{L^2(\sph{d-1})}^{q-2} \|g\|_{L^2(\sph{d-1})} \|h\|_{L^2(\sph{d-1})}.\notag
\end{align}
Equation \eqref{eq:ELnonconv} is seen to be equivalent to the eigenvalue problem
\begin{equation}\label{eq:eigvalprobl}
T_f(f)=\lambda \|f\|_{L^2(\sph{d-1})}^{q-2}f,\quad \sigma_{d-1}\text{-a.e. on }\sph{d-1}.
\end{equation}
The following results capture some of the properties enjoyed by the kernel $K_f$ and the associated operator $T_f$.

\begin{lemma}\label{lem:Tf}
Let $d\geq 2$ and $q\geq 2\frac{d+1}{d-1}$.
Let ${\bf 0}\neq f\in L^2(\sph{d-1})$.
Then the kernel $K_f$ in \eqref{eq:defKf} satisfies
\begin{equation}\label{eq:Kfconj}
K_f(-\xi)=\overline{K_f(\xi)}, \text{ for every } \xi\in\R^d.
\end{equation}
Moreover, the associated operator $T_f$ in \eqref{eq:DefTf}
is self-adjoint and positive definite.
\end{lemma}

\begin{proof}
Identity \eqref{eq:Kfconj} follows at once from the fact that $|\widehat{f\sigma}_{d-1}|^{q-2}$ is real-valued.
Self-adjointness of $T_f$ can then be checked directly:
\begin{align*}
\langle T_f(g),h\rangle_{L^2(\sph{d-1})}
&=\int_{\sph{d-1}} (g\ast K_f)(\omega)\overline{h(\omega)}\,\d \sigma_{d-1}(\omega)\\
&=\int_{\sph{d-1}} \int_{\sph{d-1}} g(\nu) K_f(\omega-\nu) \,\textup{d}\sigma_{d-1}(\nu)\overline{h(\omega)}\,\d \sigma_{d-1}(\omega)\\
&=\int_{\sph{d-1}}g(\nu) \int_{\sph{d-1}} \overline{h(\omega)} \, \overline{K_f(\nu-\omega)}\,\d \sigma_{d-1}(\omega) \,\textup{d}\sigma_{d-1}(\nu)\\
&=\langle g,T_f(h)\rangle_{L^2(\sph{d-1})}.
\end{align*}
Positive definiteness of $T_f$ is also straightforward to verify. By specializing \eqref{eq:ToBeSpecialized} to $h=g$, we have that 
$$\langle T_f(g),g\rangle_{L^2(\sph{d-1})}
=\int_{\R^d} |\widehat{f\sigma}_{d-1}(x)|^{q-2} |\widehat{g\sigma}_{d-1}(x)|^2 \,\d x
\geq 0.$$
The latter identity can be used to  show that $\langle T_f(g),g\rangle_{L^2} = 0$ if and only if $g={\bf 0}$, for the functions  $\widehat{f\sigma}_{d-1}$ and $\widehat{g\sigma}_{d-1}$ are the Fourier transform of compactly supported  measures, and  therefore real-analytic.\footnote{See the discussion in \cite[\S 6.1]{BOSQ18}, and in particular the  proof of \cite[Prop.\@ 6.7]{BOSQ18} together with the comment immediately following it.}
Indeed, assume $\langle T_f(g),g\rangle_{L^2} = 0$.
Since $\widehat{f\sigma}_{d-1}$ is not identically zero, there exist $x\in\R^d$ and $r>0$, such that 
\[\ab{\widehat{f\sigma}_{d-1}}>0 \text{ and } \widehat{g\sigma}_{d-1}\equiv 0 \text{ on } x+B_r.\] 
By the Identity Theorem for real-analytic functions on $\R^d$, 
the last condition forces $\widehat{g\sigma}_{d-1}\equiv 0$ in $\R^d$, and consequently  $g=\mathbf{0}$.
\end{proof}

\begin{lemma}\label{lem:Kf}
Let $d\geq 2$ and
$q\geq \frac{4d}{d-1}$.
Let
${\bf 0}\neq f\in L^2(\sph{d-1})$.
Then the  kernel $K_f$ in \eqref{eq:defKf}  defines a continuous, bounded function on $\R^d$.
In particular, the associated operator $T_f$ in \eqref{eq:DefTf} is Hilbert--Schmidt.
\end{lemma}

\begin{proof}
Start by noting that 
$$\int_{\R^d}|\widehat{f\sigma}_{d-1}(x)|^{q-2} \d x=\|\widehat{f\sigma}_{d-1}\|_{L^{q-2}(\R^d)}^{q-2}\lesssim \|f\|_{L^2(\sph{d-1})}^{q-2}<\infty.$$
The latter estimate follows from the Tomas--Stein inequality \eqref{eq:TS}, which can be invoked since $q\geq \frac{4d}{d-1}$, or equivalently $q-2\geq 2\frac{d+1}{d-1}$.
In particular, we see that $|\widehat{f\sigma}_{d-1}|^{q-2}\in L^1(\R^d)$.
By the Riemann--Lebesgue Lemma, it then follows that the function $K_f=(|\widehat{f\sigma}_{d-1}|^{q-2})^\vee$ is continuous and bounded.
Define $\widetilde{K}_f:\sph{d-1}\times \sph{d-1}\to\Co$, 
$\widetilde{K}_f(\omega,\nu)=K_f(\omega-\nu)$. 
The previous discussion implies that $\widetilde{K}_f \in L^2(\sph{d-1}\times \sph{d-1})$, and therefore the operator $T_f$ is Hilbert--Schmidt.
\end{proof}

\begin{remark}
 Lemma \ref{lem:Kf} covers in particular the ranges $q\geq 8$ if $d=2$, and $q\geq 6$ if $d\geq 3$.
For the case of even exponents contained in these ranges, an alternative argument towards the continuity and boundedness of $K_f$ is available.
Indeed, if $d\geq 3$ and $q=6$, then
$$K_f=(|\widehat{f\sigma}_{d-1}|^4)^\vee=(2\pi)^d(f\sigma_{d-1})^{\ast 2}\ast (f_\star\sigma_{d-1})^{\ast2},$$
which is readily seen to define a continuous function on $\R^d$ since it coincides with a multiple of the convolution of $f\sigma_{d-1}\ast f_\star\sigma_{d-1}\in L^2(\R^d)$ with itself. 
The cases $q=2n$, $n\geq 4$, are then a straightforward consequence.
A similar reasoning applies to the case $(d,q)=(2,8)$, by noting that
$$K_f=(|\widehat{f\sigma}_1|^6)^\vee=(2\pi)^2(f\sigma_1)^{\ast 3}\ast (f_\star\sigma_1)^{\ast 3},$$
and that both functions $(f\sigma_1)^{\ast 3}$ and $(f_\star\sigma_1)^{\ast 3}$ belong to $L^2(\R^2)$.
We may then conclude the cases $(d,q)=(2,2n)$, $n\geq 5$.  
On the other hand, the even endpoint cases $(d,q)=(2,6)$ and $(d,4)$ for $d\geq 3$ lead to kernels which do not necessarily satisfy the conclusions of Lemma \ref{lem:Kf}. For instance, if $(d,q)=(2,6)$, then
$$K_{{\bf 1}}=(|\widehat{\sigma}_1|^4)^\vee=(2\pi)^2\sigma_1^{\ast 4},$$
which according to the asymptotic expansion \eqref{eq:sigma1zero} defines a function which is unbounded at the origin.
In particular, this shows that the first conclusions of Lemma \ref{lem:Kf} cannot be expected to hold in the full Tomas--Stein range $q\geq q_d$ in all dimensions.
On the other hand, we expect the operator $T_f$ to continue to be Hilbert--Schmidt in some region below the threshold $q=\frac{4d}{d-1}$, but have not investigated this point in detail as it is not needed for our purposes. \end{remark}

From Lemma \ref{lem:Kf}, we know that $T_f$ is a Hilbert--Schmidt operator, and in particular
it is compact. The spectral theorem for compact, self-adjoint operators then implies
the existence of an orthonormal basis of $L^2(\sph{d-1})$ consisting of eigenfunctions of $T_f$.
It turns out that the operator $T_f$ is even {\it trace class}, as the next result indicates.

\begin{proposition}\label{prop:HS}
Let $d\geq 2$ and $q\geq \frac{4d}{d-1}$.
Let $\mathbf{0}\neq f\in L^2(\sph{d-1})$.
Then the  operator $T_f$ defined by \eqref{eq:DefTf} above is
trace class.
\end{proposition}

\noindent The proof of Proposition \ref{prop:HS} relies on a classical theorem of Mercer, which is the infinite-dimensional analogue of the well-known statement that any positive semidefinite matrix is the Gram matrix of a certain set of vectors.  

\begin{theorem}[Mercer's Theorem]\label{thm:Mercer}
Let $d\geq 2$.
Given $K\in C^0(\sph{d-1}\times \sph{d-1})$, let $T_K: L^2(\sph{d-1})\to L^2(\sph{d-1})$ be the corresponding integral operator, defined by
\begin{equation}\label{eq:assocKer}
(T_Kf)(\omega)=\int_{\sph{d-1}} f(\nu)K(\omega,\nu) \d\sigma_{d-1}(\nu).
\end{equation}
Assume $K(\omega,\nu)=\overline{K(\nu,\omega)}$ for all $\omega,\nu\in\sph{d-1}$, so that $T_K$ is self-adjoint.
Let $\{\lambda_j\}_{j=0}^\infty$ be the eigenvalues of $T_K$, counted with multiplicity, with $L^2$-normalized eigenfunctions $\{\varphi_j\}_{j=0}^\infty$.
If $T_K$ is positive definite, then
\begin{equation}\label{eq:absunif}
K(\omega,\nu)=\sum_{j=0}^\infty \lambda_j \varphi_j(\omega)\overline{\varphi_j(\nu)},\,\,\,\forall \omega,\nu\in\sph{d-1},
\end{equation}
where the series converges absolutely and uniformly.
\end{theorem}

\noindent  A self-contained proof of Mercer's Theorem can be found in \cite[\S VI.4]{We18}; see also \cite{Br88, Br91}.

\begin{proof}[Proof of Proposition \ref{prop:HS}]
Given $f\in L^2(\sph{d-1})$, consider the kernel $K_f$ in \eqref{eq:defKf}, and $\widetilde{K}_f:\sph{d-1}\times \sph{d-1}\to\Co$, 
$\widetilde{K}_f(\omega,\nu)=K_f(\omega-\nu)$ as before. 
In light of Lemma \ref{lem:Kf}, we have that $\widetilde{K}_f\in C^0(\sph{d-1}\times \sph{d-1})$, and that
\[\widetilde{K}_f(\omega,\nu)=\overline{\widetilde{K}_f(\nu,\omega)}, \text{ for every } \omega,\nu\in\sph{d-1}.\]
In light of Lemma \ref{lem:Tf}, the  operator $T_f$ (associated to the kernel $\widetilde{K}_f$ in the sense of \eqref{eq:assocKer}) is self-adjoint and positive definite.
Let $\{\lambda_j\}_{j=0}^\infty\subset (0,\infty)$ denote the sequence of eigenvalues of $T_f$, counted with multiplicity, with corresponding $L^2$-normalized eigenfunctions $\{\varphi_j\}_{j=0}^\infty$. 
Then Mercer's Theorem applies, and implies that
\begin{align*}
\textup{tr}(T_f)
=\sum_{j=0}^\infty\langle T_f(\varphi_j),\varphi_j\rangle_{L^2(\sph{d-1})}
&=\sum_{j=0}^\infty\lambda_j=\int_{\sph{d-1}}\widetilde{K}_f(\omega,\omega)\,\d\sigma_{d-1}(\omega)\\
&=\omega_{d-1}K_f(0)=\omega_{d-1}\|\widehat{f\sigma}\|_{L^{q-2}(\R^d)}^{q-2}\lesssim \|f\|_{L^2(\sph{d-1})}^{q-2}<\infty.
\end{align*}
In the third identity, we appealed to the absolute and uniform convergence of the series \eqref{eq:absunif} in order to exchange the order of the sum and the integral.
This shows that the operator $T_f$ is trace class, and completes the proof of the proposition.
\end{proof}

\section{New eigenfunctions from old}\label{sec:newold}

In the previous section, we recast the Euler--Lagrange equation \eqref{eq:ELnonconv}
as the eigenvalue problem \eqref{eq:eigvalprobl} for the operator $T_f$, defined in \eqref{eq:DefTf}.
In particular, maximizers of the Tomas--Stein inequality \eqref{eq:TS} were seen to be eigenfunctions of the corresponding operator $T_f$.
In this section, we look for further linearly independent eigenfunctions of $T_f$.

Let us recall a few basic facts about the special orthogonal group $\textup{SO}(d)$ of all $d\times d$ orthogonal matrices of unit determinant.
It acts transitively on $\sph{d-1}$ in the natural way.
As a Lie group, $\textup{SO}(d)$ is  compact, connected, and has dimension $d(d-1)/2$.
Its  Lie algebra consists of skew-symmetric matrices, 
$$\mathfrak{so}(d)=\{A\in \text{Mat}(d\times d;\R): A^\intercal=-A\}.$$
The exponential map, $\exp:\frak{so}(d)\to\textup{SO}(d)$, $A\mapsto e^A$, is surjective onto $\textup{SO}(d)$. 
For more information on matrix Lie groups and algebras, we refer the interested reader to \cite{Ha15} and the references therein.

Given $d\geq 2$ and $A\in \frak{so}(d)$, we define the vector field $\partial_A$ acting on continuously differentiable functions  $f:\sph{d-1}\to\Co$ via
\begin{equation}\label{eq:defPartialA}
 \partial_A f=\frac{\partial}{\partial t}\Big\vert_{t=0} f(e^{tA}\cdot).
\end{equation}
In other words, if $\Theta(t):=e^{tA}$, then for each $\omega\in\sph{d-1}$ we have that
$$(\partial_A f)(\omega)=\lim_{t\to 0}\frac{f(\Theta(t)\omega)-f(\omega)}t.$$

\begin{lemma}\label{lem:linear}
Let $f\in C^1(\sph{d-1})$. Then the map $\frak{so}(d)\to C^0(\sph{d-1}), A\mapsto \partial_A f,$ is linear.
\end{lemma}

\begin{proof}
A function $f:\sph{d-1}\to\Co$ can be extended radially via $F(x)=f(x/|x|)$.
If $f\in C^1(\sph{d-1})$, then the function $F$ is differentiable away from the origin.
Since $e^{tA}$ acts on spheres, we have that $f(e^{tA}\omega)=F(e^{tA}\omega)$, for every $A\in\frak{so}(d)$ and $\omega\in\sph{d-1}$. 
Consequently,
$$\partial_A f(\omega)= \nabla F(\omega)\cdot A\omega,\text{ for every }\omega\in\sph{d-1},$$
from where the claimed linearity follows at once.
\end{proof}
The functional $\Phi_{d,q}$ defined in \eqref{eq:PhidqDef} is invariant under the continuous actions $(t,A)\in\R\times\frak{so}(d)\mapsto e^{tA}f$ and $\xi\in\R^d\mapsto e_\xi f$ in the following sense: 
For every $\mathbf{0}\neq f\in L^2(\sph{d-1})$,
\[\Phi_{d,q}(e^{tA}f)=\Phi_{d,q}(f)=\Phi_{d,q}(e_\xi f),\]
for any $t\in\R$, $A\in \frak{so}(d)$ and $\xi\in\R^d$, where $e_\xi$ stands for the character $e_\xi(\omega)=e^{i\xi\cdot\omega}$, and
\[(e^{tA}f)(\omega):=f(e^{tA}\omega),\,\,\,(e_\xi f)(\omega):=e^{i\xi\cdot\omega}f(\omega).\]
These symmetries give rise to new eigenfunctions in a natural way, as the next result indicates.
We write $\nu=(\nu_1,\ldots,\nu_d)\in\sph{d-1}$, and by $\nu_j f$ we  mean the function defined via $(\nu_j f)(\nu)=\nu_j f(\nu)$.
\begin{proposition}\label{lem:moreEigenfunctions}
Let $d\geq 2$ and $q\geq2\frac{d+1}{d-1}$.
Let $f\in C^1(\sph{d-1})$ be such that $f=f_\star$, and assume $T_f(f)=\lambda \|f\|_{L^2(\sph{d-1})}^{q-2}f$.
Then:
\begin{align}
T_f(\nu_jf)&=\frac{\lambda}{q-1} \|f\|_{L^2(\sph{d-1})}^{q-2}(\nu_jf),\quad \text{ for every }j\in\{1,\ldots,d\},\label{eq:eigenxjf}\\
T_f(\partial_A f)&=\frac{\lambda}{q-1} \|f\|_{L^2(\sph{d-1})}^{q-2}\partial_Af,\quad \text{ for every }A\in\frak{so}(d).\label{eq:eigendA}
\end{align}
\end{proposition}
\begin{proof}
Fix $d\geq 2$ and $q\geq q_d$.
Since there is no danger of confusion, in the course of this proof we shall abbreviate $\d\sigma=\d\sigma_{d-1}$.
Let $g,h\in C^1(\sph{d-1})$ and $A\in\frak{so}(d)$. We claim that
\begin{equation}\label{eq:weakDerivativeSphere}
\int_{\sph{d-1}}h\overline{(\partial_Ag)}\d\sigma=-\int_{\sph{d-1}}(\partial_Ah)\overline{g}\d\sigma.
\end{equation}
Indeed, let $\Theta(t)=e^{tA}$. By a change of variables,
\begin{align*}
\int_{\sph{d-1}}h\overline{(\partial_Ag)}\d\sigma&=\lim_{t\to 0}\frac{1}{t}\int_{\sph{d-1}}h(\omega)\overline{(g(\Theta(t) \omega)-g(\omega))}\d\sigma(\omega)\\
&=\lim_{t\to 0}\frac{1}{t}\int_{\sph{d-1}}(h(\Theta(-t)\omega)-h(\omega))\overline{g(\omega)}\d\sigma(\omega)\\
&=-\int_{\sph{d-1}}(\partial_Ah)\overline{g}\d\sigma.
\end{align*}
Here, we used the $\mathrm{SO}(d)$-invariance of the measure $\d\sigma$, together with the fact that the functions $g,h$ are continuously differentiable, so that the limit commutes with the integral. 
This establishes the claim.
As a consequence (recall that $A^\intercal=-A$),
\begin{align*}
\int_{\sph{d-1}}e^{-ix\cdot \omega}(\partial_Ag)&(\omega)\d\sigma(\omega)
=-\int_{\sph{d-1}}\partial_A(e^{-ix\cdot \omega})g(\omega)\d\sigma(\omega)\\
&=-\int_{\sph{d-1}}ie^{-ix\cdot \omega} (Ax\cdot \omega)g(\omega)\d\sigma(\omega)
=\partial_A\int_{\sph{d-1}}e^{-ix\cdot \omega}g(\omega)\d\sigma(\omega),
\end{align*}
where the application of $\partial_A$ to $\widehat{g\sigma}$ is understood as in \eqref{eq:defPartialA}, but for more general functions defined on $\R^d$. In short, we have verified that the Fourier extension operator commutes with the vector field $\partial_A$,
\[ \widehat{(\partial_Ag)\sigma}=\partial_A(\widehat{g\sigma}). \]
Equation \eqref{eq:eigvalprobl}  can be written, for each $\omega\in\sph{d-1}$, as
\begin{equation}\label{eq:eigrewritten}
 \int_{\sph{d-1}}\int_{\R^d} e^{iz\cdot(\omega-\nu)} f(\nu) |\widehat{f\sigma}(z)|^{q-2} \,\d z\,\d\sigma(\nu)=\lambda\|f\|_{L^2(\sph{d-1})}^{q-2} f(\omega).
 \end{equation}
Let $g\in C^1(\sph{d-1})$ be arbitrary.
Multiplying both sides of \eqref{eq:eigrewritten} by $\overline{\partial_Ag}(\omega)$, and then integrating, yields on the right-hand side
\[ \lambda\|f\|_{L^2(\sph{d-1})}^{q-2}\int_{\sph{d-1}} f\overline{\partial_Ag}\d\sigma=- \lambda\|f\|_{L^2(\sph{d-1})}^{q-2}\int_{\sph{d-1}} (\partial_Af)\overline{g}\d\sigma,\]
and on the left-hand side 
\begin{align}
\int_{\R^d}|\widehat{f\sigma}(z)|^{q-2}&\widehat{f\sigma}(z)\overline{\widehat{(\partial_Ag)\sigma}(z)}\d z
=\int_{\R^d}|\widehat{f\sigma}(z)|^{q-2}\widehat{f\sigma}(z)\partial_A\overline{\widehat{g\sigma}(z)}\d z\notag\\
&=\int_{0}^\infty\int_{\sph{d-1}}|\widehat{f\sigma}_{d-1}(r\omega)|^{q-2}\widehat{f\sigma}(r\omega)\partial_A\overline{\widehat{g\sigma}(r\omega)}\d\sigma(\omega)\,r^{d-1}\d r\label{eq:IBPpartialA}\\
&=-\int_{0}^\infty\int_{\sph{d-1}}\partial_A(|\widehat{f\sigma}(r\omega)|^{q-2}\widehat{f\sigma}(r\omega))\,\overline{\widehat{g\sigma}(r\omega)}\d\sigma(\omega)\,r^{d-1}\d r.\notag
\end{align}
Since $q>2$ and $\widehat{f\sigma}$ is real-valued (recall that $f=f_\star$), the function  $\omega\in \sph{d-1}\mapsto|\widehat{f\sigma}(r\omega)|^{q-2}\widehat{f\sigma}(r\omega)$
is of class $C^1$ whenever $r\geq 0$. In particular, the integration by parts with respect to $\partial_A$ in the inner integral of \eqref{eq:IBPpartialA} is fully justified.
Next, we compute the derivative of the product in the last integrand. 
If $\widehat{f\sigma}(r\omega)=0$, then
\begin{align*}
&\partial_A(|\widehat{f\sigma}(r\omega)|^{q-2}\widehat{f\sigma}(r\omega))
=\lim_{t\to 0}|\widehat{f\sigma}(re^{tA}\omega)|^{q-2}\frac{\widehat{f\sigma}(re^{tA}\omega)-\widehat{f\sigma}(r\omega)}{t}\\
&=|\widehat{f\sigma}(r\omega)|^{q-2}\partial_A(\widehat{f\sigma})(r\omega)=0.
\end{align*}
If $\widehat{f\sigma}(r\omega)\neq 0$, then $\widehat{f\sigma}(z)\neq 0$, for every $z$ in a neighborhood of $r\omega$, and moreover
\begin{align*}
&\partial_A(|\widehat{f\sigma}(r\omega)|^{q-2}\widehat{f\sigma}(r\omega))
=\frac{\d}{\d t}\Big\vert_{t=0}\Big(|\widehat{f\sigma}(re^{tA}\omega)|^{q-2}\widehat{f\sigma}(re^{tA}\omega)\Big)\\
&=|\widehat{f\sigma}(r\omega)|^{q-2}\partial_A\widehat{f\sigma}(r\omega)+\partial_A(|\widehat{f\sigma}(r\omega)|^{q-2})\widehat{f\sigma}(r\omega)\\
&=|\widehat{f\sigma}(r\omega)|^{q-2}\widehat{(\partial_Af)\sigma}(r\omega)+(q-2)|\widehat{f\sigma}(r\omega)|^{q-4}\Re(\overline{\widehat{f\sigma}(r\omega)}\widehat{(\partial_Af)\sigma}(r\omega))\widehat{f\sigma}(r\omega).
\end{align*}
If $f=f_\star$, then $\partial_Af=(\partial_Af)_\star$, and both functions $\widehat{f\sigma}$ and $\widehat{(\partial_Af)\sigma}$ are real-valued. This implies
\[ \partial_A(|\widehat{f\sigma}(r\omega)|^{q-2}\widehat{f\sigma}(r\omega))= (q-1)|\widehat{f\sigma}(r\omega)|^{q-2}\widehat{(\partial_Af)\sigma}(r\omega).\]
The right-hand side of the latter identity defines a continuous function on $\R^d$ which vanishes whenever $\widehat{f\sigma}$ vanishes, and we see that $|\widehat{f\sigma}(r\cdot)|^{q-2}\widehat{f\sigma}(r\cdot)$ is indeed of class $C^1(\sph{d-1})$.
The conclusion is that
\[ \int_{\R^d}|\widehat{f\sigma}(z)|^{q-2}\widehat{f\sigma}(z)\overline{\widehat{(\partial_Ag)\sigma}(z)}\d z
=-(q-1)\int_{\R^d}|\widehat{f\sigma}(z)|^{q-2}\widehat{(\partial_Af)\sigma}(z)\overline{\widehat{g\sigma}(z)} \d z, \]
and therefore
\[ (q-1)\int_{\R^d}|\widehat{f\sigma}(z)|^{q-2}\widehat{(\partial_Af)\sigma}(z)\overline{\widehat{g\sigma}}(z)\d z
=\lambda\|f\|_{L^2(\sph{d-1})}^{q-2}\int_{\sph{d-1}} (\partial_Af)\overline{g}\d\sigma. \]
Recall that $g\in C^1(\sph{d-1})$ was arbitrary. Thus we now know that
\[ (q-1)\langle T_f(\partial_Af) ,g\rangle_{L^2(\sph{d-1})}=\lambda\|f\|_{L^2(\sph{d-1})}^{q-2}\langle \partial_Af,g\rangle_{L^2(\sph{d-1})}, \]
for every $g\in C^1(\sph{d-1})$, from where it becomes apparent that
\[ T_f(\partial_Af)=\frac{\la}{q-1}\|f\|_{L^2(\sph{d-1})}^{q-2}\partial_A f,\quad \sigma\text{-a.e. in }\sph{d-1}. \]
The proof of \eqref{eq:eigendA} is now complete.

The verification of \eqref{eq:eigenxjf} is analogous, but simpler since one does not need to differentiate the function $f$.
One first realizes that the function $e^{it\nu_j}f$  is also a critical point of $\Phi_{d,q}$, and as such it satisfies the Euler--Lagrange equation \eqref{eq:ELnonconv} with the same eigenvalue as $f$. Both sides of this equation are then differentiated with respect to $t$, and finally one sets $t=0$.
This concludes the proof of the proposition. 
\end{proof}

The proof of Proposition \ref{lem:moreEigenfunctions} can be simplified in the special case 
when $q$ is an even integer, since the Euler--Lagrange equation in convolution form,  \eqref{eq:simpleEL},  leads to 
expressions which are easier to handle. 
We leave the details to the interested reader, and proceed to study the linear independence of the new eigenfunctions which have just been discovered. The next result is elementary.

\begin{lemma}\label{lem:NewOld2}
Let $d\geq 2$. If $f:\sph{d-1}\to\Co$ is not identically zero in $L^2(\sph{d-1})$, then the $d$ functions $\nu_1 f,\ldots, \nu_d f$ are linearly independent over $\Co$.
\end{lemma}
\begin{proof}
Consider a linear combination satisfying
$$z_1 \nu_1 f(\nu)+\ldots+z_d \nu_d f(\nu)=0,$$
for $\sigma_{d-1}$-almost every $\nu\in\sph{d-1}$, for some $ z=(z_1,\ldots,z_d)\in \Co^d$.
Let $E=\{\nu\in\sph{d-1}: f(\nu)\neq 0\}$ be the subset of the sphere, of positive $\sigma_{d-1}$-measure, where $f$ does not vanish.
Write $ z= a+ i b$, with $ a,  b\in\R^d$.
We then have that $ a\cdot \nu= b\cdot\nu=0$, for every $\nu\in E$.
It will be enough to find a basis of $\R^d$ inside $E$, for that will imply $ a= b=(0,\ldots,0)$, and consequently the linear independence over $\Co$ of the set $\{\nu_j f\}_{j=1}^d$. 

We proceed by induction on the dimension.
If $d=2$, then we can find two linearly independent vectors $\zeta_1,\zeta_2\in E$, for otherwise the set $E$ has at most two elements.
Let $d\geq 3$. 
Reasoning as before, there exist two linearly independent vectors $\zeta_1,\zeta_2\in E$.
If the existence of $k\in\{2,\ldots,d-1\}$ linearly independent vectors $\zeta_1,\ldots,\zeta_k\in E$ has already been established, then we can find a next one $\zeta_{k+1}\in E$ so that $\{\zeta_1,\ldots, \zeta_{k+1}\}$ is still a linearly independent set, for otherwise $E\subseteq \text{span}\{\zeta_1,\ldots,\zeta_k\}\cap \sph{d-1}$ would be a $\sigma_{d-1}$-null set.
This concludes the proof of the lemma.
\end{proof}

In order to determine the maximal number of linearly independent eigenfunctions $\partial_A f$,
we need the following  fact from Lie theory,  which is only of interest if $d\geq 3$ since $\dim\frak{so}(2)=1$.
It also reveals a curious difference that occurs in the case $d=4$.

\begin{lemma}\label{fact:Lie}
The minimal codimension of a proper subalgebra of $\frak{so}(d)$ equals $d-1$ if $d\geq 3$, $d\neq 4$, and equals 2 if $d=4$.
\end{lemma}

\noindent This result is classical. The proof is essentially contained in \cite{AFG12, Ho65}, but for completeness we provide the details, in a language which may be friendlier to the more analytically-minded reader.
Lemma \ref{fact:Lie} also follows from an inspection of the list of maximal subalgebras of the classical compact Lie algebras; see e.g. \cite[Table on p. 539]{Ma66}.
The special role played by dimension $d=4$ stems from the fact that the group SO$(4)/\{\pm I\}$ is {\it not} simple, as opposed to all other  groups SO$(d)$, $d\neq 4$, which are simple (after modding out by $\{\pm I\}$ if $d$ is even).
This is related to the fact that a rotation in $\R^4$ is determined by two 2-planes and two angles. In turn, this gives rise to a non-trivial proper normal subgroup (associated to one 2-plane and one angle), which reveals that SO$(4)/\{\pm I\}$ is not simple.

\begin{proof}[Proof of Lemma \ref{fact:Lie}]
If $d\geq 3$, $d\neq 4$, then it will suffice to show the following claim: 
If $\frak{g}$ is a subalgebra of $\frak{so}(d)$ of dimension $d(d-1)/2-r$, then $r(r+1)\geq d(d-1)$.
This is the content of \cite[Lemma 1]{Ho65}, and can be proved as follows.

Let $\frak{g}$ be a subalgebra of $\frak{so}(d)$.
Recall that there exists a unique (connected) Lie subgroup\footnote{The group $G$ consists precisely of elements of the form $e^{X_1}\cdots e^{X_m}$ with $X_1,\ldots,X_m\in\frak{g}$.} $G$ of $\textup{SO}(d)$ whose Lie algebra equals $\frak{g}$; see e.g.\@ \cite[Theorem 5.20]{Ha15}.
Given $X,Y\in\frak{so}(d)$, denote $K(X,Y)=\textup{tr}(XY)$, where $\textup{tr}(A)$ stands for the trace of the matrix $A$.
This defines a negative-definite bilinear form, the so-called ``Killing form'' on $\frak{so}(d)$.
Let $\frak{h}$ denote the orthogonal complement of $\frak{g}$ in $\frak{so}(d)$ with respect to the Killing form $K$, so that  $\frak{so}(d)=\frak{g}\oplus \frak{h}$.
In particular, $K(X,Y)=0$ for every $X\in \frak{g}$ and $Y\in \frak{h}$.
Moreover, if $\dim\frak{g}=d(d-1)/2-r$, then $\dim\frak{h}=r$.

Define the adjoint map $\textup{Ad}:\textup{SO}(d)\to\textup{Aut}(\frak{so}(d))$, $\text{Ad}_\Theta(X)=\Theta X\Theta^{-1}$, for every $\Theta\in\textup{SO}(d)$ and $X\in\frak{so}(d)$.
At the level of Lie algebras (i.e. taking derivatives), this corresponds to a map $\textup{ad}:\frak{so}(d)\to\text{End}(\frak{so}(d))$, $\text{ad}_X(Y)=[X,Y]:=XY-YX$, for every $X, Y\in\frak{so}(d)$.
One easily checks that the adjoint map preserves the Killing form, in the sense that
$$K(X,Y)=K(\textup{Ad}_\Theta(X),\textup{Ad}_\Theta(Y)),$$
for every $X,Y\in\frak{so}(d)$ and $\Theta\in\textup{SO}(d)$.
As a consequence, $\textup{Ad}_\Theta(\frak{h})\subseteq\frak{h}$ for any $\Theta\in G$, and the adjoint map restricts to a map $\textup{Ad}|_{G}: G\to O(\frak{h})$.
At the level of Lie algebras, this corresponds to a map $\textup{ad}|_{G}: \frak{g}\to \frak{o}(\frak{h})$.
Let $\mathfrak{k}=\text{ker}(\textup{ad}|_{G})$, which is an ideal in $\frak{g}$. In particular,  any element of
$$[\frak{k},\frak{so}(d)]=[\frak{k},\frak{h}]+[\frak{k},\frak{g}]$$
belongs to $\frak{k}$ since the first term on the right-hand side of the latter expression vanishes, 
while the second term is contained in $\frak{k}$ since $\frak{k}$ is an ideal in $\frak{g}$.
This shows that $\frak{k}$ is not only an ideal in $\frak{g}$, but actually an ideal in $\frak{so}(d)$.
But the Lie algebra $\frak{so}(d)$ is known to be simple if $d\geq 3$, $d\neq 4$, see e.g. \cite[Cor.\@ 8.47]{Ha15}, whence $\frak{k}=0$.
Therefore the map $\textup{ad}|_{G}$ is injective, and consequently
$$d(d-1)/2-r=\dim\frak{g}\leq\dim\frak{o}(\frak{h})=r(r-1)/2,$$
where the last identity follows from the fact that $\dim\frak{h}=r.$
Rearranging, we obtain $r(r+1)\geq d(d-1)$, as claimed. 
Further note that the latter inequality is an equality if $\frak{g}=\frak{so}(d-1)$.

The Lie algebra $\frak{so}(4)$ is {\it not} simple, and so the preceding proof breaks down if $d=4$. In that case, the maximal subgroups of $\text{SO}(4)$ are listed in \cite[Ex.\@ 5.11]{AFG12}, from where one easily verifies that the minimal codimension of a proper subalgebra of $\frak{so}(4)$ is equal to 2.
The key is to note that $\text{O}(2)\times \text{SO}(3)$ is a 4-dimensional subgroup of $\text{SO}(4)$, and so the codimension of the corresponding Lie subalgebra is $\frac{4\times 3}2-4=2$. This completes the proof of the lemma.
\end{proof}

\begin{lemma}\label{lem:NewOld4}
Let $f\colon \sph{d-1}\to\R$ be a continuously differentiable, non-constant function. 
If $d\geq 2$ and $d\neq 4$, then there exist linearly independent matrices $A_1,\ldots, A_{d-1}\in\frak{so}(d)$ such that $\partial_{A_1} f,\ldots, \partial_{A_{d-1}} f$ are linearly independent over $\Co$.
If $d=4$, then there exist linearly independent matrices $A_1,A_2\in\frak{so}(4)$ such that $\partial_{A_1} f, \partial_{A_2} f$ are linearly independent over $\Co$.
\end{lemma}

\noindent In the statement of Lemma \ref{lem:NewOld4}, note that the function $f$ is assumed to be real-valued.
Therefore $\partial_{A}f$ is also real-valued, for any $A\in\mathfrak{so}(d)$, and in particular the claimed linear independence over $\Co$ is equivalent to that over $\R$.

\begin{proof}[Proof of Lemma \ref{lem:NewOld4}]
The case $d=2$ is clear since for any nonzero $A\in\frak{so}(2)$, $\partial_Af\neq\mathbf{0}$.
Let $d\geq 3$, $d\neq 4$.
Consider the map $D\colon \mathfrak{so}(d)\to C^0(\sph{d-1})$,  given by $D(A)=\partial_Af$.
In light of Lemma \ref{lem:linear}, $D$ is linear. Let $V=\textup{ker}(D)$.
By the Rank-Nullity Theorem, the dimension of the image of $D$ equals $d(d-1)/2-\dim(V)$, and so it suffices to show that
$r:=\dim(V)\leq(d-1)(d-2)/2$.
If that were not the case, then from Lemma \ref{fact:Lie} it would follow that the Lie algebra generated by $V$ equals $\frak{so}(d)$.
The Lie algebra generated by any basis $\{A_1,\ldots,A_r\}$ of $V$ would likewise equal $\frak{so}(d)$.
By definition of $V$, we have that $\partial_{A_1} f=\ldots=\partial_{A_r} f={\bf 0}$, which in turn implies that for any matrix $B$ of the form
\begin{equation}\label{eq:prodexp}
B=e^{t_1 X_1}\cdots e^{t_k X_k} \in \textup{SO}(d),
\end{equation}
we have $f(Be_1)=f(e_1)$ (here $e_1=(1,0,\ldots,0)$ is the first canonical vector of $\R^d$, but we could have taken any other fixed vector of unit length),
where $k\geq 1$, $t_1,\ldots, t_k\in\R$ and $X_i\in\{A_1,\ldots,A_r\}$.
Now, every element of $\textup{SO}(d)$ can be written in the form \eqref{eq:prodexp}; this uses the connectedness of $\textup{SO}(d)$, see \cite[Lemma 6.2]{JS72}.
Since the action of $\textup{SO}(d)$ on $\sph{d-1}$ is transitive, we conclude that the function $f$ is constant.
The contradiction shows that $\dim(V)\leq(d-1)(d-2)/2$, as desired.
The case $d=4$ can be treated in an entirely analogous way.
\end{proof}

\section{Bootstrapping maximizers}\label{sec:BconstantMax}
In this section, we focus on several situations in which the knowledge that constant functions maximize the functional $\Phi_{d,q}$ implies an analogous, but structurally stronger, statement for $\Phi_{d,q+2}$.
To make this precise, consider the set 
\begin{align*}
\mathcal{M}_d(L^2\to L^{q})=\{f\in L^2(\sph{d-1})\colon f \text{ is a real-valued maximizer of }\Phi_{d,q} \}.
\end{align*}
It follows from \cite[Theorem 1.1]{FVV11} that complex-valued maximizers for $\Phi_{d,q}$ exist if $d\geq 2$ and $q>q_d$. If moreover $q>q_d$ is an even integer, then real-valued maximizers exist in view  of Lemma \ref{lem:posequal}. From \cite{CS12a,Sh16}, we also know that $\mathcal{M}_3(L^2\to L^{4})\neq\emptyset$ and $\mathcal{M}_2(L^2\to L^{6})\neq\emptyset$. 
In this way, we see that $\mathcal{M}_d(L^2\to L^{q})\neq\emptyset$ whenever $q\geq q_d$ is an even integer. 
For  $d\geq 2$ and $q\geq q_d$, define the quantity
\begin{equation}\label{eq:defgammad}
\gamma_d(q):=
\omega_{d-1}\frac{1+q}{2d+q-\delta_{d,4}},
\end{equation}
where  the Kronecker delta satisfies $\delta_{m,n}=1$ if $m=n$, and $\delta_{m,n}=0$ if $m\neq n$.
A special feature of the sphere is that constant functions are critical points of $\Phi_{d,q}$, for {\it every} $q\in[q_d,\infty)$. 
Our next result is restricted to even integers $q$, and makes the following bootstrapping scheme precise:
If constants maximize $\Phi_{d,q}$, and the values of $\Phi_{d,q+2}(\one)$ and $\Phi_{d,q}(\one)$ satisfy a certain inequality, then constants are the unique real-valued maximizers of $\Phi_{d,q+2}$. 

\begin{theorem}\label{prop:bootstrappedMaximizer}
	Let $d\geq 2$, and let $q$ be an even integer such that $q\geq 2\frac{d+1}{d-1}$.
	Suppose that $\one\in\mathcal{M}_d(L^2\to L^q)$, and that there exists $k_0\geq 0$ such that $\Phi_{d,q+2+2k}(\one)\geq \gamma_d(q+2k)\Phi_{d,q+2k}(\one)$, for every $k\in[0,k_0]\cap\N_0$. 
	Then 
	\[\mathcal{M}_d(L^2\to L^{q+2+2k})=\{c\one\colon c\in\R, c\neq 0\},\] for every $k\in[0,k_0]\cap\N_0$.
\end{theorem}

\noindent In the proof of Theorem \ref{prop:bootstrappedMaximizer} below, we note that the assumption $q\in 2\N$ is used twice, namely, to invoke Theorem \ref{thm:smoothnessTheorem} (thereby ensuring that a critical point of the functional $\Phi_{d,q+2}$ is continuously differentiable) and Lemma \ref{lem:NewOld4} (thereby ensuring that a sufficient number of linearly independent eigenfunctions exist).

\begin{proof}[Proof of Theorem \ref{prop:bootstrappedMaximizer}]
Since the general statement follows from the case $k_0=0$ by induction,  we limit ourselves to proving the special case $k_0=0$. 
Fix $d,q$ as in the statement of the theorem.
It was already observed that the set $\mathcal{M}_d(L^2\to L^{q+2})$ is non-empty.
Since $q$ is an even integer, our discussion in \S \ref{sec:symm} implies that nonnegative, antipodally symmetric maximizers exist. It will then suffice to show that any real-valued, antipodally symmetric, non-constant critical point $f$ of $\Phi_{d,q+2}$ satisfies $\Phi_{d,q+2}(f)<\Phi_{d,q+2}(\one)$. 
If no such $f$ exists, then constant functions are the unique real-valued maximizers of $\Phi_{d,q+2}$, and there is nothing left to prove.
Therefore no generality is lost in assuming that such an $f$ does exist; we further assume it to be $L^2$-normalized, $\|f\|_{L^2}=1$.

Since $f$ is an $L^2$-normalized critical point of $\Phi_{d,q+2}$, it satisfies the Euler--Lagrange equation
\begin{equation}\label{eq:ELhere}
\Bigl(|\widehat{f\sigma}_{d-1}|^{q} \widehat{f\sigma}_{d-1}\Bigr)^{\vee}\Bigl\vert_{\sph{d-1}}
=\lambda  f,
\quad\sigma_{d-1}\text{-a.e. on }\sph{d-1},
\end{equation}
for some $\lambda>0$. 
Multiplying both sides of the latter identity by ${f}$ (recall that $f$ is assumed to be real-valued)
and integrating, reveals that $\lambda=\Phi_{d,q+2}(f)$.
Since $q=2n$ is an even integer, and $f=f_\star$, equation \eqref{eq:ELhere} can be written in convolution form, 
\[\Bigl((f\sigma_{d-1})^{\ast (2n+1)}\Bigr) \Big\vert_{\mathbb S^{d-1}}
=(2\pi)^{-d}\la  f,\quad\sigma_{d-1}\text{-a.e. on }\mathbb S^{d-1}.\]
Theorem \ref{thm:smoothnessTheorem} then implies that $f\in C^\infty(\sph{d-1})$, but in the sequel we will only use the fact that $f$ is continuously differentiable.
The Euler--Lagrange equation \eqref{eq:ELhere} can be equivalently rewritten as
$T_f(f)=\la f,$
where $T_f$ is the integral operator  with convolution kernel $K_f=(|\widehat{f\sigma}_{d-1}|^q)^\vee=(2\pi)^d(f\sigma_{d-1})^{\ast (2n)}$. 

Since the operator $T_f$ is associated to the exponent $q+2\geq 4d/(d-1)$, we may invoke Proposition \ref{prop:HS}, and ensure that $T_f$ is trace class. 
Let $\{\lambda_j\}_{j=0}^\infty$ denote the sequence of its eigenvalues, counted with multiplicity, with corresponding $L^2$-normalized eigenfunctions $\{\varphi_j\}_{j=0}^\infty$. 
Since $T_f$ is self-adjoint and positive definite (Lemma \ref{lem:Tf}), its eigenvalues satisfy $\lambda_j>0$,  
for all $j$.
Compactness of $T_f$  ensures that $\lambda_j\to 0$, as $j\to\infty$. 
If moreover the function $f$ is nonnegative, then the Krein--Rutman Theorem \cite{KR50} further reveals that the eigenvalue $\la$ is the largest one; however, this fact will not be needed in the remainder of the proof.

Set $\widetilde{K}_f:\sph{d-1}\times \sph{d-1}\to\R$, $\widetilde{K}_f(\omega,\nu)=K_f(\omega-\nu)$. Mercer's Theorem \ref{thm:Mercer} ensures that
$$\widetilde{K}_f(\omega,\nu)
=\sum_{j=0}^{\infty} \lambda_j \varphi_j(\omega)\overline{\varphi_j(\nu)},\text{ for all } \omega, \nu\in\sph{d-1},$$
where the series converges absolutely and uniformly.
Thus we may interchange the sum and the integral, and conclude that
\begin{equation}\label{eq:trace}
\int_{\sph{d-1}} \widetilde{K}_f(\omega,\omega)\,\d\sigma_{d-1}(\omega)
=\sum_{j=0}^{\infty} \lambda_j.
\end{equation}
We have already noted that the function $f$ is an eigenfunction of the operator $T_f$ with eigenvalue $\lambda=\Phi_{d,q+2}(f).$
In \S \ref{sec:newold}, we have determined $2d-1$ further eigenfunctions of $T_f$, provided $d\neq 4$, namely
\begin{equation}\label{eq:2d-1}
\nu_1f,\ldots,\nu_df,\partial_{A_1}f,\ldots,\partial_{A_{d-1}}f, \text{ for some  }A_1,\ldots,A_{d-1}\in\frak{so}(d),
\end{equation}
 while for $d=4$ we have determined $6$ further eigenfunctions of $T_f$,
\[ \nu_1f,\nu_2f,\nu_3f,\nu_4f,\partial_{A_1}f,\partial_{A_{2}}f, \text{ for some  }A_1,A_{2}\in\frak{so}(4), \]
each corresponding to the same eigenvalue $(1+q)^{-1}{\lambda}$. 
This is the content of Proposition \ref{lem:moreEigenfunctions}; 
 here we are using the fact that $f$ is continuously differentiable and non-constant. 

Let us finish the proof of the theorem in the case $d\neq 4$. 
Since $f$ is antipodally symmetric,
Lemmata \ref{lem:NewOld2} and \ref{lem:NewOld4} together imply that the $2d-1$ functions in \eqref{eq:2d-1} are linearly independent over $\Co$.
Indeed, the functions $\{\nu_j f\}_{j=1}^d$ are real-valued and antipodally anti-symmetric, whereas the functions $\{\partial_{A_j} f\}_{j=1}^{d-1}$ are real-valued and antipodally symmetric.
Since $\lambda_j> 0$, for all $j$, we can use \eqref{eq:trace} to estimate
\begin{equation}\label{eq:Concl1}
\int_{\sph{d-1}} \widetilde{K}_f(\omega,\omega)\,\d\sigma_{d-1}(\omega)> \lambda+(2d-1)\frac{\lambda}{1+q}=\frac{2d+q}{1+q}\lambda.
\end{equation}
On the other hand, 
$$\widetilde{K}_f(\omega,\omega)=K_f(0)
=(2\pi)^d(f\sigma_{d-1})^{\ast (2n)}(0)=(2\pi)^d\|(f\sigma_{d-1})^{\ast n}\|_{L^2(\R^d)}^2.$$
Consequently,
\begin{equation}\label{eq:Concl2}
\int_{\sph{d-1}} \widetilde{K}_f(\omega,\omega)\,\d\sigma_{d-1}(\omega)
=\omega_{d-1}(2\pi)^d\|(f\sigma_{d-1})^{\ast n}\|_{L^2(\R^d)}^2
=\omega_{d-1}\Phi_{d,q}(f),
\end{equation}
where the latter identity follows from the expression for $\Phi_{d,q}$ in \eqref{eq:PhiConvForm}  and the normalization $\|f\|_{L^2}=1$.
By assumption, we have that
\begin{equation}\label{eq:concl5}
\Phi_{d,q}(f)\leq \Phi_{d,q}({\bf 1}).
\end{equation}
Since $\Phi_{d,q+2}(f)=\lambda$, it follows from \eqref{eq:Concl1}, \eqref{eq:Concl2}, \eqref{eq:concl5} that
\[\frac{2d+q}{1+q} \Phi_{d,q+2}(f)=\frac{2d+q}{1+q}\la<\omega_{d-1}\Phi_{d,q}(f)\leq\omega_{d-1}\Phi_{d,q}({\bf 1})\leq \frac{2d+q}{1+q} \Phi_{d,q+2}({\bf 1}),\]
where the last inequality holds by assumption.
This implies
$$\Phi_{d,q+2}(f)<\Phi_{d,q+2}({\bf 1}),$$
as had to be shown.
This completes the proof of the theorem when $d\neq 4$. 
The argument for $d=4$ is exactly the same, now taking into account the existence of six linearly independent eigenfunctions associated to the eigenvalue $(1+q)^{-1}\lambda$. 
\end{proof}

\begin{remark}\label{rem:upperBoundEdq}
	The function $h=\omega_{d-1}^{-1/2}\one$ is an $L^2$-normalized critical point of $\Phi_{d,q+2}$, for every admissible $q\geq q_d$. 
	The functions $\nu_1,\dotsc,\nu_d$ are eigenfunctions of the corresponding operator $T_h$, each associated to the eigenvalue $(1+q)^{-1}\Phi_{d,q+2}(\one)$. 
	Even though $\partial_A h={\bf 0}$, we can still
	apply part of the argument from the proof of Theorem \ref{prop:bootstrappedMaximizer}, yielding the upper bound
	\begin{equation}\label{eq:usefullater}
	 \Phi_{d,q+2}(\one)\leq \frac{\omega_{d-1}(1+q)}{d+1+q}\Phi_{d,q}({\bf 1}). 
	 \end{equation}
	 	In this way, we see that the inequality in the assumption of Theorem \ref{prop:bootstrappedMaximizer}, 
	\[\Phi_{d,q+2}(\one)\geq \frac{\omega_{d-1}(1+q)}{2d+q-\delta_{d,4}} \Phi_{d,q}(\one),\] 
	 seems to be non-trivial for general $q\geq q_d$; note, for instance, that 
	$$\Abs{\frac{\omega_{d-1}(1+q)}{d+1+q}-\frac{\omega_{d-1}(1+q)}{2d+q-\delta_{d,4}}}\to 0, \quad\text{as } q\to\infty.$$
\end{remark}

\section{Proof of Theorem \ref{thm:MainThm}}\label{sec:ProofMain}

In this section, we prove Theorem \ref{thm:MainThm}, whose strategy can be summarized in the following steps:
\begin{itemize}
		\item From Theorem \ref{prop:bootstrappedMaximizer}, we know that if constant functions maximize the functional $\Phi_{d,q}$, for some even integer $q$, and there exists $k_0\geq 0$ such that 
		\[\Phi_{d,q+2+2k}(\one)\geq \gamma_d(q+2k)\Phi_{d,q+2k}(\one),\]
		for every $k\in[0,k_0]\cap\N_0$, then constants are the unique real-valued maximizers of $\Phi_{d,q+2+2k}$, for every $k\in[0,k_0]\cap\N_0$. 
		\item Constant functions are known to be the unique real-valued maximizers of $\Phi_{d,4}$ for $d=3$, \cite{Fo15}, and $d\in\{4,5,6,7\}$, \cite{COS15}.
		\item The inequality $\Phi_{d,q+2}(\one)\geq \gamma_d(q)\Phi_{d,q}(\one)$ holds in dimensions $d=2$ and $3\leq d\leq 7$, for every even integer $6\leq q\leq 28$ and $4\leq q\leq 28$, respectively. This is the content of Proposition \ref{lem:verifiedIneq} below.
		\item If $d=3$, then the latter point can be sharpened as follows: The inequality $\Phi_{3,q+2}(\one)>\gamma_3(q)\Phi_{3,q}(\one)$ holds, for every even integer $q\geq 6$. This will be proved in Proposition \ref{prop:d=3} below.
		\item By a more involved refinement of the argument for $d=3$, we will show that the inequality $\Phi_{d,q+2}(\one)>\gamma_d(q)\Phi_{d,q}(\one)$ holds, for all $d\in\{2,3,4,5,6,7\}$ and {every} $q\geq q_\star(d)$, for some explicit $q_\star(d)\in[8,12]\cap 2\N$. This is the content of Proposition \ref{prop:sharpIneqD} below.
\end{itemize}

\noindent Our task is thus reduced to showing that the inequality
\begin{equation}\label{eq:ineqForPhis}
\Phi_{d,q+2}(\one)\geq  \gamma_d(q)\Phi_{d,q}(\one)
\end{equation}
holds, for all $d\in\{2,3,4,5,6,7\}$ and even $q\geq q_d$.
We split the analysis, and deal with the cases of small $q$ and large $q$ in \S \ref{sec:proofSmallq} and \S \ref{sec:proofLargeq}, respectively.

\subsection{ Proof of inequality \eqref{eq:ineqForPhis} for small values of $q$}\label{sec:proofSmallq}

The following result is a straightforward consequence of Lemma \ref{lem:explicitL2norms}. The different behavior observed in dimension $d=9$ is one of the reasons why Theorem \ref{thm:MainThm} is restricted to dimensions $d\leq 7$.

\begin{lemma}\label{cor:2fold3fold}
	Let $d\in\{3,5,7\}$. Then:
	\begin{equation}\label{eq:ineq24}
	\Phi_{d,6}({\bf 1})\geq \gamma_d(4)\Phi_{d,4}({\bf 1}).
	\end{equation}
	If $d=9$, the  reverse inequality holds in the strict sense that $\Phi_{9,6}({\bf 1})< \gamma_9(4)\Phi_{9,4}({\bf 1})$.
\end{lemma}

\begin{proof}
	For odd $d\in 2\N+1$, consider the quantity
	\begin{equation}\label{eq:Ed4def} 
	E(d,4):=\frac{\Phi_{d,6}(\one)}{\omega_{d-1}\Phi_{d,4}(\one)}-\frac{5}{2d+4}=\frac{\norma{\sigma_{d-1}^{\ast 3}}_{L^2(\R^d)}^2}{\omega_{d-1}^2\norma{\sigma_{d-1}^{\ast 2}}_{L^2(\R^d)}^2}-\frac{5}{2d+4},
	\end{equation}
	where the second identity follows from \eqref{eq:Phi1vsNorms}.
	Note that inequality \eqref{eq:ineq24} is equivalent to $E(d,4)\geq 0$. From the explicit values computed in Lemma \ref{lem:explicitL2norms}, it follows that $E(3,4)=0$, $E(5,4)=\frac{1597}{40040}$, $E(7,4)=\frac{379609}{104838048}$, and $E(9,4)=-\frac{5596158031}{187343544960}$.
\end{proof}
Invoking \eqref{eq:Phi1vsNorms} and \eqref{eq:equivalenceL2norm}, inequality \eqref{eq:ineqForPhis} can be restated in the following equivalent forms:
\begin{equation}\label{eq:restatedPhi_inequality}
\sigma_{d-1}^{\ast (q+2)}(0)\geq\omega_{d-1}\gamma_d(q)\sigma_{d-1}^{\ast q}(0),\,\text{ or }\, \sigma_{d-1}^{\ast (q+1)}(1)\geq\omega_{d-1}\gamma_d(q)\sigma_{d-1}^{\ast (q-1)}(1).
\end{equation}
The values $\sigma_{d-1}^{\ast (2n)}(0)$ and $\sigma_{d-1}^{\ast (2n+1)}(1)$ increase quickly with $n$.
Therefore, instead of working directly with the difference $\Phi_{d,q+2}(\one)-\gamma_d(q)\Phi_{d,q}(\one)$, we find it convenient to consider the quantity $E(d,q)$ defined as
\begin{equation}\label{eq:defEdq}
E(d,q):=\frac{\Phi_{d,q+2}(\one)}{\omega_{d-1}\Phi_{d,q}(\one)}-\frac{\gamma_d(q)}{\omega_{d-1}},
\end{equation}
 for integers $d\geq 2$ and even integers $q\geq q_d$.
 Naturally, this boils down to \eqref{eq:Ed4def} if $d\geq 3$ is odd and $q= 4$.
Moreover,  inequality \eqref{eq:ineqForPhis} can be equivalently recast as $E(d,q)\geq 0$.
In view of \eqref{eq:equivalenceL2norm}, the following equivalent formulations are available:
\begin{equation}\label{eq:defEdq2} 
E(d,q)=\frac{\sigma_{d-1}^{\ast (q+2)}(0)}{\omega_{d-1}^2\sigma_{d-1}^{\ast q}(0)}-\frac{1+q}{2d+q-\delta_{d,4}}=\frac{\sigma_{d-1}^{\ast (q+1)}(1)}{\omega_{d-1}^2\sigma_{d-1}^{\ast (q-1)}(1)}-\frac{1+q}{2d+q-\delta_{d,4}}. 
\end{equation}
	The quantity $E(d,q)$  can also be written in terms of the densities $p_n(m-1/2;r)$ defined in \eqref{eq:relationSigmaDensity}, as follows (recall that $m=(d-1)/2$):
	\begin{equation}\label{eq:EdqandDensity}
	E(d,q)=\lim_{r\to 0^+}\frac{p_{q+2}(m-1/2;r)}{p_{q}(m-1/2;r)}-\frac{1+q}{2d+q-\delta_{d,4}}=\frac{p_{q+1}(m-1/2;1)}{p_{q-1}(m-1/2;1)}-\frac{1+q}{2d+q-\delta_{d,4}}.
	\end{equation}
The following result holds.
\begin{lemma}\label{lem:rationalEdq}
	Let $d\geq 3$ be an odd integer. For every even integer $q\geq 4$, $E(d,q)$ is a rational number satisfying $E(d,q)\leq 1$.
\end{lemma}
\begin{proof}
	In order to verify that $E(d,q)\in\mathbb{Q}$, by \eqref{eq:defEdq2} and  \eqref{eq:BSnthConvo} it suffices to check that the following number is rational:
	\[ \ell(d):=\frac{1}{\omega_{d-1}^2}\biggl(\frac{\omega_{d-1}\Gamma(d-1)}{2^{\frac{d-1}{2}}\Gamma(\frac{d-1}{2})}\biggr)^2=\frac{\Gamma(d-1)^2}{2^{d-1}\Gamma(\frac{d-1}{2})^2}. \]
	But $d$ is odd, and so $(d-1)/2$ is an integer, which readily implies $\ell(d)\in\Qr$. The second assertion is an immediate  consequence of inequality \eqref{eq:usefullater}. Indeed,
	\[ E(d,q)\leq \frac{1+q}{d+1+q}-\frac{1+q}{2d+q}=\frac{(d-1)(1+q)}{(d+1+q)(2d+q)}\leq 1. \]
	This completes the proof of the lemma.
\end{proof}

We are now ready to prove inequality \eqref{eq:ineqForPhis} for small values of $q$.
\begin{proposition}\label{lem:verifiedIneq}
	Let $d\in\{3,4,5,6,7\}$. Then the inequality
	\begin{equation}\label{eq:Phi_inequality}
	\Phi_{d,q+2}(\one)\geq \gamma_d(q)\Phi_{d,q}(\one)
	\end{equation}
	holds, for every $q\in[4,28]\cap2\N$, and is strict except for $(d,q)=(3,4)$.
	If $d=2$, then the strict inequality in \eqref{eq:Phi_inequality} holds, for every $q\in[6,28]\cap2\N$.
\end{proposition}

\begin{proof}
We treat the cases of odd and even dimensions separately, providing an analytic proof in the former case, and a numerically assisted proof in the latter case.\\

\noindent {\bf Case} $d\in\{3,5,7\}$.
The case $q=4$ is the content of Lemma \ref{cor:2fold3fold}. 
For  $q\in [6,28]\cap 2\N$, we proceed as follows.
Thanks to \eqref{eq:defEdq2}, identity \eqref{eq:BSnthConvo} at
$r=1$ can be used to compute the relevant values of $E(d,q)$ {\it exactly}. 
In view of Lemma \ref{lem:rationalEdq}, we could then list those values in the form $E(d,q)=\frac{a}{b}$, for some $a,b\in\Z$ (as in the proof of Lemma \ref{cor:2fold3fold}). 
Since the number of digits of $a$ and $b$ grows quickly as $d,q$ increase, and we are only interested in checking that $E(d,q)\geq 0$, we found it more convenient to instead present the decimal expansion of $100\times E(d,q)$.
The result, truncated to two decimal places, is displayed in Table \ref{table:357Edq}.
It was produced with the following \textit{Maxima} code, having in mind the second expression  for $E(d,q)$ from \eqref{eq:EdqandDensity}:\\

	\begin{lstlisting}
	/* Calculation of 100*E(d,q), for d=3,5,7 and q=4,6,...,28 */
	C(m,x):=sum((m-1+k)!/(2^k*k!*((m-1-k)!))*x^k,k,0,m-1);
	coef(m,n,r,j):=ratcoef(C(m,t)^r*C(m,-t)^(n-r),t,j);
	Heaviside(x):=if(is(x<0)) then 0 else 1;
	m(d):=(d-1)/2;
	p(n,d,x):=(gamma(2*m(d))/(2^(m(d))*gamma(m(d))))^n*
		  sum(binomial(n,r)*(-1)^(m(d)*r)*Heaviside(n-2*r-x)*
		  sum((2)^k*binomial(m(d)-1,k-1)*(2*m(d)-1-k)!/(2*m(d)-1)!*x^k*
		  sum((n-2*r-x)^(m(d)*n-1+j-k)/(m(d)*n-1+j-k)!*
		  coef(m(d),n,r,j),j,0,(m(d)-1)*n),k,1,m(d)),r,0,floor((n-1)/2));
	E(d,q):=p(q+1,d,1)/p(q-1,d,1)-(q+1)/(q+2*d);
	makelist([2*n,100*E(3,2*n),100*E(5,2*n),100*E(7,2*n)],n,2,14),bfloat;
	\end{lstlisting}
	\vspace{.2cm}

\noindent{\bf Case} $d\in\{2,4,6\}$.
If $d$ is an even integer, then no explicit formula for $\sigma_{d-1}^{\ast (2k)}(0)$ or $\sigma_{d-1}^{\ast (2k-1)}(1)$ is available. 
Instead we use the expressions coming from identifying $\Phi_{d,q}(\one)$ with an integral of a Bessel function.
The Fourier transform of the surface measure $\sigma_{d-1}$, normalized as in \eqref{eq:NormalizationFT},
can be expressed in terms of Bessel functions as follows:
\begin{equation}\label{eq:FourierSurfaceMeasure}
\widehat{\sigma}_{d-1}(x)=(2\pi)^{{d}/{2}}\ab{x}^{-\frac{d-2}2}J_{\frac{d-2}{2}}(\ab{x});
\end{equation}
see \cite[Ch.\@ VII \S 3]{St93} and \cite[Eq.\@ (2.3)]{COSS19}.
For general exponents $q\geq q_d$, we then have that
\begin{equation}\label{eq:BesselexpressionPhi}
\Phi_{d,q}(\one)=\omega_{d-1}^{-{q}/{2}}\norma{\widehat{\sigma}_{d-1}}_{L^q(\R^d)}^q
=\frac{(2\pi)^{{qd}/{2}}}{\omega_{d-1}^{{q}/{2}-1}}\int_{0}^\infty \ab{J_{\frac{d-2}2}(r)}^q\,r^{d-1-q\frac{d-2}{2}}\d r.
\end{equation}
We are thus led to extending definition  \eqref{eq:defEdq} in the following way:
For integers $d\geq 2$ and general $q\geq q_d$, set $\nu=(d-2)/2$, and let
\begin{equation}\label{eq:BesselEdq}
E(d,q):=2^{d-2}\Gamma(\tfrac{d}{2})^2\frac{\int_{0}^\infty \ab{J_\nu(r)}^{q+2}\,r^{d-1-(q+2)\nu}\d r}{\int_{0}^\infty \ab{J_\nu(r)}^q\,r^{d-1-q\nu}\d r}-\frac{1+q}{2d+q-\delta_{d,4}}.
\end{equation}
Naturally, expressions \eqref{eq:defEdq2} and \eqref{eq:BesselEdq} coincide whenever both of them are well-defined.
On the other hand, the integrals in \eqref{eq:BesselEdq} are amenable to rigorous numerical evaluation; see Appendix \ref{sec:numerics}.
The result for even values of $q\in[4,28]\cap2\N$ in dimensions $2\leq d\leq 11$ is displayed in Table \ref{table:2to11Edq}. 
In particular, we see that $E(d,q)> 0$ for every $d\in\{2,4,6\}$ and $q\in[4,28]\cap2\N$, with the exception of $(d,q)=(2,4)$ for which the Tomas--Stein inequality \eqref{eq:TS} does not hold.
This completes the proof of the proposition.
\end{proof}

\begin{table}
\centerline{
		\begin{tabular}{|c||c|c|c|c|c|c|c|c|c|c|c|c|c|c|c|c|c|c|c|}
			\hline 
			\backslashbox{$d$}{$q$}&4  &6  &8  &10  &12  &14  &16  &18  &20  &22  &24  &26  &28  \\ 
			\hline \hline
			3& 0 & 8.33 & 8.63 & 8.29 & 7.85 & 7.40 & 6.97 & 6.57 & 6.20 & 5.87 & 5.56 & 5.29 & 5.03   \\ 
			\hline 
			5& 3.98 & 7.68 & 9.13 & 9.77 & 9.97 & 9.94 & 9.77 & 9.53 & 9.25 & 8.96 & 8.66 & 8.37 & 8.09 \\ 
			\hline 
			7& 0.36 & 4.46 & 7.03 & 8.62 & 9.57 & 10.10 & 10.38 & 10.47 & 10.45 & 10.35 & 10.21  & 10.03 & 9.83  \\ 
			\hline
		\end{tabular}
}
		\caption{Values of $100\times E(d,q)$,  obtained through formula \eqref{eq:BSnthConvo} and truncated to two decimal places.}
		\label{table:357Edq}
\end{table}

\begin{table}
\centerline{
	\begin{tabular}{|c||c|c|c|c|c|c|c|c|c|c|c|c|c|c|c|c|c|c|c|}
		\hline 
		\backslashbox{$d$}{$q$}&4  &6  &8  &10  &12  &14  &16  &18  &20  &22  &24  &26  &28  \\ 
		\hline \hline
		2& $\ast$ &0.57 & 5.16 & 5.32 & 4.95 & 4.55 & 4.19 & 3.88 & 3.61 & 3.37 & 3.17 & 2.98 & 2.82   \\ 
		\hline
		3& 0 & 8.33 & 8.63 & 8.29 & 7.85 & 7.40 & 6.97 & 6.57 & 6.20 & 5.87 & 5.56 & 5.29 & 5.03   \\ 
		\hline 
		4& 0.82 & 4.83 & 5.67 & 5.93 & 5.94 & 5.83 & 5.66 & 5.45 & 5.24 & 5.04 & 4.83 & 4.64 & 4.46  \\ 
		\hline
		5& 3.98 & 7.68 & 9.13 & 9.77 & 9.97 & 9.94 & 9.77 & 9.53 & 9.25 & 8.96 & 8.66 & 8.37 & 8.09  \\ 
		\hline 
		6& 2.26 & 6.16 & 8.24 & 9.38 & 9.97 & 10.22 & 10.27 & 10.18 & 10.03 & 9.82 & 9.59 & 9.34 & 9.09 \\ 
		\hline
		7& 0.36 & 4.46 & 7.03 & 8.62 & 9.57 & 10.10 & 10.38 & 10.47 & 10.45 & 10.35 & 10.21 & 
		10.03 & 9.83 \\ 
		\hline
		8& -1.42 & 2.75 & 5.66 & 7.62 & 8.89 & 9.70 & 10.20 & 10.47 & 10.60 & 10.62 & 10.57 &
		10.47 & 10.34 \\ 
		\hline
		9& -2.98 & 1.11 & 4.25 & 6.49 & 8.04 & 9.10 & 9.81 & 10.26 & 10.54 & 10.69 & 10.74 &
		10.73 & 10.66 \\ 
		\hline
		10& -4.31 & -0.39 & 2.85 & 5.31 & 7.09 & 8.36 & 9.27 & 9.89 & 10.31 & 10.59 & 10.74 & 10.82 & 10.83   \\ 
		\hline
		11& -5.41 & -1.76 & 1.51 & 4.11 & 6.08 & 7.54 & 8.62 & 9.40 & 9.96 & 10.35 & 10.62 & 10.78 & 10.88   \\ 
		\hline
	\end{tabular}
}
	\caption{Values of $100\times E(d,q)$, obtained through Bessel integrals and truncated to two decimal places.}
	\label{table:2to11Edq}
\end{table}

\begin{remark}
	The following integral expression for $\sigma_{d-1}^{\ast n}$ can be obtained via  \eqref{eq:FourierSurfaceMeasure} and Fourier inversion:
	\begin{equation*}
	\sigma_{d-1}^{\ast n}(r)={(2\pi)^{(n-1)\frac{d}{2}}}r^{-\nu}
	\int_0^\infty t^{\frac{d}{2}-n\nu}J_\nu(rt)J_\nu(t)^n\d t,
	\end{equation*}
	where $\nu=(d-2)/2$;
	see also \cite[Eq.\@ (1.9)]{COS15}, \cite[Theorem 2.1]{BSV16}, \cite[\S III]{GP12}, \cite{Pe}, and \cite[Chap XIII, \S 13.48]{Wa44}. 
\end{remark}

\begin{remark}
We emphasize that the integral expression \eqref{eq:BesselEdq} for $E(d,q)$ was used to compute {\it both} the even and the odd dimensional entries of Table \ref{table:2to11Edq}.  Within at least the first 10 significant digits, we observed no difference between the values produced in this way and the ones from Table \ref{table:357Edq}, provided by the explicit formula \eqref{eq:BSnthConvo}.
We also mention that the truncation to 3 significant digits of the numerical value obtained for $E(3,4)$  was $1.77\times 10^{-15}$, which is very accurate since in reality $E(3,4)=0$.
We refer the reader to Appendix \ref{sec:numerics} for further details on the numerical calculations.
\end{remark}

\begin{remark}\label{rem:final71}
Further numerical experimentation 
indicates that the quantity $E(d,4)$ fails to be nonnegative if $d\geq 8$.
Moreover, the initial range of exponents $q$ for which $E(d,q)$ is negative is seen to increase with the dimension $d$, but at some point $q=q_\star(d)$ we observe that $E(d,q)$ becomes positive, and seems to remain positive for all $q\geq q_\star(d)$. 
Such observations regarding inequality \eqref{eq:Phi_inequality} for general $d, q$ will be  studied  at the end of \S \ref{sec:fullAsymp}, where a partial answer is given; see Proposition \ref{cor:neighborhoodInfty} below.
\end{remark}

\subsection{Proof of inequality \eqref{eq:ineqForPhis} for large values of $q$}\label{sec:proofLargeq}

The following result is the missing ingredient to complete the proof of Theorem \ref{thm:MainThm} when $d=3$.

\begin{proposition}\label{prop:d=3}
	For every even integer $q\geq 6$, the following inequality holds:
	\begin{equation}\label{eq:d=3ineq}
	 \Phi_{3,q+2}(\one)>\gamma_3(q)\Phi_{3,q}(\one). 
	 \end{equation}
\end{proposition}

\noindent In Proposition \ref{prop:sharpIneqD} below, we will establish the inequality $\Phi_{d,q+2}(\one)>\gamma_d(q)\Phi_{d,q}(\one)$, for all $d\in\{2,3,4,5,6,7\}$ and $q\geq q_\star(d)$, for a certain $q_\star(d)\in[8,12]\cap 2\N$. In view of \S \ref{sec:proofSmallq}, this includes the case $d=3$ considered in Proposition \ref{prop:d=3}.
However, the proof of Proposition \ref{prop:sharpIneqD} is more involved, and we consider it valuable to include a separate treatment of the case $d=3$ since it is considerably shorter, and sets the  idea for the upcoming proof of the general case.
\begin{proof}[Proof of Proposition \ref{prop:d=3}]
	Let us start by showing that inequality \eqref{eq:d=3ineq} holds, for all real numbers $q\geq 71$. 
	Since $d=3$, identity \eqref{eq:FourierSurfaceMeasure} amounts to
	\[ \widehat{\sigma}_2(r)=(2\pi)^{3/2}r^{-1/2}J_{1/2}(r). \]
	The Bessel functions of half-integer order are elementary; in particular, 
	\[ J_{1/2}(r)=\Bigl(\frac{2}{\pi r}\Bigr)^{1/2}\sin(r). \]
	Expression \eqref{eq:BesselexpressionPhi} then becomes
	\begin{align*}
	\Phi_{3,q}(\one)=\frac{(2\pi)^{3q/2}}{(4\pi)^{q/2-1}}\int_0^\infty \ab{r^{-1/2}J_{1/2}(r)}^q r^2\d r=(2\pi^{1/2})^{q+2}\int_0^\infty \frac{\ab{\sin(r)}^q}{r^{q-2}}\d r.
	\end{align*}
	Recalling that $\gamma_3(q)=4\pi(1+q)/(6+q)$, we obtain
	\begin{align*}
	\Phi_{3,q+2}(\one)-\gamma_3(q)\Phi_{3,q}(\one)=(2\pi^{1/2})^{q+4}\biggl(\int_0^\infty \frac{\ab{\sin(r)}^{q+2}}{r^{q}}\d r-\frac{1+q}{6+q}\int_0^\infty \frac{\ab{\sin(r)}^q}{r^{q-2}}\d r\biggr).
	\end{align*}
	Therefore it suffices to show that
	\begin{equation}\label{eq:claimedIneqQ}
	\int_0^\infty \frac{\ab{\sin(r)}^{q+2}}{r^{q}}\d r>\frac{1+q}{6+q}\int_0^\infty \frac{\ab{\sin(r)}^q}{r^{q-2}}\d r,\text{ for all }q\geq 71.
	\end{equation}
	Asymptotics for integrals of this kind  were studied in \cite{AK17}, thereby generalizing the well-known case of the powers of the $\textup{sinc}$ function, $\textup{sinc}(x)=\sin(x)/x$; see e.g. \cite{Ba86,NP00}.
	We present a slight refinement of \cite[Theorem 2]{AK17} which will be useful for our purposes. 
	The argument in \cite{AK17} starts with the two-sided estimate
	\begin{equation}\label{eq:ineqSinc0Pi}
	1-\frac{r^2}{6}\leq \frac{\sin(r)}{r}\leq e^{-r^2/6},\text{ for all } r\in[0,\pi], 
	\end{equation}
	the left-most inequality being useful only when $r\in[0,\sqrt{6}]$.
	As in the proof of \cite[Theorem 2]{AK17}, for $\alpha>\beta-1>0$ we have that
	\begin{align}
	\int_0^{\pi}\frac{\ab{\sin(r)}^\alpha}{r^{\beta}}\d r&\leq \int_0^\infty\frac{\ab{\sin(r)}^\alpha}{r^{\beta}}\d r=\sum_{k=0}^\infty\int_{k\pi}^{(k+1)\pi}\frac{\ab{\sin(r)}^\alpha}{r^{\beta}}\d r\nonumber\\
	&=\sum_{k=0}^\infty\int_{0}^{\pi}\frac{\ab{\sin(r)}^\alpha}{(r+k\pi)^{\beta}}\d r\nonumber\\
	\label{eq:0PiIneqSinc}
	&\leq \sum_{k=0}^\infty\frac{1}{(k+1)^{\beta}}\int_0^\pi\frac{\ab{\sin(r)}^\alpha}{r^{\beta}}\d r.
	\end{align}
	Here, besides the periodicity of the sine function, we used  the fact that, for $r\in[0,\pi]$, $r+k\pi\geq r+kr=(k+1)r$, for all $k\geq 0$. It is then proved via \eqref{eq:ineqSinc0Pi} that
	\begin{equation}\label{eq:IneqInt0PiSinc}
	\begin{split} \frac{1}{2}6^{\frac{\alpha-\beta+1}{2}}\frac{\Gamma(\frac{\alpha-\beta+1}{2})\Gamma(\alpha+1)}{\Gamma(\frac{\alpha-\beta+1}{2}+\alpha+1)}&\leq \int_0^{\sqrt{6}}\frac{\ab{\sin(r)}^\alpha}{r^\beta}\d r\\
	&\leq \int_0^\pi\frac{\ab{\sin(r)}^\alpha}{r^\beta}\d r\leq \frac{1}{2}\Bigl(\frac{6}{\alpha}\Bigr)^{\frac{\alpha-\beta+1}{2}}\Gamma\Bigl(\frac{\alpha-\beta+1}{2}\Bigr).
	\end{split}
	\end{equation}
	Combining \eqref{eq:0PiIneqSinc} and \eqref{eq:IneqInt0PiSinc}, we obtain the following two-sided estimate:
	\begin{multline}\label{eq:twoSidedEstimate}
	\frac{1}{2}6^{\frac{\alpha-\beta+1}{2}}\frac{\Gamma(\frac{\alpha-\beta+1}{2})\Gamma(\alpha+1)}{\Gamma(\frac{\alpha-\beta+1}{2}+\alpha+1)}\leq \int_0^\infty\frac{\ab{\sin(r)}^\alpha}{r^\beta}\d r\\
	\leq\frac{1}{2}\Bigl(\frac{6}{\alpha}\Bigr)^{\frac{\alpha-\beta+1}{2}}\Gamma\Bigl(\frac{\alpha-\beta+1}{2}\Bigr)\sum_{k=0}^\infty\frac{1}{(k+1)^{\beta}}.
	\end{multline}
	We use the latter estimate with $\alpha=q$ and $\beta=q-2$, 
	i.e.
	\begin{equation}\label{eq:twoSidedQ}
	\frac{1}{2}6^{\frac{3}{2}}\frac{\Gamma(\frac{3}{2})\Gamma(q+1)}{\Gamma(q+\frac{5}{2})}\leq \int_0^\infty\frac{\ab{\sin(r)}^q}{r^{q-2}}\d r\\
	\leq\frac{1}{2}\Bigl(\frac{6}{q}\Bigr)^{\frac{3}{2}}\Gamma\Bigl(\frac{3}{2}\Bigr)\sum_{k=0}^\infty\frac{1}{(k+1)^{q-2}}.
	\end{equation}
	Using the upper bound from \eqref{eq:twoSidedQ} for the exponent $q$ together with the lower bound from \eqref{eq:twoSidedQ} for the exponent $q+2$, the desired inequality \eqref{eq:claimedIneqQ} would follow from
	\begin{equation}\label{eq:claimedIneqQ2}
	\frac{\Gamma(q+3)}{\Gamma(q+\frac{9}{2})}>\frac{q+1}{q+6}\frac{1}{q^{3/2}}\sum_{k=0}^\infty\frac{1}{(k+1)^{q-2}},
	\end{equation}
	which we now prove.
	Consider the upper bound
	\begin{equation}\label{eq:upperboundN}
	\sum_{k=1}^\infty\frac{1}{k^{q-2}}\leq \sum_{k=1}^N\frac{1}{k^{q-2}}+\int_{N}^\infty\frac{1}{r^{q-2}}\d r=\sum_{k=1}^N\frac{1}{k^{q-2}}+\frac{1}{(q-3)N^{q-3}},
	\end{equation}
	which is valid for any $N\geq 1$. If $N=2$, then \eqref{eq:upperboundN} amounts to 
	\[ \sum_{k=1}^\infty\frac{1}{k^{q-2}}\leq 1+ \frac{1}{2^{q-2}}+\frac{1}{(q-3)2^{q-3}}=1+\frac{q-1}{(q-3)2^{q-2}}.\]
 	In order to prove \eqref{eq:claimedIneqQ2}, it will thus suffice to establish the inequality
	\begin{equation}\label{eq:claimedIneqQ3}
	\frac{\Gamma(q+3)}{\Gamma(q+\frac{9}{2})}>\frac{q+1}{(q+6)q^{3/2}}\biggl(1+\frac{q-1}{(q-3)2^{q-2}}\biggr),\text{ for all }q\geq 71.
	\end{equation}
Several simple and effective lower bounds for the ratio of  two Gamma functions are available; see e.g. 
\cite{Ke83,La84}. 
From \cite[Eq.\@ (1.3)]{Ke83}, we have that	
\[ \frac{\Gamma(x+1)}{\Gamma(x+s)}>\Bigl(x+\frac{s}{2}\Bigr)^{1-s},\]
provided $x>0$ and $0<s<1$.
	With $x=q+2$ and $s=1/2$, we obtain
	\[ \frac{\Gamma(q+3)}{\Gamma(q+\frac{5}{2})}>\Bigl(q+\frac{9}{4}\Bigr)^{1/2}. \]
	Since $\Gamma(q+9/2)=(q+7/2)(q+5/2)\Gamma(q+5/2)$, inequality \eqref{eq:claimedIneqQ3} is seen to follow from 
	\begin{equation}\label{eq:claimedIneqQ4}
	\frac{(q+\frac{9}{4})^{1/2}}{(q+\frac{7}{2})(q+\frac{5}{2})}>\frac{q+1}{(q+6)q^{3/2}}\biggl(1+\frac{(q-1)2^{-(q-2)}}{q-3}\biggr),\text{ for all }q\geq 71;
	\end{equation}
	see Figure \ref{fig:d=3} below.
	Cross multiplying, we find the following equivalent form: 
	\begin{equation}\label{eq:claimedIneqQ5}
	q^{\frac32}(q+\tfrac{9}{4})^{\frac12}(q+6)>(q+1)(q+\tfrac{5}{2})(q+\tfrac{7}{2})\biggl(1+\frac{(q-1)2^{-(q-2)}}{q-3}\biggr),\text{ for all }q\geq 71.
	\end{equation}
	One easily checks that inequality \eqref{eq:claimedIneqQ5} holds for every sufficiently large $q$. Indeed, defining the auxiliary functions
	\[ f(q):=q^{\frac32}(q+\tfrac{9}{4})^{\frac12}(q+6),\text{ and } \, g(q):=(q+1)(q+\tfrac{5}{2})(q+\tfrac{7}{2})\biggl(1+\frac{(q-1)2^{-(q-2)}}{q-3}\biggr), \]
	we have that
	\[ \frac{f(q)}{q^{3}}=(1+\tfrac{9}{4q})^{\frac12}(1+\tfrac{6}{q}),\text{ and }\, \frac{g(q)}{q^{3}}=(1+\tfrac{1}{q})(1+\tfrac{5}{2q})(1+\tfrac{7}{2q})\biggl(1+\frac{(q-1)2^{-(q-2)}}{q-3}\biggr). \]
	Raising to power $q$ and taking the limit as $q\to\infty$, we find that 
	\[\lim_{q\to\infty}\Big(\frac{f(q)}{q^{3}}\Big)^q= e^{\frac{57}8}, \text{ and } \lim_{q\to\infty}\Big(\frac{g(q)}{q^{3}}\Big)^q=e^{7}.\] 
	Since $\frac{57}8>7$, we conclude that $f(q)>g(q)$, for all sufficiently large $q$. In order to obtain the more precise form of the inequality $f(q)>g(q)$, valid in the desired range $q\geq 71$, we rewrite \eqref{eq:claimedIneqQ5} in yet another equivalent form by squaring. We aim to show that
	\begin{multline}\label{eq:claimedIneqQ6}
	q^{3}(q+\tfrac{9}{4})(q+6)^2-(q+1)^2(q+\tfrac{5}{2})^2(q+\tfrac{7}{2})^2\\
	>(q+1)^2(q+\tfrac{5}{2})^2(q+\tfrac{7}{2})^2\frac{(q-1)2^{-(q-2)}}{q-3}\biggl(2+\frac{(q-1)2^{-(q-2)}}{q-3}\biggr),	
	\end{multline}
	 for every $q\geq 71$.
	Let $P(q)$ and $Q(q)$ denote the functions on the left- and right-hand sides of \eqref{eq:claimedIneqQ6}, respectively. The polynomial $P$ simplifies to
	\[ P(q)=\frac{1}{16}(4q^5-248q^4-2288q^3-5441q^2-4130q-1225). \]
	We first establish a lower bound for $P$ on the interval $[71,\infty)$. First of all,
	\begin{align*}
	16P'(q)&=20q^4-992q^3-6864q^2-10882q-4130,\\
	16P''(q)&=80q^3-2976q^2-13728q-10882,\\
	16P'''(q)&=240q^2-5952q-13728.
	\end{align*}
	One easily checks that $P'''(q)>0$, for all $q\geq 30$, and $P''(50)>0$, so that $P''(q)>0$, for all $q>50$. On the other hand, $P'(60)>0$, so that $P'(q)>0$, for all $q>60$. Finally, $P(71)>0$, so that $P(q)\geq P(71)=17049403/4$, for all $q\geq 71$. This is the desired lower bound. We proceed to find an upper bound for $Q$ on the interval $[71,\infty)$. If  (say) $q\geq 10$, then
	\begin{align*}
	Q(q)=(q+1)^2(q+\tfrac{5}{2})^2(q+\tfrac{7}{2})^2\frac{(q-1)2^{-(q-2)}}{q-3}\biggl(2+\frac{(q-1)2^{-(q-2)}}{q-3}\biggr)\leq \frac{6(2q)^6}{2^{q-2}}=\frac{3q^6}{2^{q-9}}.
	\end{align*}
	The function $\widetilde Q(q)=q^6 2^{-q}$ is easily seen to be decreasing for $q\geq 6/\log(2)\sim 8.6$, since $\widetilde{Q}'(q)=q^5(6-q\log(2))2^{-q}$. It follows that $Q(q)\leq 3\cdot 2^9\widetilde{Q}(71)\leq 10^{-7}$, for all $q\geq 71$. As a conclusion, $P(q)>Q(q)$, for all $q\geq 71$, which is what we wanted to prove.
	
	The remaining cases $q\in[6,70]\cap2\N$ are verified as in the course of the proof of Proposition \ref{lem:verifiedIneq}, the only modification to the code presented there occurring in the last line:
	\begin{lstlisting}
	makelist(E(3,2*n),n,2,35), bfloat;
	\end{lstlisting}
	In this way, we obtained the following sample values:
	\begin{align*}
	E(3,30)&=\tfrac{2079536996068415060993}{43278269405974205935140}\sim 0.0480,\\
	E(3,50)&=\tfrac{661588152841131630707817424352070248343232410691}{20232905795389421109643117078167417269636270235000}\sim 0.0326,\\
	E(3,70)&=\tfrac{227847339920288097615743574709286092473495402056835456260649730379156981583371563}{9224623180989915110488440015966247230150898847019509922030740261030681342888853700}\\
	&\sim 0.0246,
	\end{align*}
	and in any case $E(3,q)>0$, for every $q\in[6,70]\cap2\N$.
	The proof of the proposition is now complete.
\end{proof}

\begin{figure}
\centering
\includegraphics[width=.45\linewidth]{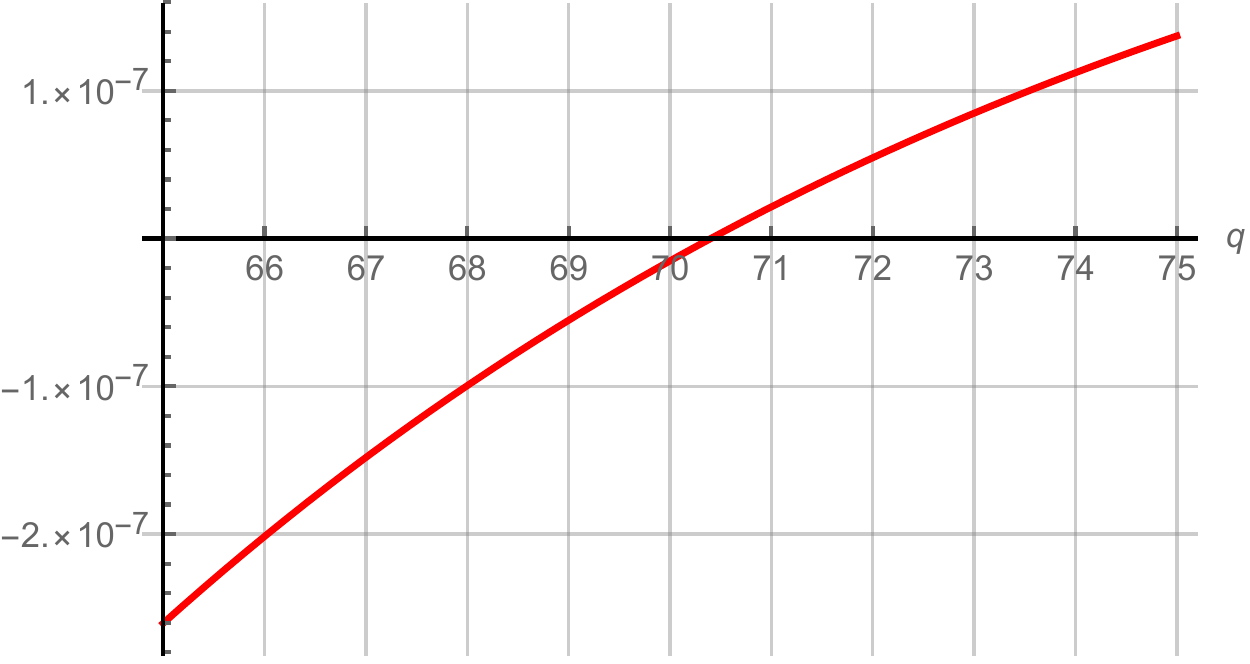}
\includegraphics[width=.45\linewidth]{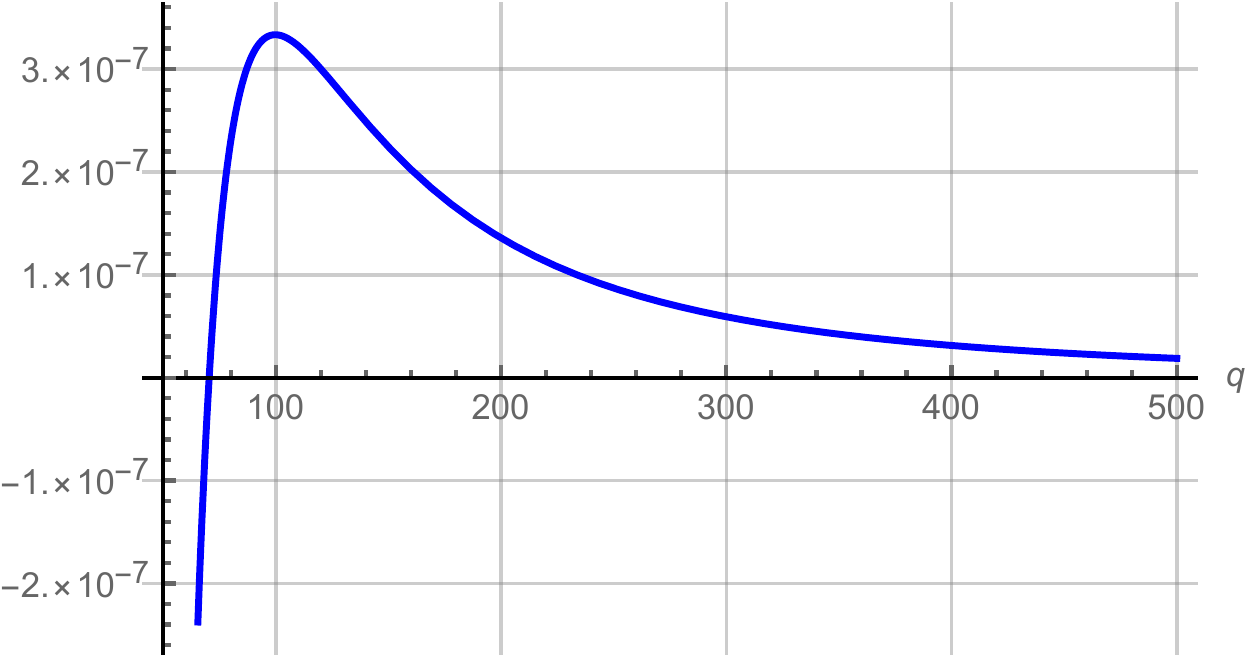}
\caption{Plot of the function considered in inequality \eqref{eq:claimedIneqQ4},
$F(q)=\frac{(q+\frac{9}{4})^{1/2}}{(q+\frac{7}{2})(q+\frac{5}{2})}-\frac{q+1}{(q+6)q^{3/2}}\left(1+\frac{(q-1)2^{-(q-2)}}{q-3}\right)$, 
 in the ranges $q\in[65,75]$ and $q\in[50,500]$.}
\label{fig:d=3}
\end{figure}

Although a two-sided inequality similar to \eqref{eq:ineqSinc0Pi} holds for the function $r^{-\nu}J_{\nu}(r)$, for every $\nu\geq 0$, the same strategy of the proof of Proposition \ref{prop:d=3} does not seem to lead to a proof of an analogous statement for the remaining values of $d\in \{2,4,5,6,7\}$. Therefore, we refine our analysis in order to prove the following result.
\begin{proposition}\label{prop:sharpIneqD}
	Let $d\in\{2,3,4,5,6,7\}$. Then there exists $q_\star(d)<\infty$ such that, for every $q\geq q_\star(d)$, the following inequality holds: 
	\begin{equation}\label{eq:PhiPhi}
	\Phi_{d,q+2}(\one)>\gamma_d(q)\Phi_{d,q}(\one). 
	\end{equation}
	Moreover, we can take 
	\begin{equation}\label{eq:valueQ0}
	q_\star(2)=12,\text{ and }\,q_\star(d)=8 \text{ if } d\in\{3,4,5,6,7\}.
	\end{equation}
\end{proposition}

Before diving into the proof of Proposition \ref{prop:sharpIneqD},
 we discuss some preliminaries. 
 For $\nu\geq 0$, we denote the {\it normalized} Bessel function of the first kind by
\[ \J_\nu(r)=2^{\nu}\Gamma(\nu+1)r^{-\nu}J_{\nu}(r). \]
It is well-known that $\J_\nu(0)=1$, and that $\J_\nu(r)\leq 1$, for every $r\geq 0$; $\J_{1/2}(r)=\text{sinc}(r)$ corresponds to the case considered in Proposition \ref{prop:d=3}. The function $\J_\nu$ is analytic in the whole complex plane, and admits the following power series representation:
\[ \J_\nu(r)=\sum_{j=0}^\infty(-1)^j\frac{\Gamma(\nu+1)(\frac{1}{4}r^2)^j}{j!\Gamma(\nu+j+1)}. \]
Let $T_{\nu,k}$  denote the $k$-th order partial sum of the power series for $\J_\nu$,
\begin{equation}\label{eq:defTk}
T_{\nu,k}(r):=\sum_{j=0}^k(-1)^j\frac{\Gamma(\nu+1)(\frac{1}{4}r^2)^j}{j!\Gamma(\nu+j+1)}.
\end{equation}
If no danger of confusion exists, we simply write $T_k=T_{\nu,k}$. 
Note that $T_{\nu,k}(0)=\J_\nu(0)=1$, for all $\nu, k\geq 0$.

If $\nu\geq 0$, then the function $\J_\nu$ has only real zeros. The smallest positive zero of $\J_\nu$ is known to be simple, and will be denoted $j_{\nu,1}$. Moreover, the function $\nu\mapsto j_{\nu,1}$ is strictly increasing for $\nu\geq 0$; see \cite[\S 10.21]{DLMF} and the references therein. With 15 decimal places, we tabulate the relevant values of $j_{\nu,1}$:
\begin{equation}\label{eq:Tabulatejnu1}
	\begin{split}
	j_{0,1}\sim 2.404825557695772,\; j_{\frac{1}{2},1}\sim 3.141592653589793,\; j_{1,1}\sim 3.831705970207512,\\
	j_{\frac{3}{2},1}\sim 4.493409457909064,\; j_{2,1}\sim 5.135622301840682,\; j_{\frac{5}{2},1}\sim 5.763459196894549.
	\end{split}
\end{equation}
	
Our next result concerns the pointwise relation between $\J_\nu$ and $T_{\nu,k}$ in the interval $[0,j_{\nu,1}]$, and is probably well-known. Since we could not locate a reference,  we include a proof for completeness.

\begin{lemma}\label{lem:comparisonJT}
	Let $\nu\geq 0$ and $k\in \N_0$. Then, for all $r\in(0,j_{\nu,1}]$, we have the strict inequalities
	\[ T_{\nu,2k+1}(r)< \J_{\nu}(r)< T_{\nu,2k}(r). \]
\end{lemma}

\begin{proof}
	Let $k\geq 0$ be given, and denote $T_k=T_{\nu,k}$. Note that $r^{\nu}T_k$ satisfies the inhomogeneous Bessel equation
	\[ r^2y''+ry'+(r^2-\nu^2)y=\frac{(-1)^k2^{-2k}\Gamma(\nu+1)}{k!\,\Gamma(\nu+k+1)}r^{2k+\nu+2},\quad r\geq 0. \]
	Let $\phi_{1,k}=r^\nu (T_k-\J_\nu)$. We want to show that $\phi_{1,k}>0$ on $(0,j_{\nu,1}]$ if $k$ is even, and that $\phi_{1,k}<0$ on $(0,j_{\nu,1}]$ if $k$ is odd. For brevity, we will write $\phi_1=\phi_{1,k}$. The function $\vphi_1(r):=r^{-\nu}\phi_1(\sqrt{4r})$ satisfies $\vphi_1(0)=\vphi_1'(0)=\dots=\vphi_1^{(k)}(0)=0$, while $\vphi_1^{(k+1)}(0)=-\frac{(-1)^{k+1}\Gamma(\nu+1)}{\Gamma(\nu+k+1)}$, so that $\vphi_1^{(k+1)}(0)>0$ if $k$ is even, while $\vphi_1^{(k+1)}(0)<0$ if $k$ is odd. Therefore, if we let $r_1>0$ denote the smallest positive zero of $\phi_1$, then $\phi_1(r)>0$ on $(0,r_1)$ if $k$ is even, and $\phi_1(r)<0$ on $(0,r_1)$ if $k$ is odd. If  no such $r_1$ exists, then the conclusion of the lemma follows immediately; therefore we assume the existence of $r_1$.
	
	We seek to find the sign of $\phi_1$ on $[0,j_{\nu,1}]$.
	With this purpose in mind, we study the relationship between $r_1$ and $j_{\nu,1}$. The function $\phi_1$ satisfies the equation
	\[ r^2\phi_1''+r\phi_1'+(r^2-\nu^2)\phi_1=\frac{(-1)^k2^{-2k}\Gamma(\nu+1)}{k!\,\Gamma(\nu+k+1)}r^{2k+\nu+2}. \]
	Let $\phi_2=2^{\nu}\Gamma(\nu+1)J_\nu$, which in turn satisfies the equation
	\[ r^2\phi_2''+r\phi_2'+(r^2-\nu^2)\phi_2=0. \]
	We aim to show that $r_1> j_{\nu,1}$, and will do so via techniques from Sturm--Liouville theory. 
	Aiming at a contradiction, assume that $r_1\leq j_{\nu,1}$.
	Rewrite the equations satisfied by $\phi_1,\phi_2$ by introducing the auxiliary functions $u_1=r^{1/2}\phi_1$, $u_2=r^{1/2}\phi_2$. Then, for $r>0$,
	\begin{align}
	\label{eq:eqforu1}
	u_1''+\biggl(1+\frac{1-4\nu^2}{4r^2}\biggr)u_1&=\frac{(-1)^k2^{-2k}\Gamma(\nu+1)}{k!\,\Gamma(\nu+k+1)}r^{2k+\nu+\frac{1}{2}},\\
	\label{eq:eqforu2}
	u_2''+\biggl(1+\frac{1-4\nu^2}{4r^2}\biggr)u_2&=0.
	\end{align}
	Multiply \eqref{eq:eqforu1} by $u_2$, multiply \eqref{eq:eqforu2} by $u_1$, and subtract. 
	Then, for $r>0$,
	\begin{equation}\label{eq:deriveU1U2}
	u_1''u_2-u_1u_2''=\frac{(-1)^k2^{-2k}\Gamma(\nu+1)}{k!\,\Gamma(\nu+k+1)}r^{2k+\nu+\frac{1}{2}}u_2.
	\end{equation}
	Let $r\in(0,r_1)$. Since $u_1''u_2-u_1u_2''=(u_1'u_2-u_1u_2')'$, integrating identity \eqref{eq:deriveU1U2} on the interval $(r,r_1)$ reveals that
	\[ (u_1'u_2-u_1u_2')\big\vert_{r}^{r_1}=\frac{(-1)^k2^{-2k}\Gamma(\nu+1)}{k!\,\Gamma(\nu+k+1)}\int_{r}^{r_1}r^{2k+\nu+\frac{1}{2}}u_2(r)\d r.\]
	Observing that $\phi_1(0)=0$ implies $u_1(r)u_2'(r)=r^{1/2}\phi_1(r)(2^{-1}r^{-1/2}\phi_2(r)+r^{1/2}\phi_2'(r))\to 0$   and $u_1'(r)u_2(r)=r^{1/2}\phi_2(r)(2^{-1}r^{-1/2}\phi_1(r)+r^{1/2}\phi_1'(r))\to0$, as $r\to0^+$, 
	we obtain
	\[ (u_1'u_2-u_1u_2')(r_1)=\frac{(-1)^k2^{-2k}\Gamma(\nu+1)}{k!\,\Gamma(\nu+k+1)}\int_{0}^{r_1}r^{2k+\nu+\frac{1}{2}}u_2(r)\d r. \]
	Since $u_1(r_1)=0$ and $0<r_1\leq j_{\nu,1}$, it follows that 
	\[ u_1'(r_1)u_2(r_1)=\frac{(-1)^k2^{-2k}\Gamma(\nu+1)}{k!\,\Gamma(\nu+k+1)}\int_{0}^{r_1}r^{2k+\nu+\frac{1}{2}}u_2(r)\d r, \quad\int_{0}^{r_1}r^{2k+\nu+\frac{1}{2}}u_2(r)\d r>0. \]
	In particular, $r_1<j_{\nu,1}$, and so $u_2(r_1)>0$. Then, if $k$ is even, we conclude $u_1'(r_1)>0$, which is a contradiction since $u_1(r)>0$ in $(0,r_1)$, so as $u_1(r_1)=0$ we must have $u_1'(r_1)\leq 0$. 
	Similarly, if $k$ is odd, we conclude $u_1'(r_1)<0$, which is again a contradiction since $u_1< 0$ on $(0,r_1)$ and $u_1(r_1)=0$. We conclude that $r_1>j_{\nu,1}$, and this finishes the proof of the lemma.
\end{proof}

The absolute value of the coefficients in the polynomial expansion of $T_{\nu,k}$ and $\J_\nu$, 
\[\alpha_j:=\frac{\Gamma(\nu+1)(\frac{1}{4}r^2)^j}{j!\Gamma(\nu+j+1)}, \;\;\; j\in\{0,\dots,k\},\] 
satisfy $\alpha_{j+1}< \alpha_{j}$, for all $j\geq j_0$, provided $j_0$ is such that
\[ (j_0+1)(\nu+j_0+1)> \frac{r^2}{4}. \]
On the interval $[0,j_{\nu,1}]$, the latter condition is ensured uniformly in $r\in[0,j_{\nu,1}]$ if
\begin{equation}\label{eq:alternatingCutOff}
j_0> \frac{1}{2}\Bigl(\sqrt{\nu^2+j_{\nu,1}^2}-(\nu+2)\Bigr)=:\mathfrak{c}(j_{\nu,1}),
\end{equation}
which in particular holds if $j_0\geq j_{\nu,1}/2-1$, $\nu\geq 0$. 
It then follows from Lemma \ref{lem:comparisonJT} and the alternating nature of the expression \eqref{eq:defTk} that, if $k$ is even and $k\geq \lfloor \mathfrak{c}(j_{\nu,1})\rfloor$, then
\begin{equation}\label{eq:decreasingT}
T_{\nu,k}(r)>T_{\nu,k+2}(r)>T_{\nu,k+4}(r)>\dots> \J_\nu(r),\text{ for all }r\in(0,j_{\nu,1}], 
\end{equation}
and that, if $k$ is odd and $k\geq \lfloor \mathfrak{c}(j_{\nu,1})\rfloor$, then
\begin{equation}\label{eq:increasingT}
T_{\nu,k}(r)<T_{\nu,k+2}(r)<T_{\nu,k+4}(r)<\dots< \J_\nu(r),\text{ for all }r\in(0,j_{\nu,1}]. 
\end{equation}
In particular, using the values in \eqref{eq:Tabulatejnu1} 
and \eqref{eq:alternatingCutOff}, we find that if $\nu\in \{0,\frac{1}{2},1,\frac{3}{2},2,\frac{5}{2}\}$, then \eqref{eq:decreasingT} holds if $k\geq 2$ is even, while \eqref{eq:increasingT} holds if $k\geq 1$ is odd.

The normalized Bessel functions $\J_\nu$ are dominated by a Gaussian on the interval $[0,j_{\nu,1}]$.
A precise statement from \cite[Eq.\@ (1.7)]{IS90} is as follows:
\begin{equation}\label{eq:GaussianAboveJ}
0\leq \J_\nu(r)\leq \exp\Bigl(-\frac{r^2}{4(\nu+1)}-\frac{r^4}{32(\nu+1)^2(\nu+2)}\Bigr),\; r\in[0,j_{\nu,1}].
\end{equation}
Only the quadratic part of the latter exponential function will be of relevance to us, in which case a simpler proof can be obtained by combining \cite[Prop.\@ 4]{AK17}, \cite[Eq.\@ (40)]{Sn60} (see also \cite[Eq.\@ (2)]{dL14}) 
with the product formula for the Bessel function $J_\nu$ \cite[Eq.\@ 10.21.15]{DLMF}; an even simpler proof follows \cite[Lemma 3.6]{Br13} and \cite[Prop.\@ 12]{KK01}.
We will be interested in the following stronger estimate: For all $\nu\geq 0$, there exists $k_0\geq 2$, such that, for all $k\geq k_0$, we have that
\begin{equation}\label{eq:GaussianAboveTk}
T_{\nu,k}(r)\leq \exp\Bigl(-\frac{r^2}{4(\nu+1)}\Bigr),\; r\in[0,j_{\nu,1}].
\end{equation}
That \eqref{eq:GaussianAboveTk} holds for all sufficiently large $k$ follows from estimates \eqref{eq:decreasingT}, \eqref{eq:GaussianAboveJ}, together with the facts that inequality \eqref{eq:GaussianAboveTk} holds for every $k\geq 2$, for all sufficiently small $r$, and that $T_{\nu,k}$ converges uniformly to $\J_\nu$ in the interval $[0,j_{\nu,1}]$, as $k\to\infty$. That $k_0=3$ is a valid choice when $\nu\in\{0,\frac{1}{2},1,\frac{3}{2},2,\frac{5}{2}\}$ is the content of the following result. 

\begin{lemma}\label{lem:GoodTk}
	Let $d\in\{2,3,4,5,6,7\}$ and $\nu=(d-2)/2$. Then
	\begin{equation}\label{eq:T4Good}
	T_{\nu,4}(r)\leq \exp\Bigl(-\frac{r^2}{4(\nu+1)}\Bigr), \text{ for all } r\in[0,j_{\nu,1}].
	\end{equation}
	Moreover, $T_{\nu,n}\leq T_{\nu,4}$ in $[0,j_{\nu,1}]$, for all $n\geq 3$.
\end{lemma}
\begin{proof}
	The last assertion is a consequence of the comments following \eqref{eq:increasingT}, and so we focus on the upper bound \eqref{eq:T4Good}.
	Let us then show that the function 
	\[f_\nu(r):=\exp\Bigl(\frac{r^2}{4(\nu+1)}\Bigr)T_{\nu, 4}(r)\] 
	satisfies the upper bound $f_\nu(r)\leq 1$, for every $r\in [0,j_{\nu,1}]$.
	First of all, direct computations reveal that $f_\nu(0)=1$, whereas $f_\nu'(0)=f_\nu''(0)=f_\nu'''(0)=0$ and $f_\nu^{(4)}(0)<0$.
	Another straightforward computation reveals that
	$$e^{-\frac{r^2}{4(\nu+1)}}f_\nu'(r)=\frac{r^3}{1+\nu}(a_0+a_2 r^2+a_4 r^4+a_6 r^6),$$
	where the coefficients $a_{2j}=a_{2j}(\nu)$, $0\leq j\leq 3$, are given by
	\begin{gather*}
	a_0=-\frac1{8(1+\nu)(2+\nu)},\, a_2=\frac{1}{32(1+\nu)(2+\nu)(3+\nu)},\\
	a_4=-\frac{1}{256(1+\nu)(2+\nu)(3+\nu)(4+\nu)},\, a_6=\frac{1}{12288(1+\nu)(2+\nu)(3+\nu)(4+\nu)}.
	\end{gather*}
	Let 
	$g_\nu(s):=a_0+a_2 s+a_4 s^2+a_6 s^3$,
	so that
	\[ e^{-\frac{r^2}{4(\nu+1)}}f_\nu'(r)=\frac{r^3}{1+\nu}g_\nu(r^2),\]
	and note that $g_\nu'(s)>0$, for every $s\in\R$. To check this claim, observe that $g_\nu'(s)=a_2+2a_4s+3a_6s^2$ is such that $g_\nu'(0)>0$ and $g_\nu'$ has no real zeros, since (half of) its discriminant satisfies 
	\[a_4^2-3a_2a_6=-\frac{1}{131072(1+\nu)^2(2+\nu)(3+\nu)^2(4+\nu)^2}<0.\]
	We now split the analysis into the cases $\nu\in\{0,\frac12,1,\frac32, 2\}$ and $\nu=\frac{5}{2}$. In the former case, it will suffice to show that 
	\begin{equation}\label{eq:ETCf'neg}
	f_\nu'(r)<0, \text{ for every } r\in(0,j_{\nu,1}).
	\end{equation}
	Inequality \eqref{eq:ETCf'neg} will follow if we check that 
	$g_\nu(s)<0$, for every $s\in(0,j^2_{\nu,1})$.
	In turn, since $g_\nu'(s)>0$, for every $s\in\R$, this will follow from $g_\nu(j^2_{\nu,1})<0$,
	which can be checked directly for each $\nu\in\{0,\frac12,1,\frac32,2\}$.
	
	Let us now consider $\nu=\frac{5}{2}$. Since \[g_{5/2}(j^2_{5/2,1})>0,\] the same argument does not apply to $d=7$.
	It is however still true that $f_{5/2}(r)\leq 1$, for every $r\in[0,j_{5/2,1}]$, and this can be checked as follows.
	Since $g_{5/2}'>0$ in $\R$, the function $g_{5/2}$ is strictly increasing. On the other hand, $g_{5/2}(0)<0$.
	Therefore, $g_{5/2}$ can have at most one zero in $\R$ and, in particular, 
	$f_{5/2}'$ can have at most one zero in the interval $(0,j_{5/2,1})$.
	We lose no generality in assuming that such a zero $z\in (0,j_{5/2,1})$ exists, for otherwise the proof is the same as in the case $d\leq 6$.
	But $f_{5/2}'(r)$ and $g_{5/2}(r^2)$ have the same sign on the positive half-line $(0,\infty)$, and so $f_{5/2}$ must be decreasing on $(0,z)$ and increasing on $(z,j_{5/2,1})$. Therefore $z$ is a local minimum of $f_{5/2}$, and matters reduce to the verification of the numerical inequality $f_{5/2}(j_{5/2,1})<f_{5/2}(0)=1$. But $f_{5/2}(j_{5/2,1})\sim 0.715$, and we are done.
\end{proof}

We set some useful notation. For $m\in\N_0$, let
\begin{align}
\label{eq:GaussianPSLow}
P_{\nu,m}(r)&=\sum_{j=0}^m\frac{(\frac{1}{4}r^2)^j}{j!\,(\nu+1)^j},\\
\label{eq:GaussianPSHigh} E_{\nu,m}(r)&=\sum_{j=m+1}^\infty\frac{(\frac{1}{4}r^2)^j}{j!\,(\nu+1)^j}=\exp\left({\frac{r^2}{4(\nu+1)}}\right)-P_{\nu,m}(r),
\end{align}
so that $P_{\nu,m}$ corresponds to the $m$-th order partial sum of the Taylor series of $\exp({\frac{r^2}{4(\nu+1)}})$ at $r=0$. 
The tail $E_{\nu,m}\geq 0$ satisfies the following upper bound: 
\begin{equation}\label{eq:GaussianErrorBound}
E_{\nu,m}(r)\leq \frac{1}{(m+1)!}\left(\frac{r^2}{4(\nu+1)}\right)^{m+1}\exp\biggl(\frac{j_{\nu,1}^2}{4(\nu+1)}\biggr),\text{ for all }r\in[0,j_{\nu,1}].
\end{equation}
Whenever no risk of confusion arises, we shall denote $P_m=P_{\nu,m}$ and $E_m=E_{\nu,m}$.
Further let
\begin{equation}\label{eq:DefInu}
 I_\nu(q):=q^{\nu+1}\int_0^\infty \ab{\J_\nu(r)}^q r^{2\nu+1}\d r. 
 \end{equation}
The somewhat long proof of Proposition \ref{prop:sharpIneqD} will partially follow the outline of 
(the arXiv version of) 
\cite[\S 2 and \S 3]{KOS15}. 

\begin{proof}[Proof of Proposition \ref{prop:sharpIneqD}]
	Given $d\in\{2,3,4,5,6,7\}$, let $q\geq q_d$, and set $\nu:=(d-2)/2$. 
	Let $T_k=T_{\nu,k}$, $P_m=P_{\nu,m}$, $E_{m}=E_{\nu,m}$ be as  in \eqref{eq:defTk}, \eqref{eq:GaussianPSLow}, \eqref{eq:GaussianPSHigh}, respectively. 
	From \eqref{eq:BesselexpressionPhi}, it follows that $\omega_{d-1}^{1+q/2} I_\nu(q)=q^{\nu+1}\Phi_{d,q}({\bf 1})$,
	and so inequality \eqref{eq:PhiPhi} holds if and only if
	\begin{equation}\label{eq:ETCheck}
	\frac{I_\nu(q+2)}{(q+2)^{\nu+1}}>\frac{q+1}{q+4(\nu+1)-\delta_{\nu,1}}\frac{I_\nu(q)}{q^{\nu+1}}.
	\end{equation}
	It then becomes natural to look for effective lower and upper bounds for the quantity $I_\nu(q)$.
	In the spirit of \cite{KOS15, Pe}, we would like to obtain an asymptotic expansion of the form
	\[ I_\nu(q)=\alpha_0+\frac{\alpha_1}{q}+\frac{\alpha_2}{q^2}+\dots+\frac{\alpha_n}{q^n}+O\biggl(\frac{1}{q^{n+1}}\biggr), \]
	in the sense that there exist finite constants $\{e_{n}\}_{n=1}^\infty$, independent of $q$, such that
	\[ \abs{I_\nu(q)-\Bigl(\alpha_0+\frac{\alpha_1}{q}+\frac{\alpha_2}{q^2}+\dots+\frac{\alpha_n}{q^n}\Bigr)}\leq \frac{e_{n+1}}{q^{n+1}}, \]
	for all $n\in\N_0$ and $q\geq q_\star$, provided $q_\star$ is sufficiently large.
	Furthermore, we will need precise bounds for the error term $e_{n+1}$ and for the threshold $q_\star$. 
	For the purpose of the present proof, it will be enough to aim at $n=1$, but in \S \ref{sec:fullAsymp} below we comment on the necessary changes in order to obtain the full asymptotic expansion.
	
	Let us start by analyzing the effect of replacing $I_\nu(q)$ by the corresponding integral over the bounded interval $[0,j_{\nu,1}]$.
	Invoking Landau's estimate \cite{La00}, 
	\begin{equation}\label{eq:Landau}
		 \ab{J_\nu(r)}\leq c_\star r^{-1/3},\, r>0, \nu>0,\, c_\star:=0.7857468704\dots,
	\end{equation}
	the following tail estimate holds:
	\begin{align}
	q^{\nu+1}\int_{j_{\nu,1}}^\infty\ab{\J_\nu(r)}^qr^{2\nu+1}\d r
	&\leq q^{\nu+1} (2^\nu\Gamma(\nu+1)c_\star)^q\int_{j_{\nu,1}}^\infty r^{2\nu+1-q(\nu+1/3)}\d r\label{eq:preupperbound}\\
	&=\frac{q^{\nu}j_{\nu,1}^{2(\nu+1)}}{\bigl(\nu+\frac{1}{3}-\frac{2(\nu+1)}{q}\bigr)}\biggl(\frac{2^{\nu}\Gamma(\nu+1)c_\star}{j_{\nu,1}^{\nu+{1}/{3}}}\biggr)^q
	=:\mathscr E_{1}(\nu,q).
	\label{eq:upperBoundTail}
	\end{align}
	For the integral on the right-hand side of \eqref{eq:preupperbound} to converge, it is necessary that 
	$q>3\frac{2\nu+2}{3\nu+1}$, or equivalently\footnote{The conditions $\frac{6d}{3d-4}\leq q\in\N$ translate into $q\geq 6,\, q\geq 4,\, q\geq 3, \, q\geq 3, \, q\geq 3,\, q\geq 3$ for $d=2,3,4,5,6,7$, respectively. These constraints on $q$ will not interfere with the arguments below.} $q\geq \frac{6d}{3d-4}$. In this case, the function $\mathscr E_{1}(\nu,q)$ defined in \eqref{eq:upperBoundTail} is seen to decay exponentially in $q$,  since 
	\begin{equation}\label{eq:lessthanone}
	\frac{2^{\nu}\Gamma(\nu+1)c_\star}{j_{\nu,1}^{\nu+1/3}}<1, \text{ for } \nu\in\left\{0,\frac12,1,\frac32,2,\frac 52\right\}.
	\end{equation} 
		From the comments following \eqref{eq:increasingT}, 
		we have that, for all odd integers $k\geq 1$ and all $\ell\in\N_0$,
	\begin{gather}\label{eq:T3JT4}
	T_k(r)\leq T_{k+2\ell}(r)\leq \J_\nu(r)\leq T_{k+1+2\ell}(r)\leq T_{k+1}(r),\text{ for all } r\in [0,j_{\nu,1}],\\
	\label{eq:Tkzero}
	T_k(r)\geq 0, \text{ for all } r\in [0,y_{\nu,k}],
	\end{gather}
	where $y_{\nu,k}\in(0,j_{\nu,1})$ denotes the first zero of the polynomial $T_k$. 
	 With 15 decimal places, we tabulate the relevant values of $y_{\nu,k}$:
	\begin{align*}
	y_{0,3}\sim 2.391646690891294,\; y_{\frac{1}{2},3}\sim 3.078642304481513,\;y_{1,3}\sim 3.657890099963828,\\
	y_{\frac{3}{2},3}\sim 4.147402173994025,\, y_{2,3}\sim 4.570330034603563,\, y_{\frac{5}{2},5}\sim 5.650828448306438.
	\end{align*}

\noindent We shall consider $k\in2\N+1$ satisfying \eqref{eq:T3JT4}, \eqref{eq:Tkzero}, 
	and additionally assume that 
	\begin{equation}\label{eq:ToBeProved}
	T_{k+1}(r)\leq \exp\Bigl(-\frac{r^2}{4(\nu+1)}\Bigr),\text{ for every } r\in [0,j_{\nu,1}], 
	\end{equation}
	which holds in view of the discussion following \eqref{eq:GaussianAboveTk}, provided $k$ is large enough depending on $\nu$. Later on in the proof, we will set explicit values for $k$, for each $\nu\in \{0,\frac12,1,\frac32,2,\frac 52\}$.	
	Inequalities \eqref{eq:T3JT4}, \eqref{eq:Tkzero},  \eqref{eq:ToBeProved} are the starting point for the effective lower and upper bounds for  $I_\nu(q)$, which will now be the focus of our attention.
	
	\noindent {\bf Lower bound for $I_\nu(q)$.} 
	Let $m\in\N$, whose value will be set later on in the course of the proof. 
	From 
	\eqref{eq:GaussianPSLow}, \eqref{eq:GaussianPSHigh}, we have that
	\begin{equation}\label{eq:GaussianFactor}
	\exp\Bigl(\frac{r^2}{4(\nu+1)}\Bigr)=\sum_{j=0}^\infty \frac{(\frac{1}{4}r^2)^j}{j!(\nu+1)^j}=P_m(r)+E_m(r),
	\end{equation}
	which will be of relevance in the region $r\in[0,y_{\nu,k}]$.
	Since all summands in \eqref{eq:GaussianFactor} are nonnegative, a trivial lower bound for the corresponding left-hand side is obtained by keeping only the first $m+1$ terms of the power series. 
	Therefore, for $r\in[0,y_{\nu,k}\sqrt{q}]$, the following bounds hold:
	\begin{equation}\label{eq:PreExpansion}
	1\geq\exp\Bigl(\frac{r^2}{4q(\nu+1)}\Bigr)T_k\Bigl(\frac r{\sqrt{q}}\Bigr)\geq P_m\Bigl(\frac r{\sqrt{q}}\Bigr)T_k\Bigl(\frac r{\sqrt{q}}\Bigr)\geq 0.
	\end{equation}
	The latter term can be expanded in powers of $r^2/(4q)$. 
	By doing so, the resulting polynomial has degree $m+k$ and no linear term, hence we may write
	\begin{equation}\label{eq:PmTkSuma}
	P_m\Bigl(\frac r{\sqrt{q}}\Bigr)T_k\Bigl(\frac r{\sqrt{q}}\Bigr)
	=1+\sum_{j=2}^{m+k} a_j\Bigl(\frac{r^2}{4q}\Bigr)^j,
	\end{equation}
	for some coefficients $a_j=a_j(\nu)$ which can be computed explicitly. 
	From \eqref{eq:PreExpansion}, it follows that
	\begin{equation}\label{eq:abetween01}
	 -1\leq {\sum_{j=2}^{m+k} a_j\Bigl(\frac{r^2}{4q}\Bigr)^j}\leq 0,\,\,\, \text{ for every } r\in[0,y_{\nu,k}\sqrt{q}],
	 \end{equation}
	the upper bound being strict if $r\neq 0$.
	Now, a simple change of variables yields
	\[ K_\nu(q):=q^{\nu+1}\int_0^{y_{\nu,k}}{T_k^q(r)} r^{2\nu+1}\d r=\int_0^{y_{\nu,k}\sqrt{q}}{T_k^q\Bigl(\frac{r}{\sqrt{q}}\Bigr)} r^{2\nu+1}\d r,\]
	which we estimate from below as follows:
	\begin{align*}
	K_\nu(q)&=\int_0^{y_{\nu,k}\sqrt{q}}e^{-\frac{r^2}{4(\nu+1)}}\Bigl( e^{\frac{r^2}{4q(\nu+1)}} T_k\Bigl(\frac{r}{\sqrt{q}}\Bigr)\Bigr)^q r^{2\nu+1}\d r\\
	&\geq\int_0^{y_{\nu,k}\sqrt{q}}e^{-\frac{r^2}{4(\nu+1)}}\Bigl(P_m\Bigl(\frac r{\sqrt{q}}\Bigr)T_k\Bigl(\frac{r}{\sqrt{q}}\Bigr)\Bigr)^q r^{2\nu+1}\d r\\
	&=\int_0^{y_{\nu,k}\sqrt{q}}e^{-\frac{r^2}{4(\nu+1)}}\Bigl(1+\sum_{j=2}^{m+k} a_j\Bigl(\frac{r^2}{4q}\Bigr)^j\Bigr)^q r^{2\nu+1}\d r\\
	&\geq \int_0^{y_{\nu,k}\sqrt{q}}e^{-\frac{r^2}{4(\nu+1)}}\Bigl( 1+q\sum_{j=2}^{m+k} a_j\Bigl(\frac{r^2}{4q}\Bigr)^j\Bigr) r^{2\nu+1}\d r.
	\end{align*}
	In the last line, we used Bernoulli's inequality, which can be invoked in view of \eqref{eq:abetween01} since $q\geq 1$.
	Disregarding the tail error for a moment, we are thus led to define the quantity
	\begin{equation}\label{eq:defLnu}
	L_\nu(q):= \int_0^{\infty}e^{-\frac{r^2}{4(\nu+1)}}\biggl(1+q\sum_{j=2}^{m+k} a_j\Bigl(\frac{r^2}{4q}\Bigr)^j\biggr)r^{2\nu+1}\d r.
	\end{equation}
	The latter integral can be  expanded as a sum in powers of $q^{-1}$, with coefficients given in terms of the Gamma function at integers or half-integers, yielding 
	\[ L_\nu(q)=\beta_0+\frac{\beta_1}{q}+\dots+\frac{\beta_{m+k-1}}{q^{m+k-1}},\]
	for some coefficients $\beta_j=\beta_j(\nu)$ which can be determined explicitly. 
	In this way, we obtain the lower bound 
	\begin{equation}\label{eq:LowerBound}
	I_\nu(q)\geq K_\nu(q)\geq L_\nu(q)+\eps_2(\nu,q), 
	\end{equation}
	where
	\begin{equation}\label{eq:Eps2}
	\eps_2(\nu, q):=-\int_{y_{\nu,k}\sqrt{q}}^\infty e^{-\frac{r^2}{4(\nu+1)}}\biggl(1+q\sum_{j=2}^{m+k} a_j \Bigl(\frac{r^2}{4q}\Bigr)^j\biggr) r^{2\nu+1}\d r.
	\end{equation}
	We proceed to obtain an explicit lower bound for $\eps_2(\nu,q)$. 
	With this purpose in mind, recall the definition of the  
	incomplete Gamma function, $\Gamma(a,x):=\int_{x}^\infty e^{-t}t^{a-1}\d t$. From \cite[Eq.\@ (3.2)]{NP00a}, for any $B>1$, the following upper bound holds:
	\[ \Gamma(a,x)\leq B x^{a-1}e^{-x}, \text{ for all } a\geq 1 \text{ and } x\geq \frac{B}{B-1}(a-1); \]
	see also \cite[Cor.\@ 2.5]{BC09}.
	On the other hand, from \cite[Eq.\@ (3.3)]{NP00a}, we have that $\Gamma(a,x)\geq x^{a-1}e^{-x}$, for all $a\geq 1$ and $x\geq 0$.	For   $\alpha>0$, $\beta\geq 0$, and $x\geq \bigr(\frac{B\beta}{(B-1)\alpha}\bigl)^{1/2}$, the integral
	\begin{equation*}
	\int_{x}^\infty e^{-\alpha r^2}r^{2\beta+1}\d r=\frac{1}{2\alpha^{\beta+1}}\Gamma(\beta+1,\alpha x^2),
	\end{equation*}
	is thus seen to satisfy the two-sided estimate
	\begin{equation}\label{eq:igammaBound}
	\frac{x^{2\beta}}{2\alpha}e^{-\alpha x^2}\leq \int_{x}^\infty e^{-\alpha r^2}r^{2\beta+1}\d r\leq \frac{Bx^{2\beta}}{2\alpha}e^{-\alpha x^2}.
	\end{equation}	
	To obtain a lower bound for $\eps_2(\nu,q)$, we keep only the terms on the right-hand side of \eqref{eq:Eps2} for which $a_j>0$, and bound the resulting integral with \eqref{eq:igammaBound}, yielding
	\begin{align}
	-\eps_2(\nu, q)&\leq \int_{y_{\nu,k}\sqrt{q}}^\infty e^{-\frac{r^2}{4(\nu+1)}}\Big(1+q\sum_{j=2,\,a_j>0}^{m+k} a_j \Bigl(\frac{r^2}{4q}\Bigr)^j\Big) r^{2\nu+1}\d r\notag\\
	&\leq 2B(\nu+1)y_{\nu,k}^{2\nu} \biggl(1
	+q\sum_{j=2,\,a_j>0}^{m+k}\frac{a_j}{4^{j}}y_{\nu,k}^{2j}\biggr)q^{\nu} \exp\left(-\frac{y_{\nu,k}^2 q}{4(\nu+1)}\right)=:\mathscr E_2(\nu,q),
	\label{eq:defErrorE2}
	\end{align}
	provided $B>1$, $\nu\geq 0$, and $q\geq 4B(B-1)^{-1}y_{\nu,k}^{-2} (\nu+1)(\nu+m+k)$. In this way, we obtain the following lower bound for $I_\nu(q)$:
	\begin{equation}\label{eq:LowerBoundInu}
	I_\nu(q)\geq L_\nu(q)-\mathscr E_2(\nu,q).
	\end{equation}
	
	\noindent {\bf Upper bound for $I_\nu(q)$.}
	In order to obtain an effective upper bound for $I_\nu(q)$, recall that $k$ is odd, and thus $\J_\nu\leq T_{k+1}$ on $[0,j_{\nu,1}]$.	As in \eqref{eq:GaussianFactor}, we decompose
	\begin{equation}\label{eq:expWithError}
	\exp\Bigl(\frac{r^2}{4q(\nu+1)}\Bigr)=P_m\Bigl(\frac{r}{\sqrt{q}}\Bigr)+E_m\Bigl(\frac{r}{\sqrt{q}}\Bigr).
	\end{equation}
	From \eqref{eq:GaussianErrorBound}, the following upper bound for the tail holds:
	\begin{equation}\label{eq:ErrorBoundTaylor}
	E_m\Bigl(\frac{r}{\sqrt{q}}\Bigr)\leq \frac{1}{(m+1)!}\biggl(\frac{r^{2}}{4q(\nu+1)}\biggr)^{m+1}{\exp\Bigl(\frac{j_{\nu,1}^2}{4(\nu+1)}\Bigr)},\text{ for all } r\in[0,j_{\nu,1}\sqrt{q}].
	\end{equation}
	Arguing as in \eqref{eq:PmTkSuma}, we can write
	\[\exp\Bigl(\frac{r^2}{4q(\nu+1)}\Bigr)T_{k+1}\Bigr(\frac r{\sqrt{q}}\Bigl)=1+\sum_{j=2}^\infty b_j\Bigl(\frac{r^2}{4q}\Bigr)^j,\]
		for some coefficients $b_j=b_j(\nu)$ which can be computed explicitly. The sum on the latter right-hand side is again seen to be non-positive, with absolute value bounded by $1$, provided $r\in[0,j_{\nu,1}\sqrt{q}]$; this follows from \eqref{eq:ToBeProved}, together with the fact that $T_{k+1}\geq 0$ on $[0,j_{\nu,1}]$.
	Now, consider the quantity
	\[ \widetilde K_\nu(q)
	:=q^{\nu+1}\int_0^{j_{\nu,1}}{T_{k+1}^q(r)}r^{2\nu+1}\d r 
	=\int_0^{j_{\nu,1}\sqrt{q}}{T_{k+1}^q\Bigl(\frac{r}{\sqrt{q}}\Bigr)}r^{2\nu+1}\d r, \]
	which can be estimated as follows:
	\begin{align*}
	\widetilde K_\nu(q)&=\int_0^{j_{\nu,1}\sqrt{q}}e^{-\frac{r^2}{4(\nu+1)}}\Bigl( e^{\frac{r^2}{4q(\nu+1)}} T_{k+1}\Bigl(\frac{r}{\sqrt{q}}\Bigr)\Bigr)^q r^{2\nu+1}\d r\\
	&=\int_0^{j_{\nu,1}\sqrt{q}}e^{-\frac{r^2}{4(\nu+1)}}\biggl(1+\sum_{j=2}^\infty b_j\Bigl(\frac{r^2}{4q}\Bigr)^j\biggr)^q r^{2\nu+1}\d r\\
	&\leq \int_0^{j_{\nu,1}\sqrt{q}}e^{-\frac{r^2}{4(\nu+1)}}\biggl( 1+q\biggl[\sum_{j=2}^\infty b_j\Bigl(\frac{r^2}{4q}\Bigr)^j\biggr]+\frac{q(q-1)}{2}\biggl[\sum_{j=2}^\infty b_j\Bigl(\frac{r^2}{4q}\Bigr)^j\biggr]^2\biggr) r^{2\nu+1}\d r.
	\end{align*}
	In the last line, we used the facts that $-1\leq\sum_{j=2}^\infty b_j\bigl(\frac{r^2}{4q}\bigr)^j\leq 0$ and $q\geq 2$  in order to ensure that
	\[ \biggl(1+\sum_{j=2}^\infty b_j\Bigl(\frac{r^2}{4q}\Bigr)^j\biggr)^q\leq 1+q\biggl[\sum_{j=2}^\infty b_j\Bigl(\frac{r^2}{4q}\Bigr)^j\biggr]+\frac{q(q-1)}{2}\biggl[\sum_{j=2}^\infty b_j\Bigl(\frac{r^2}{4q}\Bigr)^j\biggr]^2. \]
	It would be preferable to instead analyze a finite sum. 
	Using \eqref{eq:expWithError}, we can  express
	\begin{equation}\label{eq:PreSquare}
	\sum_{j=2}^\infty b_j\Bigl(\frac{r^2}{4q}\Bigr)^j
	=\biggl(P_m\Bigl(\frac{r}{\sqrt{q}}\Bigr)T_{k+1}\Bigl(\frac{r}{\sqrt{q}}\Bigr)-1\biggr)+E_m\Bigl(\frac{r}{\sqrt{q}}\Bigr)T_{k+1}\Bigl(\frac r{\sqrt{q}}\Bigr)
	\end{equation}
	as a finite linear combination of powers of $r^2/(4q)$, plus a well-controlled term, as dictated by \eqref{eq:ErrorBoundTaylor} and the fact that $0\leq T_{k+1}(r/\sqrt{q})\leq 1$, for all $r\in[0,j_{\nu,1}\sqrt{q}]$. 
	We may further square both sides of \eqref{eq:PreSquare}, and invoke the elementary inequality $(x+y)^2\leq 2(x^2+y^2)$ in order to  estimate:
	\begin{equation*} 
	\biggl(\sum_{j=2}^\infty b_j\Bigl(\frac{r^2}{4q}\Bigr)^j\biggr)^2
	\leq
	2\biggl(P_m\Bigl(\frac{r}{\sqrt{q}}\Bigr)T_{k+1}\Bigl(\frac{r}{\sqrt{q}}\Bigr)-1\biggr)^2+2E_m\Bigl(\frac{r}{\sqrt{q}}\Bigr)^2T_{k+1}\Bigl(\frac{r}{\sqrt{q}}\Bigr)^2.
	\end{equation*}
	Taking the previous bounds into account, and recalling \eqref{eq:ErrorBoundTaylor}, we are then led to define the quantity
	\begin{equation}\label{eq:defUv}
	\begin{split}
	U_\nu(q)&:=\int_0^\infty e^{-\frac{r^2}{4(\nu+1)}}\biggl( 1+q\Bigl(P_m\Bigl(\frac{r}{\sqrt{q}}\Bigr)T_{k+1}\Bigl(\frac{r}{\sqrt{q}}\Bigr)-1\Bigr)\biggr)r^{2\nu+1}\d r\\
	&\quad+q(q-1)\int_0^\infty e^{-\frac{r^2}{4(\nu+1)}} \biggl(P_m\Bigl(\frac{r}{\sqrt{q}}\Bigr)T_{k+1}\Bigl(\frac{r}{\sqrt{q}}\Bigr)-1\biggr)^2 r^{2\nu+1}\d r\\
	&\quad+\frac{qe^{\frac{j_{\nu,1}^2}{4(\nu+1)}}}{(m+1)!\,(4q(\nu+1))^{m+1}}\int_0^\infty e^{-\frac{r^2}{4(\nu+1)}}T_{k+1}\Bigl(\frac{r}{\sqrt{q}}\Bigr)r^{2\nu+2m+3}\d r\\
	&\quad+\frac{q(q-1)e^{\frac{j_{\nu,1}^2}{2(\nu+1)}}}{((m+1)!)^2(4q(\nu+1))^{2(m+1)}}  \int_0^\infty e^{-\frac{r^2}{4(\nu+1)}}T^2_{k+1}\Bigl(\frac{r}{\sqrt{q}}\Bigr) r^{2\nu+4m+5}\d r.
	\end{split}
	\end{equation}

	\noindent One easily checks that $U_\nu(q)$ is a polynomial in $q^{-1}$, with coefficients which can be expressed in terms of the Gamma function on integers or half-integers. Moreover, recalling the definition of $\mathscr E_1(\nu,q)$ from \eqref{eq:upperBoundTail}, we have that
	\begin{equation}\label{eq:UpperBound}
	I_\nu(q)\leq U_\nu(q)+\mathscr E_1(\nu,q)+\mathscr \eps_3(\nu,q),
	\end{equation}
	where
	\begin{align*}
	\eps_3(\nu,q)&:=-\int_{j_{\nu,1}\sqrt{q}}^\infty e^{-\frac{r^2}{4(\nu+1)}}\biggl( 1+q\Bigl(P_m\Bigl(\frac{r}{\sqrt{q}}\Bigr)T_{k+1}\Bigl(\frac{r}{\sqrt{q}}\Bigr)-1\Bigr)\biggr)r^{2\nu+1}\d r\\
	&\quad-\frac{qe^{\frac{j_{\nu,1}^2}{4(\nu+1)}}}{(m+1)!\,(4q(\nu+1))^{m+1}}\int_{j_{\nu,1}\sqrt{q}}^\infty e^{-\frac{r^2}{4(\nu+1)}}T_{k+1}\Bigl(\frac{r}{\sqrt{q}}\Bigr)r^{2\nu+2m+3}\d r.
	\end{align*}
	As in the case of $\eps_2(\nu,q)$ treated above, we can obtain an explicit upper bound for $\eps_3(\nu,q)$. 
	With this purpose in mind, write
	\[ P_m\Bigl(\frac{r}{\sqrt{q}}\Bigr)T_{k+1}\Bigl(\frac{r}{\sqrt{q}}\Bigr)=1+\sum_{j=2}^{m+k+1}
	c_j\Bigl(\frac{r^2}{4q}\Bigr)^j, \]
	for some coefficients $c_j=c_j(\nu)$ which can be determined explicitly.
	If $q\geq 4B(B-1)^{-1}j_{\nu,1}^{-2}(\nu+1)(\nu+m+k+1)$, then \eqref{eq:igammaBound} implies
	\begin{align*}
	-\int_{j_{\nu,1}\sqrt{q}}^\infty &e^{-\frac{r^2}{4(\nu+1)}}\biggl( 1+q\Bigl(P_m\Bigl(\frac{r}{\sqrt{q}}\Bigr)T_{k+1}\Bigl(\frac{r}{\sqrt{q}}\Bigr)-1\Bigr)\biggr)r^{2\nu+1}\d r\\
	&\leq \int_{j_{\nu,1}\sqrt{q}}^\infty e^{-\frac{r^2}{4(\nu+1)}}\biggl( -1+q\sum_{j=2,\,c_j<0}^{m+k+1}
	\ab{c_j}\Bigl(\frac{r^2}{4q}\Bigr)^j \biggr)r^{2\nu+1}\d r\\
	&\leq -2(\nu+1)j_{\nu,1}^{2\nu}q^\nu e^{-\frac{j_{\nu,1}^2 q}{4(\nu+1)}}+ 2B(\nu+1)j_{\nu,1}^{2\nu} \biggl(
	\sum_{j=2,\,c_j<0}^{m+k+1}\frac{\ab{c_j}}{4^{j}}{j}_{\nu,1}^{2j}\biggr)q^{\nu+1} e^{-\frac{j_{\nu,1}^2 q}{4(\nu+1)}},
	\end{align*}
	and
	\begin{align*}
	-\int_{j_{\nu,1}\sqrt{q}}^\infty e^{-\frac{r^2}{4(\nu+1)}}&T_{k+1}\Bigl(\frac{r}{\sqrt{q}}\Bigr)r^{2\nu+2m+3}\d r\\
	&\leq \sum_{j=0,\,j\text{ odd}}^k\frac{\Gamma(\nu+1)}{4^{j}q^j\, j!\,\Gamma(\nu+j+1)}\int_{j_{\nu,1}\sqrt{q}}^\infty e^{-\frac{r^2}{4(\nu+1)}}r^{2\nu+2m+2j+3}\d r\\
	&\leq2B(\nu+1)j_{\nu,1}^{2(\nu+m+1)}\biggl(\sum_{j=0,\,j\text{ odd}}^k\frac{\Gamma(\nu+1)(\frac{1}{4}j_{\nu,1}^{2})^j}{j!\,\Gamma(\nu+j+1)}\biggr)q^{\nu+m+1}e^{-\frac{j_{\nu,1}^2q}{4(\nu+1)}}.
	\end{align*}
	From the two previous estimates,  we have that
	\begin{equation}\label{eq:defErrorE3}
	\begin{split}
	\eps_3(\nu,&q)\leq 2(\nu+1)j_{\nu,1}^{2\nu} \biggl(-1
	+qB\sum_{j=2,\,c_j<0}^{m+k+1}\frac{\ab{c_j}}{4^{j}}{j}_{\nu,1}^{2j}\biggr)q^{\nu} e^{-\frac{j_{\nu,1}^2 q}{4(\nu+1)}}\\
	&\quad+\frac{B}{2}\frac{j_{\nu,1}^{2(\nu+m+1)}e^{\frac{j_{\nu,1}^2}{4(\nu+1)}}}{(m+1)!\,(4(\nu+1))^{m}}\biggl(\sum_{j=0,\, j\text{ odd}}^k\frac{\Gamma(\nu+1)(\frac{1}{4}j_{\nu,1}^{2})^j}{j!\,\Gamma(\nu+j+1)}\biggr)q^{\nu+1}e^{-\frac{j_{\nu,1}^2q}{4(\nu+1)}}
	=:\mathscr E_3(\nu,q),
	\end{split}
	\end{equation}
	provided $q\geq 4B(B-1)^{-1}j_{\nu,1}^{-2}(\nu+1)(\nu+m+k+1)$.
	In this way, we obtain the following upper bound for $I_\nu(q)$:
	\begin{equation}\label{eq:UpperBoundInu}
	I_\nu(q)\leq U_\nu(q)+\mathscr E_1(\nu,q)+\mathscr E_3(\nu,q).
	\end{equation}
	
	\noindent {\bf Putting it all together.}
	From \eqref{eq:LowerBoundInu} and \eqref{eq:UpperBoundInu}, we obtain the effective two-sided estimate\footnote{Recall \eqref{eq:defLnu}, \eqref{eq:defUv} for the definition of $L_\nu, U_\nu$, and \eqref{eq:upperBoundTail}, \eqref{eq:defErrorE2}, \eqref{eq:defErrorE3} for the definition of the error terms $\mathscr E_1,\mathscr E_2, \mathscr E_3$, respectively.} for $I_\nu(q)$,
	\[L_\nu(q)-\mathscr E_2(\nu,q) \leq I_\nu(q)\leq U_\nu(q)+\mathscr E_1(\nu,q)+\mathscr E_3(\nu,q),\]
	provided that
	\begin{equation}\label{eq:lowerBoundq}
	q\geq \frac{4B(\nu+1)}{B-1}\max\{y_{\nu,k}^{-2}(\nu+m+k),j_{\nu,1}^{-2}(\nu+m+k+1)\}.
	\end{equation}
	In order to verify \eqref{eq:ETCheck}, and therefore the desired inequality \eqref{eq:PhiPhi}, it will therefore suffice to check that
	\begin{equation}\label{eq:ToCheck}
	\frac{L_\nu(q+2)-\mathscr E_2(\nu,q+2)}{(q+2)^{\nu+1}}\geq \frac{q+1}{q+4(\nu+1)-\delta_{\nu,1}}\frac{U_\nu(q)+\mathscr E_1(\nu,q)+\mathscr E_3(\nu,q)}{q^{\nu+1}}. 
	\end{equation}
	
\noindent	For each $d\in\{2,3,4,5,6,7\}$, we need to select $k\in2\N+1$ for which \eqref{eq:ToBeProved} holds, and then choose appropriate values for $m, B$. Lemma \ref{lem:GoodTk} implies that any odd integer $k\geq 3$ is in principle a valid choice. Taking this into account, we choose the values $(k,m)$ as follows: 
	\begin{center}
	\begin{TAB}(r,1cm,0.5cm)[4pt]{|c|c|c|c|c|c|c|}{|c|c|}
	$d$    & $2$ & $3$ & $4$ & $5$  &  $6$ & $7$\\
	$(k,m)$   &  $(3,3)$ & $(3,3)$ & $(3,4)$& $(3,4)$&$(3,4)$& $(5,6)$\\
	\end{TAB}
	\end{center}

	\noindent We further set $B=5$.
	For integer values of $q$, inequality \eqref{eq:lowerBoundq} then translates into 
	\begin{equation}\label{eq:restriction1}
	q\geq 7, \, q\geq 6, \, q\geq 7, \, q\geq 7, \, q\geq 7, \, q\geq 8, \,\,\,\text{ for } d=2,3,4,5,6,7,
	\end{equation}
	respectively.
	Inequality \eqref{eq:ToCheck} can be addressed in a similar way to \eqref{eq:claimedIneqQ4}. 
The idea is to first work without the error terms $\mathscr E_1,\mathscr E_2,\mathscr E_3$, and for each relevant dimension $d$ to find $q_0(d)$, in such a way that the inequality without error terms holds, for all $q\geq q_0(d)$. Since it can be transformed into a polynomial inequality,  this step is in principle an easy one, similar to what is done at the end of the proof of Proposition \ref{prop:d=3}. The next step is to show that the presence of the error terms does not significantly change the value of $q_0(d)$, since they are exponentially decreasing in $q$. Since the analysis is more cumbersome than in the case $d=3$ of Proposition \ref{prop:d=3},  we instead present plots of the difference of the left- and right-hand sides of inequality \eqref{eq:ToCheck}; see Figures \ref{fig:alld1} and \ref{fig:alld2}. In particular, taking \eqref{eq:restriction1} into consideration, this reveals that, if $d\in\{2,3,4,5,6,7\}$, then inequality \eqref{eq:ToCheck} is satisfied for every $q\geq q_\star(d)$, with $q_\star(d)\in2\N$ as stated in \eqref{eq:valueQ0}.

This completes the proof of the proposition.
\end{proof}

\begin{figure}
\centering
\includegraphics[width=.495\linewidth]{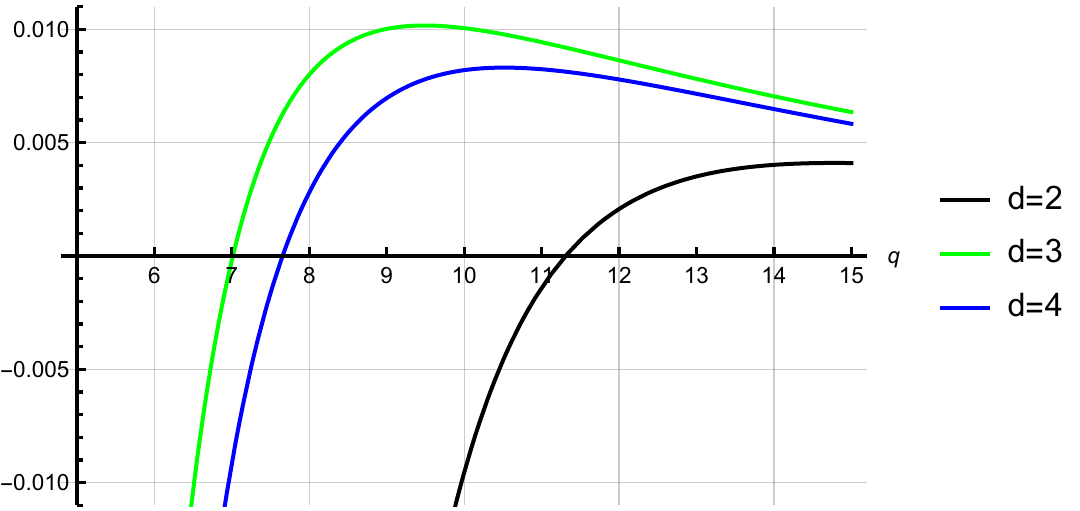}
\includegraphics[width=.495\linewidth]{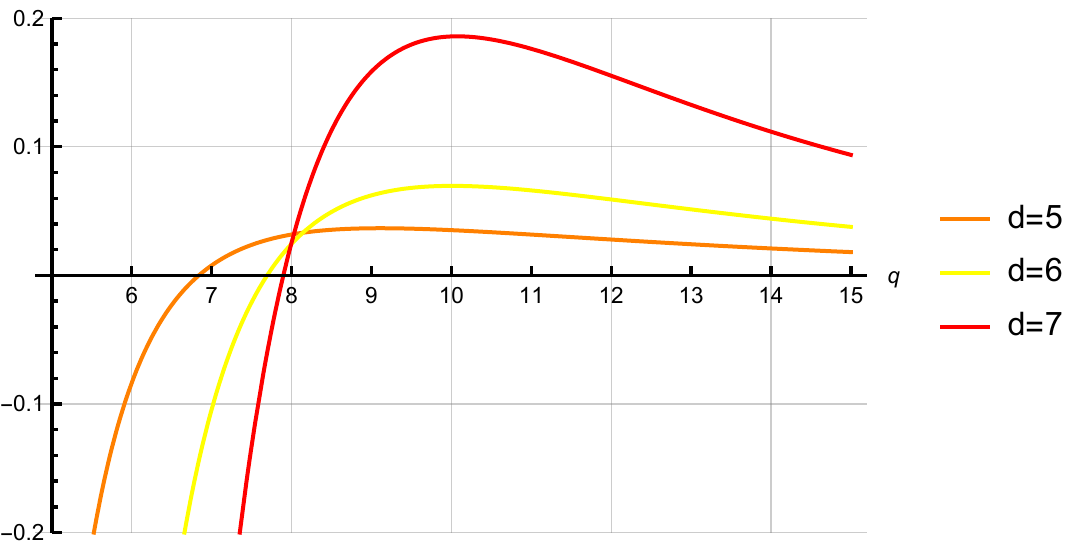}
\caption{Plot of the difference of the left- and right-hand sides of inequality \eqref{eq:ToCheck},
 in the range $q\in[5,15]$.}
\label{fig:alld1}
\end{figure}

\begin{figure}
\centering
\includegraphics[width=.9\linewidth]{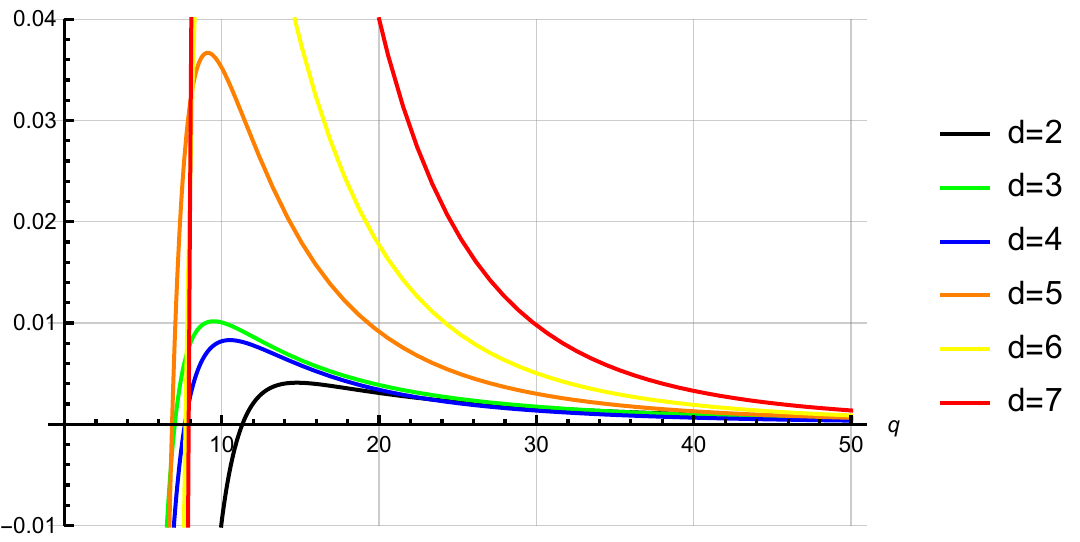}
\caption{Plot of the difference of the left- and right-hand sides of inequality \eqref{eq:ToCheck},
 in the range $q\in[5,50]$.}
\label{fig:alld2}
\end{figure}

\section{Asymptotic expansion}\label{sec:fullAsymp}

In this last section, we would like to discuss the full asymptotic expansion of the integral $I_\nu(q)$ defined in \eqref{eq:DefInu}, thereby complementing the discussion in \cite{KOS15}. 
Interestingly, more than one century ago, Pearson \cite{Pe} already used a similar method tho study the probability density $p_n(0;r)$ associated to the $n$-step uniform random walk in $\R^2$; see also \cite[\S 6]{GD55}.

It will be of no additional difficulty to consider a slightly more general case. Define
\begin{equation}\label{eq:generalInumu}
I_{\nu,\mu}(q):=q^{\frac{\mu+1}{2}}\int_0^\infty \ab{\J_{\nu}(r)}^q\, r^\mu \d r,
\end{equation}
where $\nu\geq 0,\,\mu> -1$, and $q> 2\mu$. 
Given $n\geq 1$, we aim at an expansion of the form
\begin{equation}\label{eq:asymptoticI}
I_{\nu,\mu}(q)=\alpha_0+\frac{\alpha_1}{q}+\dots+\frac{\alpha_n}{q^n}+O\biggl(\frac{1}{q^{n+1}}\biggr),
\end{equation}
valid for every $q\geq q_\star$, for $q_\star$ large enough depending on $n,\nu,\mu$, and coefficients $\{\alpha_j\}_{j=0}^n$ that depend only on $\nu,\mu$. Our starting point is the inequality
\begin{equation}\label{eq:GaussianAboveJ2}
0\leq \J_\nu(r)< \exp\Bigl(-\frac{r^2}{4(\nu+1)}\Bigr),\; r\in(0,j_{\nu,1}],
\end{equation}
which is a consequence of \eqref{eq:GaussianAboveJ}. We could proceed as in the proof of Proposition \ref{prop:sharpIneqD}, through the use of the truncations $T_{\nu,k}$, but since we will not be interested in sharp bounds for the error terms nor for the threshold $q_\star$, the simpler argument from \cite[\S 2 and \S 3]{KOS15} suffices. The steps leading to \eqref{eq:asymptoticI} can be summarized as follows:
\begin{enumerate}
	\item[1.] Reduce to a bounded domain of integration, say $[0,j_{\nu,1}]$;
	\item[2.] Introduce the Gaussian weight, $\exp\bigl(-\frac{r^2}{4q(\nu+1)}\bigr)$;
	\item[3.] Expand the function $\exp\bigl(-\frac{r^2}{4q(\nu+1)}\bigr)\J_\nu\bigl(\frac{r}{\sqrt{q}}\bigr)$ in power series;
	\item[4.] Use the binomial expansion for $(1+x)^q$ on the interval $[0,j_{\nu,1}\sqrt{q}]$, where 
	$x=\exp\bigl(-\frac{r^2}{4q(\nu+1)}\bigr)\J_\nu\bigl(\frac{r}{\sqrt{q}}\bigr)-1$ satisfies $-1\leq x\leq 0$;
	\item[5.] Estimate the error terms, and obtain an explicit formula for the coefficients $\{\alpha_j\}_{j=0}^n$.
\end{enumerate}

Since the analysis is analogous to \cite[\S 2 and \S 3]{KOS15}, we omit most details, except for the explicit formulae for the coefficients $\{\alpha_j\}_{j=0}^n$. Henceforth, we take $q\geq n\geq 1$. 

Concerning Step 1, recall that inequality \eqref{eq:lessthanone}  was invoked 
in order to ensure that the integral defining $I_\nu(q)$ decays exponentially in $q$ when restricted to the interval $[j_{\nu,1},\infty)$. For general $\nu\geq 0$,  \eqref{eq:lessthanone} can be verified using the fact that $j_{\nu,1}>j_{0,1}+\nu$, for every $\nu\in(0,\infty)$; see \cite[Eq.\@ (2.4)]{LM83}.
However,  \eqref{eq:lessthanone} turns out not to be essential, in the sense that by splitting $[j_{\nu,1},\infty)=[j_{\nu,1},x_\nu]\cup[x_\nu,\infty)$, and invoking Landau's upper bound \eqref{eq:Landau} for $J_\nu$ on the interval $[x_\nu,\infty)$, together with the trivial $L^\infty$-bound 
$\sup_{j_{\nu,1}\leq r\leq x_\nu}\ab{\J_{\nu}(r)}<1,$
 one obtains 
 \[q^{\frac{\mu+1}{2}}\int_{j_{\nu,1}}^\infty \ab{\J_\nu(r)}^q r^{\mu}\d r=O(e^{-aq}),\]
for some $a=a(\nu,\mu)>0$, provided $x_\nu$ is chosen appropriately as a function of $\nu$. It is therefore enough to study the asymptotic expansion of the integral
\[ K_{\nu,\mu}(q):=q^{\frac{\mu+1}{2}}\int_{0}^{j_{\nu,1}}\J_\nu(r)^q r^{\mu}\d r
=\int_{0}^{j_{\nu,1}\sqrt{q}}\J_\nu\Bigl(\frac{r}{\sqrt{q}}\Bigr)^q r^{\mu}\d r. \]

As for Step 4, it is useful to note that, for all $q\geq 2$, $x\in[-1,0]$, and odd $k\in[1,q-1]\cap\N$, we have that
\[ \sum_{j=0}^k\binom{q}{j}x^j\leq (1+x)^q\leq \sum_{j=0}^{k+1}\binom{q}{j}x^j. \]
Here, $\binom{q}{j}={q(q-1)\dotsm(q-j+1)}/{j!}$ as usual.

To estimate the error terms in Step 5, recall \eqref{eq:igammaBound}, together with the aforementioned bounds for $\Gamma(a,x)$ from \cite{BC09,NP00a}. The conclusion is that the coefficients $\{\alpha_j\}_{j=0}^n$ can be read off from the expansion
\begin{align}
\nonumber
&I_{\nu,\mu}(q)=\int_0^{\infty}e^{-\frac{r^2}{4(\nu+1)}} \biggl(1+\sum_{j=2}^{2n} b_j\Bigl(\frac{r^2}{4q}\Bigr)^j\biggr)r^{\mu}\d r +O\Bigl(\frac{1}{q^{n+1}}\Bigr)\\
\label{eq:explicitAE}
&=2^{\mu}(\nu+1)^{\frac{\mu+1}{2}}\Gamma\Bigl(\frac{\mu+1}{2}\Bigr)+2^{\mu}(\nu+1)^{\frac{\mu+1}{2}}\sum_{j=2}^{2n}\frac{1}{q^j}b_j(\nu+1)^{j}\Gamma\Bigl(\frac{\mu+2j+1}{2}\Bigr)+O\Bigl(\frac{1}{q^{n+1}}\Bigr),
\end{align}
where the coefficients $b_j=b_j(\nu,\mu)$ are determined by the identity
\begin{equation}\label{eq:coeffbj}
\biggl[\exp\biggl({\frac{r^2}{4q(\nu+1)}}\biggr)\J_\nu\biggl(\frac{r}{\sqrt{q}}\biggr)\biggr]^q=1+\sum_{j=2}^{\infty} b_j\Bigl(\frac{r^2}{4q}\Bigr)^j.
\end{equation}
We point out that the choice of auxiliary exponential function $\exp\bigl(\frac{r^2}{4q(\nu+1)})$ is not arbitrary, since it allows the truncation of the sum in \eqref{eq:explicitAE} up to $2n$.
Writing
\begin{equation}\label{eq:multipliedPS}
\exp\biggl({\frac{r^2}{4q(\nu+1)}}\biggr)\J_\nu\biggl(\frac{r}{\sqrt{q}}\biggr)=1+\sum_{j=2}^\infty a_j\Bigl(\frac{r^2}{4q}\Bigr)^j,
\end{equation}
where
\[ a_j:=\sum_{i=0}^j\frac{(-1)^{j-i}\Gamma(\nu+1)}{i!\,(j-i)!\,(\nu+1)^i\Gamma(\nu+j-i+1)}, \]
we have that
\[ \biggl(1+\sum_{j=2}^\infty a_jz^j\biggr)^q=1+\sum_{j=2}^{\infty} b_jz^j. \]
Binomially expanding, we find that
\[ b_k(\nu,q)=\sum_{i=1}^{\lfloor \frac{k}{2}\rfloor}\binom{q}{i}
\sum_{\substack{(\ell_1,\dots,\ell_i)\in(\N\setminus\{1\})^{i}\\ \ell_1+\dots+\ell_i=k}}\prod_{j=1}^i a_{\ell_j},\quad k\geq 2. \]
In particular,
\begin{gather*}
a_2=\frac{-1}{2(\nu+1)^2(\nu+2)},\,a_3=\frac{-2}{3(\nu+1)^3(\nu+2)(\nu+3)},\\
a_4=\frac{\nu-5}{8(\nu+1)^4(\nu+2)(\nu+3)(\nu+4)},
\end{gather*}
and
\[ b_2=qa_2,\,\,\, b_3=qa_3,\,\,\, b_4=qa_4+\frac{q(q-1)}{2}a_2^2. \]
With the notation from \eqref{eq:asymptoticI}, we then obtain
\begin{equation}\label{eq:c0c1}
\begin{split}
&\alpha_0=2^{\mu}(\nu+1)^{\frac{\mu+1}{2}}\Gamma\Bigl(\frac{\mu+1}{2}\Bigr),\,\,\, \alpha_1=-\frac{2^{\mu-1}(\nu+1)^{\frac{\mu+1}{2}}}{\nu+2}\Gamma\Bigl(\frac{\mu+5}{2}\Bigr),\\
\alpha_2&=2^{\mu}(\nu+1)^{\frac{\mu+1}{2}}\biggl(\frac{-2}{3(\nu+2)(\nu+3)}\Gamma\Bigl(\frac{\mu+7}{2}\Bigr)+\frac{1}{8(\nu+2)^2}\Gamma\Bigl(\frac{\mu+9}{2}\Bigr)\biggr).
\end{split}
\end{equation}
A similar asymptotic expansion can be obtained for the expression analogous to \eqref{eq:generalInumu} but without absolute value inside the integral, i.e.
$q^{\frac{\mu+1}{2}}\int_0^\infty \J_{\nu}(r)^q\, r^\mu \d r,$
as long as we restrict $q$ to be an integer.

\subsection{Applications}
We close this section with a few selected applications.
The first one is the following result, which verifies one of the observations based on  numerical experimentation from Remark \ref{rem:final71}.
\begin{proposition}\label{cor:neighborhoodInfty}
	Let $d\geq 2$. Then the inequality
	\[ \Phi_{d,q+2}(\one)> \gamma_{d}(q)\Phi_{d,q}(\one) \]
	holds for all $q\geq q_\star$, provided $q_\star=q_\star(d)<\infty$ is sufficiently large.
\end{proposition}
\noindent Following the strategy outlined in \S \ref{sec:proofLargeq}, it is in principle possible to obtain  effective upper bounds for the threshold $q_\star$ (as in the proof of Proposition \ref{prop:d=3}) and the error term $e_2$ in \eqref{eq:asymp2terms} below (as in the proof of Proposition \ref{prop:sharpIneqD}), but we have not investigated this point in detail.

\begin{proof}[Proof of Proposition \ref{cor:neighborhoodInfty}]
Let $\nu=(d-2)/2.$
	As in \eqref{eq:ETCheck}, we need to check that
	\[ \frac{I_\nu(q+2)}{(q+2)^{\nu+1}}>\frac{q+1}{q+2d-\delta_{d,4}}\frac{I_\nu(q)}{q^{\nu+1}}, \]
	for all sufficiently large $q$. Let $q_\star, e_2$ be such that, for all $q\geq q_\star$,
	\begin{equation}\label{eq:asymp2terms}
	\abs{I_\nu(q)-\Bigl(\alpha_0+\frac{\alpha_1}{q}\Bigr)}\leq \frac{e_2}{q^2}. 
	\end{equation}
	It is then enough to check that, for all sufficiently large $q$:
	\[ \frac{\alpha_0}{(q+2)^{\nu+1}}+\frac{\alpha_1}{(q+2)^{\nu+2}}-\frac{e_2}{(q+2)^{\nu+3}}> \frac{q+1}{q+2d-\delta_{d,4}}\Bigl(\frac{\alpha_0}{q^{\nu+1}}+\frac{\alpha_1}{q^{\nu+2}}+\frac{e_2}{q^{\nu+3}}\Bigr). \]
	Since $\alpha_0>0$, recall \eqref{eq:c0c1},  this is equivalent to showing that the following inequality holds, for all sufficiently large $q$:
	\[ \frac{1}{(1+\frac{2}{q})^{\nu+1}}+\frac{\alpha_1}{\alpha_0q(1+\frac{2}{q})^{\nu+2}}-\frac{e_2}{\alpha_0q^2(1+\frac{2}{q})^{\nu+3}}> \frac{1+\frac{1}{q}}{1+\frac{2d-\delta_{d,4}}{q}}\Bigl(1+\frac{\alpha_1}{\alpha_0q}+\frac{e_2}{\alpha_0q^2}\Bigr). \]
	Raising the latter inequality to power $q$, and then taking the limit as $q\to\infty$, yields
	\begin{align*}
	\biggl(\frac{1}{(1+\frac{2}{q})^{\nu+1}}+\frac{\alpha_1}{\alpha_0q(1+\frac{2}{q})^{\nu+2}}-\frac{e_2}{\alpha_0q^2(1+\frac{2}{q})^{\nu+3}}\biggr)^q&\to e^{-2(\nu+1)+\frac{\alpha_1}{\alpha_0}}=e^{-d+\frac{\alpha_1}{\alpha_0}},\\
	\biggl(\frac{1+\frac{1}{q}}{1+\frac{2d-\delta_{d,4}}{q}}\Bigl(1+\frac{\alpha_1}{\alpha_0q}+\frac{e_2}{\alpha_0q^2}\Bigr)\biggr)^q&\to e^{1-2d+\delta_{d,4}+\frac{\alpha_1}{\alpha_0}}.
	\end{align*}
	The result follows at once, since $d<2d-1-\delta_{d,4}$, for all $d\geq 2$.
\end{proof}

\noindent As a second application, we can now prove Theorem \ref{prop:conditionalMax}.
\begin{proof}[Proof of Theorem \ref{prop:conditionalMax}]
	From Proposition \ref{cor:neighborhoodInfty}, we know that there exists 
	$q_\star=q_\star(d)<\infty$, such that the inequality $\Phi_{d,q+2}(\one)> \gamma_{d}(q)\Phi_{d,q}(\one)$ holds, for every $q\geq q_\star$. Let $q\geq q_\star$, and suppose that $\Phi_{d,q}$ is maximized by the constant functions, and that there exists  a real-valued, continuously differentiable maximizer $f$ of $\Phi_{d,q+2}$. The argument from the proof of Theorem \ref{prop:bootstrappedMaximizer} applied to $f$ shows that if $f$ is non-constant, then $\Phi_{d,q+2}(f)<\Phi_{d,q+2}(\one)$. This concludes the proof of the theorem.
\end{proof}

\noindent As a third and last application, we compute the limiting value of $\mathbf{T}_{d,q}$, as $q\to\infty$, as promised by Theorem \ref{prop:continuityTInfty}. Recall that $\mathbf{T}_{d,q}$ denotes the optimal constant in \eqref{eq:TS}, defined in \eqref{eq:bestconstant}. 
\begin{proof}
[Proof of Theorem \ref{prop:continuityTInfty}]
	From Theorem \ref{thm:MainThm}, we have that $\mathbf{T}_{d,2k}=\Phi_{d,2k}^{1/(2k)}(\one)$, for all $d\in\{3,4,5,6,7\}$ and  integers $k\geq 3$. Let $q\geq 6$ be given, and choose $k\in\N$ in such a way that $q\in[2k,2k+2]$. By interpolation, $\mathbf{T}_{d,q}\leq \mathbf{T}_{d,2k}^\te\mathbf{T}_{d,2k+2}^{1-\te}$, where $\te\in[0,1]$ satisfies $\frac{1}{q}=\frac{\te}{2k}+\frac{1-\te}{2k+2}$. On the other hand, $\mathbf{T}_{d,q}\geq \Phi_{d,q}^{1/q}(\one)$, so that
	\[ \Phi_{d,q}^{1/q}(\one)\leq \mathbf{T}_{d,q}\leq \Phi_{d,2k}^{\frac{\te}{2k}}(\one) \Phi_{d,2k+2}^{\frac{1-\te}{2k+2}}(\one).\]
	Therefore, it suffices to show that $\lim_{q\to\infty}\Phi_{d,q}^{1/q}(\one)$ equals the expression on the right-hand side of \eqref{eq:valueLimitT}.
	Recall from the line prior to \eqref{eq:ETCheck} that $\Phi_{d,q}(\one)=\omega_{d-1}^{1+q/2}q^{-(\nu+1)}I_\nu(q)$, where $\nu=(d-2)/2$. It follows that
	\begin{align*}
	\lim_{q\to\infty}\Phi_{d,q}^{1/q}(\one)&=\omega_{d-1}^{1/2}\lim_{q\to\infty}I_\nu^{1/q}(q)=\omega_{d-1}^{1/2}\lim_{q\to\infty}\Bigl(\alpha_0+\frac{\alpha_1}{q}+O\Bigl(\frac{1}{q^2}\Bigr)\Bigr)^{1/q}\\
	&=\omega_{d-1}^{1/2}=\biggl(\frac{2\pi^{d/2}}{\Gamma(\frac{d}{2})}\biggr)^{1/2}.
	\end{align*}
	This completes the proof of the theorem.
\end{proof}

\noindent Regarding the observations from Remark \ref{rem:final71} in relation to inequality \eqref{eq:Phi_inequality}, we have already noted that Proposition \ref{cor:neighborhoodInfty} answers one of them. As for the other one, we present the following conjecture, which we plan to revisit in the nearby future.
\begin{conjecture}\label{conj:conjecturePhi}
	Let $d\geq 2$.
		If $q_\star\geq q_d$ is such that $\Phi_{d,q_\star+2}(\one)\geq \gamma_d(q_\star)\Phi_{d,q_\star}(\one)$, then $\Phi_{d,q+2}(\one)> \gamma_d(q)\Phi_{d,q}(\one)$, for every $q>q_\star$.
\end{conjecture}
\noindent Assuming the validity of Conjecture \ref{conj:conjecturePhi}, a natural problem is to determine the exact value of 
\[\inf\{q\geq q_d\colon \Phi_{d,q+2}(\one)\geq \gamma_d(q)\Phi_{d,q}(\one)\}.\]

\section*{Acknowledgements}
The computer algebra systems {\it Mathematica}, {\it Maxima} and \textit{Octave} were used to compute the entries of Tables \ref{table:357Edq} and \ref{table:2to11Edq}.
D.O.S.\@ is supported by the EPSRC New Investigator Award ``Sharp Fourier Restriction Theory'', grant no.\@ EP/T001364/1,
and is grateful to Jorge Vit\'oria for a valuable discussion during the preparation of this work. 
The authors thank Pierpaolo Natalini for providing a copy of \cite{NP00a},
and the anonymous referee for carefully reading the manuscript and valuable suggestions.

\appendix
\section{Revisiting  Garc\'ia-Pelayo (\cite{GP12})}\label{sec:AppendixGP}
In this appendix, we review the steps leading to the proof of identity \eqref{eq:GPcompactFormula}, which in turn comes from \cite[Formula (32)]{GP12}. Our discussion will sometimes be informal. 
Let $f:\R^d\to\Co$ be a radial, compactly supported function.
Given a line $L$ passing through the origin in $\R^d$, we define the projection of $f$ onto $L$, denoted $P^{d-1}_Lf$, to be the function defined on $L$,  which we identify with the real line $\R$, whose value at $r$ equals the integral of $f$ over the unique hyperplane orthogonal to $L$ that contains $r$, 
\[ (P^{d-1}_Lf)(r)=\int_{r+L^\perp}f(y)\d y. \]
Here, $\d y$ denotes the induced Lebesgue measure on the hyperplane $r+L^\perp$.
Since $f$ is radial, the projection $P^{d-1}_Lf$ is independent of the line $L$, and will be simply denoted by $P^{d-1} f$.
In this way, choosing $L$ to coincide with the $x_{d}$-axis, we have that
\[ (P^{d-1}f)(r)=\int_{\R^{d-1}} f(y,r)\d y. \]
This defines an instance of the so-called {\it Abel transform}, which bears a clear resemblance to the usual Radon transform.
Since $f$ is radial, we abuse notation slightly and write $f(y,r)=f(\sqrt{\ab{y}^2+r^2})$. In this way,
\begin{align}
\nonumber
(P^{d-1}f)(r)&=\int_{\R^{d-1}}f(\sqrt{\ab{y}^2+r^2})\d y=\omega_{d-2}\int_{0}^\infty f(\sqrt{s^2+r^2})s^{d-2}\d s\\
\label{eq:formulaProj}
&=\omega_{d-2}\int_{\ab{r}}^\infty f(t)t(t^2-r^2)^{\frac{d-3}{2}}\d t.
\end{align}
The function $f$ can be recovered from $P^{d-1}f$ by differentiation, under mild assumptions on $f$. 
For instance, if $d=3$ and $r>0$, then
\begin{equation}\label{eq:P^2}
(P^{3-1}f)(r)=2\pi\int_r^\infty f(t)t\d t,
\end{equation}
so that differentiating both sides yields

\begin{equation}\label{eq:fFromP2}
f(r)=-\frac{1}{2\pi r}\frac{\d}{\d r}(P^{3-1}f)(r).
\end{equation}
On the other hand, if $d\geq 4$ and $r>0$, then
\begin{align*}
\frac{\d}{\d r}(P^{d-1}f)(r)
&=-\omega_{d-2}(d-3)r\int_{r}^\infty f(t)t(t^2-r^2)^{\frac{d-5}{2}}\d t,
\end{align*}
so that
\begin{equation}\label{eq:recreltn}
 -\frac{1}{r}\frac{\d}{\d r}(P^{d-1}f)(r)=\frac{\omega_{d-2}}{\omega_{d-4}}(d-3)(P^{(d-2)-1}f)(r). 
 \end{equation}
If the dimension $d$ is odd, then we may apply $(d-1)/2$ times the recurrence relation \eqref{eq:recreltn}, and eventually arrive at expression \eqref{eq:P^2} for $P^{3-1}f$, leading to: 
\[ \Bigl(-\frac{1}{r}\frac{\d}{\d r}\Bigr)^{\frac{d-1}2} (P^{d-1}f)(r)=\frac{\omega_{d-2}}{\omega_1}(d-3)\dotsm 2\cdot 2\pi f(r). \]
In this way,
using $\omega_{d-2}=2\pi^{\frac{d-1}{2}}\Gamma(\frac{d-1}{2})^{-1}$ and $(d-3)\dotsm 2=2^{\frac{d-3}{2}}\Gamma(\frac{d-1}{2})$, we obtain
\begin{equation}\label{eq:GP22}
f(r)=\frac{2}{\omega_{d-2}\Gamma(\frac{d-1}{2})}\Bigl(-\frac{1}{2r}\frac{\d}{\d r}\Bigr)^{\frac{d-1}{2}} (P^{d-1}f)(r)=\Bigl(-\frac{1}{2\pi r}\frac{\d}{\d r}\Bigr)^{\frac{d-1}{2}} (P^{d-1}f)(r). 
\end{equation}
The latter formula coincides with \cite[Formula (22)]{GP12}, for $m=1$ and $j=\frac{d-1}{2}$.

Let us now consider the case of the $n$-fold convolution in $\R^d$ of a given radial function $F:\R^d\to\R$.
Fubini's Theorem can be used to check that the convolution operation commutes with the projection $P^{d-1}$:
\begin{align*}
(P^{d-1}F^{\ast n})(r)&=\int_{\R^{d-1}} F^{\ast n}(y,r)\d y
=\int_\R\int_{\R^{d-1}}\int_{\R^{d-1}}F^{\ast(n-1)}(y-z,r-s)F(z,s)\d y\d z\d s\\
&=\int_\R (P^{d-1}F^{\ast(n-1)})(r-s)(P^{d-1}F)(s)\d s=\bigl((P^{d-1}F^{\ast (n-1)})\ast (P^{d-1}F)\bigr)(r).
\end{align*}
Repeated applications of the same formula yield
\[ (P^{d-1}F^{\ast n})(r)=(P^{d-1}F)^{\ast n}(r).\]
The projection of a radial measure such as $\d\sigma_{d-1}(y,r)=\ddirac{\sqrt{|y|^2+r^2}-1}\, \d y \d r$ can be defined  in a similar way, namely through
\begin{align}
\nonumber
(P^{d-1}\sigma_{d-1})(r)&=\int_{\R^{d-1}}\ddirac{\sqrt{\ab{y}^2+r^2}-1}\d y
=\omega_{d-2}\int_{\ab{r}}^\infty \ddirac{t-1}t(t^2-r^2)^{\frac{d-3}{2}}\d t\\
\label{eq:formulaPsigma}
&=\omega_{d-2}(1-r^2)_+^{\frac{d-3}{2}}.
\end{align}
Considering  $n$-fold convolutions $\sigma_{d-1}^{\ast n}$, we find as above that $P^{d-1}(\sigma_{d-1}^{\ast n})=(P^{d-1}\sigma_{d-1})^{\ast n}$, 
so that \eqref{eq:formulaPsigma} implies
\begin{equation}\label{eq:explicitPsigma}
P^{d-1}(\sigma_{d-1}^{\ast n})=\omega_{d-2}^n\bigl((1-r^2)_+^{\frac{d-3}{2}}\bigr)^{\ast n}. 
\end{equation}
From \eqref{eq:GP22} and \eqref{eq:explicitPsigma}, it follows that
\begin{align*}
\sigma_{d-1}^{\ast n}(r)&=\omega_{d-2}^n\Bigl(-\frac{1}{2\pi r}\frac{\d}{\d r}\Bigr)^{\frac{d-1}{2}}\Bigl((1-r^2)_+^{\frac{d-3}{2}}\Bigr)^{\ast n}.
\end{align*}
Formula \eqref{eq:GPcompactFormula} follows at once.
In a similar way, the convolution of normalized surface measure $\bar\sigma_{d-1}=\omega_{d-1}^{-1}\sigma_{d-1}$ satisfies
\begin{align*}
\bar\sigma_{d-1}^{\ast n}(r)&=\Bigl(\frac{\omega_{d-2}}{\omega_{d-1}}\Bigr)^n\Bigl(-\frac{1}{2\pi r}\frac{\d}{\d r}\Bigr)^{\frac{d-1}{2}}\Bigl((1-r^2)_+^{\frac{d-3}{2}}\Bigr)^{\ast n}.
\end{align*}
The factor $(\omega_{d-2}/\omega_{d-1})^n$ admits the following useful expansion whenever $d$ is an odd integer. 
If $d=3$, then $\omega_1/\omega_2={1}/{2}$; if $d\geq 5$ is odd, then
\begin{equation}\label{eq:quotient}
\frac{\omega_{d-2}}{\omega_{d-1}}
=\frac{\Gamma(\frac{d}{2})}{\pi^{\frac{1}{2}}\Gamma(\frac{d-1}{2})}=\frac{1}{2}\frac{(d-2)\dotsm 1}{(d-3)\dotsm 2},
\end{equation}
where the numerator equals the product of all odd natural numbers up to $(d-2)$, and the denominator equals the product of all even natural numbers up to $(d-3)$. We note that there is a typographical error in  \cite[Formula (32)]{GP12}  whenever $d\geq 5$, since the latter contains a factor  $2^{-\frac{d-1}{2}}$ instead of the factor $\frac12$ on the right-hand side of \eqref{eq:quotient}.

\begin{remark}\label{rem:BS}
For the probability density $p_n(m-1/2;r)$ defined in  \eqref{eq:relationSigmaDensity}, we obtain the following expression 
\begin{align*}
p_n(m-1/2;r)
&=\frac{2r^{d-1}\Gamma(\frac{d}{2})^{n-1}}{\pi^{\frac{n-1}{2}}\Gamma(\frac{d-1}{2})^n}\Bigl(-\frac{1}{2r}\frac{\d}{\d r}\Bigr)^{\frac{d-1}{2}}\Bigl((1-r^2)_+^{\frac{d-3}{2}}\Bigr)^{\ast n}.
\end{align*}
If $d$ is an odd integer, then
\begin{align*}
\frac{2\Gamma(\frac{d}{2})^{n-1}}{\pi^{\frac{n-1}{2}}\Gamma(\frac{d-1}{2})^n}=\frac{2^{d-1}\Gamma(\frac{d-1}{2})}{\Gamma(d-1)}\biggl(\frac{\Gamma(\frac{d}{2})}{\Gamma(\frac{1}{2})\Gamma(\frac{d-1}{2})}\biggr)^n.
\end{align*}
The right-hand side of the latter identity coincides with the coefficient in the formula for $p_n(m-1/2;r)$ as it appears in \cite[Theorem 2]{BS}. So we are in agreement with \cite{BS}, and Theorem \ref{thm:formulaConvoBS} follows from \cite[Cor.\@ 6]{BS}.
\end{remark}

\section{Numerical estimates}\label{sec:numerics} 

In this appendix, we explain why the numerical evaluations from the proof of Proposition \ref{lem:verifiedIneq} are mathematically rigorous. 
The entries of Table \ref{table:2to11Edq} were obtained via the function \texttt{besselint} run on \textit{Octave}\footnote{The entries of Table \ref{table:2to11Edq} were later re-checked against the {\it Mathematica} command \texttt{NIntegrate}, as well as the \textit{Maxima} command \texttt{quad\_qagi}, which is part of the \texttt{QUADPACK} library \cite{PDKUK}. } \cite{vDC06a,vDC06b}; see also \cite{vDC08}. The function \texttt{besselint} was specifically designed to calculate integrals of the form 
\begin{equation}\label{eq:generalBesselProd}
I(\mathbf{a},\boldsymbol{\nu},m):=\int_{0}^{\infty}  r^{m}\prod_{i=1}^k J_{\nu_i}(a_i r)\d r,
\end{equation}
where $\mathbf{a}=(a_1,\dots,a_k)$, $\boldsymbol{\nu}=(\nu_1,\dots,\nu_k)$, $a_i>0$ for all $i=1,\dots,k$, $\nu_i\in \R$,  and $m+\sum_{i=1}^k\nu_i>-1$; the latter condition ensures an integrable singularity at $r=0$. 

As an illustrative example, we present a simple \textit{Octave}/\textit{Matlab} code which can be used to calculate the entries in Table \ref{table:2to11Edq}, for $4\leq q\leq 12$: 

\begin{lstlisting}
function f=E(d,q)
  if (d==4),
     g(d,q)=(1+q)/(7+q);
  else
     g(d,q)=(1+q)/(2*d+q);
  end
  f=2^(d-2)*gamma(d/2)^2 ...
    *besselint(ones(1,q+2),(d-2)/2*ones(1,q+2),d-1-(q+2)*(d-2)/2) ...
    /besselint(ones(1,q),(d-2)/2*ones(1,q),d-1-q*(d-2)/2)-g(d,q);
\end{lstlisting}

We will indicate the main points in the implementation of \texttt{besselint}, and refer to \cite{vDC06a,vDC06b} for further details, as well as for the discussion of estimating relative and absolute errors. In order to numerically calculate the integral in \eqref{eq:generalBesselProd}, a number $x_0>0$ is suitably chosen, and the integral is then split as the sum of two terms, the so-called  finite and infinite parts, respectively given by
\begin{equation}\label{eq:I1I2}
I_1(\mathbf{a},\boldsymbol{\nu},m,x_0)=\int_{0}^{x_0} r^{m}\prod_{i=1}^k J_{\nu_i}(a_i r)\d r,\quad I_2(\mathbf{a},\boldsymbol{\nu},m,x_0)=\int_{x_0}^{\infty} r^{m}\prod_{i=1}^k J_{\nu_i}(a_i r)\d r. 
\end{equation}

For the finite part, $I_1(\mathbf{a},\boldsymbol{\nu},m,x_0)$, the interval $[0,x_0]$ is divided into a certain number of subintervals at equidistant points. The number of subintervals roughly corresponds to the number of zeros of the integrand in the interval $[0,x_0]$, estimated according to the well-known approximation
\begin{equation}\label{eq:asymptJnu}
 J_\nu(r)\sim \Bigl(\frac{2}{\pi r}\Bigr)^{1/2}\cos\Bigl(r-\Bigl(\frac{\nu}{2}+\frac{1}{4}\Bigr)\pi\Bigr). 
 \end{equation}
In each subinterval, the integration is then estimated via a Gauss--Legendre quadrature rule.
In most cases, this amounts to a 15-point rule in order to reach full precision, followed by a 19-point rule in order to obtain an estimate for the error as the absolute value of the difference of the two.
This part is handled by the sub-routine \texttt{fri} (``finite range integration").  We prescribed the upper bound $10^{-15}$ for the relative error estimates.

For the infinite part, $I_2(\mathbf{a},\boldsymbol{\nu},m,x_0)$, the asymptotic expansion of the Bessel functions  \cite[7.21 (1)]{Wa44}  is invoked in order to write
\[ J_\nu(r)=e^{ir}F(\nu,r)+e^{-ir}\overline{F(\nu,r)}, \]
where $F(\nu,r)=(2\pi r)^{-1/2}\exp(-i(\frac{\pi}{2}\nu+\frac{\pi}{4}))(P(\nu,r)+iQ(\nu,r))$, for certain functions $P(\nu,r)$ and $Q(\nu,r)$ which admit  asymptotic expansions in powers of $r^{-1}$; see \cite[7.3 (1)]{Wa44}. Truncating the expansions to $n+1$ terms yields corresponding functions $P_n,\,Q_n, F_n$, in such a way that $J_\nu$ is approximated by $J_{\nu,n}(r):=e^{ir}F_n(\nu,r)+e^{-ir}\overline{F_n(\nu,r)}$. This leads to the definition of an integral similar to $I_2(\mathbf{a},\boldsymbol{\nu},m,x_0)$, but now with the product of the $J_{\nu_i,n}(a_ir)$'s. Expanding out the resulting integrand, $r^m\prod_{i=1}^kJ_{\nu_i,n}(a_ir)$, leads to integrals of the form
\begin{equation}\label{eq:reducedToigamma}
\int_{x_0}^\infty e^{i\eta_\ell r}r^{m-k/2-j}\d r,\quad \ell=1,2,\dots,2^{k-1}, j=0,1,\dots,k(2n+1),
\end{equation}
where $\ell$ indexes all the possible sign combinations of $\eta_\ell:=a_1\pm a_2\pm\dots\pm a_k$. 
The integral in \eqref{eq:reducedToigamma} can be expressed in terms of the incomplete Gamma function $\Gamma(a,x)$, introduced in \S \ref{sec:proofLargeq}, via the identity
\begin{equation}\label{eq:iGammaDef}
\int_{x_0}^\infty e^{i\alpha r}r^{\beta}\d r=\Bigl(\frac{i}{\alpha}\Bigr)^{\beta+1}\Gamma(\beta+1,-i\alpha x_0),\quad \alpha\neq 0.
\end{equation}
In turn, \eqref{eq:iGammaDef} and Legendre's continuous fraction expansion for the incomplete Gamma function, see 
\cite[Eq.\@ 8.358]{GR07}, are used to efficiently evaluate each integral in \eqref{eq:reducedToigamma} for which $\eta_\ell\neq 0$, the case $\eta_\ell=0$ being handled by direct integration. The function \texttt{besselint} includes the sub-routine \texttt{igamma} to calculate these terms. It should be pointed out that the values of $x_0, n$ are determined in such a way as to minimize a certain cost function, which takes into account the cost of evaluating the Bessel functions at the nodes of the quadrature in the integral from $0$ to $x_0$, versus the cost of evaluating the incomplete Gamma function for the different values of $\ell$ and $j$ as in \eqref{eq:reducedToigamma}. Typical values in the construction of Table \ref{table:2to11Edq} were $x_0\sim 500$ and $n\sim 3$.

For large values of $k$, we found that the function \texttt{besselint} performs too slowly. 
The \textit{Octave} profiler tool was used to analyze the distribution of time spent in the various sub-routines of \texttt{besselint}, and in this way we found that the reduced speed was due to the exponential increase in the number of evaluations of the function \texttt{igamma} and the calculation of the coefficients associated to each integral in \eqref{eq:reducedToigamma}. Therefore, we exploited the explicit form of the particular case under consideration in order to speed up the calculation when $k\geq 14$. 
We proceed to explain these modifications. The case of interest is the integral
\[ I(\one_m,\nu\one_m,d-1-m\nu)=\int_0^\infty r^{d-1-m\nu}J_\nu(r)^m\d r, \]
where $m\geq 4$ is an even integer, and $\one_m\in\R^m$ is the vector all whose components are equal to $1$. In order to build Table \ref{table:2to11Edq}, we have $d\in[2,11]\cap\N$, $\nu=(d-2)/2\in[0,9/2]\cap\frac12\Z$ and
$m\in[4,30]\cap2\N$, but it is only necessary to introduce modifications for $m\geq 14$ since the lower values of $m$ are quickly computable by \texttt{besselint}. 
In what follows, let us consider $\mathbf{a}=\one_m$ and $\boldsymbol{\nu}=\nu\one_m$, for $\nu\geq 0$.

The following uniform upper bounds are known to hold, provided $\nu\in[0,1/2]$:
\begin{equation}\label{eq:unifUpperJsmall}
\ab{J_\nu(r)}\leq \Bigl(\frac{2}{\pi r}\Bigr)^{1/2},\quad \text{for every }r\geq 0;
\end{equation}
see \cite[\S 3.2]{La00}.
On the other hand,
from \cite[Theorem 3]{Kr14}, we have that
\[ \ab{J_\nu(r)}\leq \biggl(\frac{2}{\pi\ab{r^2-(\nu^2-\frac{1}{4})}}\biggr)^{1/2},\quad \text{for every }\nu\geq \frac12\text{ and }r\geq 0. \]
Combining these bounds, we conclude that, for all $\nu\geq 0$,
\begin{equation}\label{eq:unifUpperJlarge}
\ab{J_\nu(r)}\leq  \Bigl(\frac{5}{2\pi r}\Bigr)^{1/2},\quad \text{for every } r\geq\frac{5}{3}\Bigl(\nu^2-\frac{1}{4}\Bigr)^{1/2};
\end{equation}
see also \cite[Cor.\@ 2.6]{OST17}.
Given  $\epsilon>0$, let $\widetilde I_1(\mathbf{a},\boldsymbol{\nu},m,x_0)$ denote the estimated value of $I_1(\mathbf{a},\boldsymbol{\nu},m,x_0)$, with estimated relative error smaller than $\epsilon$ as in \eqref{eq:preceq} below, which is calculated using \texttt{besselint}'s sub-routine \texttt{fri} (as explained in the paragraph following \eqref{eq:I1I2}). Choose $x_0=x_0(d,m,\nu,\epsilon)\geq (5/3)(\nu^2-1/4)^{1/2}$, in such a way that
\[ \int_{x_0}^\infty r^{d-1-m\nu}\Bigl(\frac{5}{2\pi r}\Bigr)^{m/2}\d r<\epsilon \widetilde I_1(\mathbf{a},\boldsymbol{\nu},d-1-m\nu,x_0). \]
If $m\geq d/(\nu+1/2)$, then this leads to the inequalities
\begin{equation}\label{eq:nonlinearNum}
\Bigl(\frac{5}{2\pi}\Bigr)^{m/2}\frac{x_0^{d-m(\nu+\frac{1}{2})}}{m(\nu+\frac{1}{2})-d}<\epsilon \widetilde I_1(\mathbf{a},\boldsymbol{\nu},d-1-m\nu,x_0),\quad \frac{5}{3}\Bigl(\nu^2-\frac{1}{4}\Bigr)^{1/2}\leq x_0,
\end{equation}
which together are used to determine a suitable value for $x_0$. The cases of present interest are $d\in[2,11]\in\N$, i.e. $\nu=(d-2)/2\in [0,9/2]\cap\frac12\Z$, and  $m\in[14,30]\cap2\N$, for which in Table \ref{table:cutoffx0} we compiled an array of integer values of $x_0$ satisfying both inequalities in \eqref{eq:nonlinearNum} with $\epsilon=10^{-15}$; these values are optimal in the sense that no smaller integer would work.
In this way,
\[ \ab{I(\mathbf{a},\boldsymbol{\nu},d-1-m\nu)-I_1(\mathbf{a},\boldsymbol{\nu},d-1-m\nu,x_0)}<\epsilon \widetilde I_1(\mathbf{a},\boldsymbol{\nu},d-1-m\nu,x_0). \]
Through the use of the function \texttt{fri}, we have that
\begin{equation}\label{eq:preceq}
\ab{I_1(\mathbf{a},\boldsymbol{\nu},d-1-m\nu,x_0)-\widetilde I_1(\mathbf{a},\boldsymbol{\nu},d-1-m\nu,x_0)}\preceq\epsilon \widetilde I_1(\mathbf{a},\boldsymbol{\nu},d-1-m\nu,x_0),
\end{equation}
where "$\preceq$" indicates that, strictly speaking, this is not a real upper bound; rather, it is obtained through the error analysis coded into the numerical routine, and equals the absolute value of the difference between the 15- and the 19-point Gauss--Legendre quadrature rules as discussed earlier. In this way, we obtain \[\ab{I(\mathbf{a},\boldsymbol{\nu},d-1-m\nu)-\widetilde I_1(\mathbf{a},\boldsymbol{\nu},d-1-m\nu,x_0)}\preceq 2\epsilon\widetilde I_1(\mathbf{a},\boldsymbol{\nu},d-1-m\nu,x_0),\]
i.e. an estimated relative precision of $2\epsilon=2\cdot 10^{-15}$.

\begin{table}
	\centerline{
		\begin{tabular}{|c||c|c|c|c|c|c|c|c|c|}
			\hline 
			\backslashbox{$d$}{$m$}&14  &16  &18  &20  &22  &24  &26  &28 &30  \\ 
			\hline \hline
			2& 783 & 246 & 108 & 59 & 36 & 25 & 18 & 14 & 12  \\ 
			\hline
			3& 27 & 18 & 12 & 9 & 7 & 6 & 6 & 5 & 5 \\ 
			\hline 
			4& 12 & 9 & 7 & 6 & 5 & 5 & 5 & 4 & 4\\ 
			\hline
			5& 9 & 7 & 6 & 6 & 5 & 5 & 4 & 4 & 4\\ 
			\hline 
			6& 8 & 7 & 6 & 6 & 5 & 5 & 5 & 4 & 4\\ 
			\hline
			7& 8 & 7 & 6 & 6 & 5 & 5 & 5 & 5 & 5\\ 
			\hline
			8& 8 & 7 & 6 & 6 & 6 & 5 & 5 & 5 & 5\\ 
			\hline
			9& 8 & 7 & 7 & 6 & 6 & 6 & 6 & 6 & 6 \\ 
			\hline
			10& 8 & 7 & 7 & 7 & 7 & 7 & 7 & 7 & 7   \\ 
			\hline
			11& 8 & 8 & 8 & 8 & 8 & 8 & 8 & 8 & 8   \\ 
			\hline
		\end{tabular}
	}
	\caption{Integer values of $x_0$ satisfying \eqref{eq:nonlinearNum} for $\epsilon=10^{-15}$.}
	\label{table:cutoffx0}
\end{table}

\begin{remark}
The work \cite{OST17} developed a robust scheme to deal with the precise numerical evaluation of Bessel integrals which could be adapted to handle the cases of present interest. 
In particular, precise upper bounds for the error estimates were obtained through analytic methods, yielding numerical values for the Bessel  integrals with at least 7  significant digits.
For the sake of concreteness, let us briefly recall the scheme in the particular case of the integral
$I:=\int_0^\infty J_0^6(r) r\,\d r.$ 
Following \cite[\S 8]{OST17}, we split the integral into $I=I_1+I_2+I_3$, where
\begin{equation}\label{eq:split}
I_1=\int_0^S J_0^6(r) r\,\d r, \text{  }I_2=\int_S^R J_0^6(r) r\,\d r,\text{ and } I_3=\int_R^\infty J_0^6(r) r\,\d r,
\end{equation}
$S=3600$ and $R=63000$.  The integrals $I_1, I_2$ are evaluated with a Newton--Cotes quadrature rule of degree 6, with step size 0.003 for $I_1$ and 0.05 for $I_2$.
The approximation error for these integrals, denoted $\eps_1,\eps_2$, can be estimated as in \cite{OST17} via complex analysis. 
In particular, the Cauchy integral formula 
implies the following estimate for the eighth derivative of the function $f(r):=J_0^6(r)r$, 
\[|f^{(8)}(r)|\leq8!e^6(S+1),\text{ for all }r\in[0,S],\]
which then translates into the bounds $|\eps_1|\leq 10^{-8}$ and $|\eps_2|\leq 10^{-10}$.
Finally, the tail integral $I_3$ is approximated by analytic methods, taking advantage of sharp asymptotic formulae  (in the spirit of  \eqref{eq:unifUpperJsmall} and \eqref{eq:unifUpperJlarge})  which quantify \eqref{eq:asymptJnu}, with approximation error $\eps_3$ satisfying $|\eps_3|\leq 10^{-8}$. 
This strategy yields an estimated value $I\sim 0.3368280$, with precision $5\times  10^{-7}$.
We omit the details, and  invite the interested reader to consult the original work \cite{OST17}, together with the recent generalization \cite{OSTZK18}.
\end{remark}


\begin{thebibliography}{03}

\bibitem{AK17}
\textsc{F. Abi-Khuzam},
\newblock{\it Inequalities and asymptotics for some moment integrals.}
\newblock{ J. Inequal. Appl. 2017, Paper No.~257, 8 pp.}	

\bibitem{AFG12}
\textsc{F. Antoneli, M. Forger and P. Gaviria},
\newblock {\it Maximal subgroups of compact Lie groups.}
\newblock J. Lie Theory {\bf 22} (2012), no.~4, 949--1024. 

\bibitem{Ba86}
\textsc{K. Ball},
\newblock{\it Cube slicing in ${\R}^n$.}
\newblock Proc. Amer. Math. Soc. {\bf 97} (1986), no.~3, 465--473.

\bibitem{BC09}
\textsc{J. Borwein and O-Y. Chan},
\newblock{\it Uniform bounds for the complementary incomplete gamma function.}
\newblock{Math. Inequal. Appl. {\bf 12} (2009), no. 1, 115--121.}

\bibitem{BS}
\textsc{J. Borwein and C. Sinnamon},
\newblock {\it A closed form for the density functions of random walks in odd dimensions.}
\newblock Bull. Aust. Math. Soc. {\bf 93} (2016), no.~2, 330--339.

\bibitem{BSV16}
\textsc{J. Borwein, A. Straub and C. Vignat},
\newblock {\it Densities of short uniform random walks in higher dimensions.}
\newblock J. Math. Anal. Appl. {\bf 437} (2016), no.~1, 668--707. 

\bibitem{BSWZ12}
\textsc{J. Borwein, A. Straub, J. Wan and W. Zudilin},
\newblock {\it Densities of short uniform random walks.}
\newblock With an appendix by Don Zagier. Canad. J. Math. {\bf 64} (2012), no.~5, 961--990.

\bibitem{Br88}
\textsc{C. Brislawn},
\newblock {\it Kernels of trace class operators.} 
\newblock Proc. Amer. Math. Soc. {\bf 104} (1988), no.~4, 1181--1190. 

\bibitem{Br91}
\bysame
\newblock {\it Traceable integral kernels on countably generated measure spaces.} 
\newblock Pacific J. Math. {\bf 150} (1991), no.~2, 229--240. 

\bibitem{BOSQ18}
\textsc{G. Brocchi, D. Oliveira e Silva and R. Quilodr\'an},
\newblock {\it Sharp Strichartz inequalities for fractional and higher order Schr\"odinger equations.}
\newblock Anal. PDE {\bf 13} (2020), no.~2, 477--526.

\bibitem{Br11}
\textsc{P. Brzezinski},
\newblock{\it Schnittvolumina hochdimensionaler konvexer K\"orper.}
\newblock PhD thesis (2011), Christian-Albrechts-Universit\"at zu Kiel.

\bibitem{Br13}
\bysame
\newblock{\it Volume estimates for sections of certain convex bodies.}
\newblock Math. Nachr. {\bf 286} (2013), no. 17-18, 1726--1743.

\bibitem{COS15}
\textsc{E. Carneiro and D. Oliveira e Silva},
\newblock {\it Some sharp restriction inequalities on the sphere.}
\newblock Int. Math. Res. Not. IMRN (2015), no.~17, 8233--8267.

\bibitem{CFOST15}
\textsc{E. Carneiro, D. Foschi, D. Oliveira e Silva and C. Thiele},
\newblock {\it A sharp trilinear inequality related to Fourier restriction on the circle.}
\newblock  Rev. Mat. Iberoam. {\bf 33} (2017), no.~4, 1463--1486.

\bibitem{COSS19} 
\textsc{E. Carneiro, D. Oliveira e Silva and M. Sousa},
\newblock {\it Sharp mixed norm spherical restriction.}
\newblock   Adv. Math. {\bf 341} (2019), 583--608.

\bibitem{Ch11}
\textsc{M. Christ},
\newblock{\it Extremizers of a Radon transform inequality.}
\newblock Advances in analysis: the legacy of Elias M. Stein, 84--107, Princeton Math. Ser., 50, Princeton Univ. Press, Princeton, NJ, 2014. 

\bibitem{CS12a} 
\textsc{M. Christ and S. Shao},
\newblock {\it Existence of extremals for a Fourier restriction inequality.}
\newblock Anal. PDE. {\bf 5} (2012), no.~2, 261--312.

\bibitem{CS12b} 
\bysame
\newblock {\it On the extremizers of an adjoint Fourier restriction inequality.}
\newblock Adv. Math. {\bf 230} (2012), no.~3, 957--977.

\bibitem{dL14}
\textsc{J. L. deLyra},
\newblock{\it On the sums of inverse even powers of zeros of regular Bessel functions.}
\newblock Preprint at arXiv:1305.0228 [math-ph].

\bibitem{Di15}
\textsc{H. Dirksen},
\newblock{\it Sections of simplices and cylinders, volume formulas and estimates.}
\newblock PhD dissertation (2015), Christian-Albrechts-Universit\"at zu Kiel.

\bibitem{Di17}
\bysame
\newblock{\it Hyperplane sections of cylinders.}
\newblock Colloq. Math. {\bf 147} (2017), no. 1, 145--164.

\bibitem{FVV11}
\textsc{L. Fanelli, L. Vega and N. Visciglia}, 
\newblock {\it On the existence of maximizers for a family of restriction theorems.}
\newblock Bull. Lond. Math. Soc. {\bf 43} (2011), no.~4, 811--817.

\bibitem{Fo15}
\textsc{D. Foschi},
\newblock {\it Global maximizers for the sphere adjoint Fourier restriction inequality.}
\newblock J. Funct. Anal. {\bf 268} (2015), 690--702.

\bibitem{FOS17} 
\textsc{D. Foschi and D. Oliveira e Silva},
\newblock {\it Some recent progress on sharp Fourier restriction theory.} 
\newblock Anal. Math. {\bf 43} (2017), no.~2, 241--265. 

\bibitem{FLS16} 
\textsc{R. Frank, E. H. Lieb and J. Sabin},
\newblock {\it Maximizers for the Stein--Tomas inequality.}
\newblock Geom. Funct. Anal. {\bf 26} (2016), no.~4, 1095--1134. 

\bibitem{GP12}
\textsc{R. Garc\'ia-Pelayo},
\newblock {\it Exact solutions for isotropic random flights in odd dimensions.}
\newblock J. Math. Phys. {\bf 53}, 103504 (2012).

\bibitem{GRR69}
\textsc{A. Garsia, E. Rodemich and H. Rumsey},
\newblock {\it On some extremal positive definite functions.}
\newblock J. Math. Mech. {\bf 18} 1968/1969 805--834. 

\bibitem{GR07}
\textsc{I. S. Gradshteyn and I. M. Ryzhik},
\newblock{Table of integrals, series, and products. Translated from the Russian. Translation edited and with a preface by Alan Jeffrey and Daniel Zwillinger.}
\newblock Seventh edition. Elsevier/Academic Press, Amsterdam, 2007. 

\bibitem{GD55}
\textsc{J. A. Greenwood and D. Durand},
\newblock{\it The distribution of length and components of the sum of $n$ random unit vectors.} 
\newblock Ann. Math. Statist. {\bf 26} (1955), 233--246.

\bibitem{Ha15}
\textsc{B. Hall},
\newblock Lie groups, Lie algebras, and representations. An elementary introduction. 
\newblock Second edition. Graduate Texts in Mathematics, 222. Springer-Verlag, New York, 2015. 

\bibitem{Ho65}
\textsc{K. H. Hofmann}, 
\newblock {\it Lie algebras with subalgebras of co-dimension one.}
\newblock Illinois J. Math. {\bf 9} (1965), 636--643. 

\bibitem{IS90}
\textsc{E. K. Ifantis and P. D. Siafarikas},
\newblock{\it Inequalities involving Bessel and modified Bessel functions.}
\newblock J. Math. Anal. Appl. {\bf 147} (1990), no. 1, 214--227.

\bibitem{JS72} 
\textsc{V. Jurdjevic and H. J. Sussmann}, 
\newblock {\it Control systems on Lie groups.}
\newblock J. Differential Equations {\bf 12} (1972), 313--329. 

\bibitem{KOS15}
\textsc{R. Kerman, R. Ol'hava and S. Spektor}
\newblock{\it An asymptotically sharp form of Ball's integral inequality.}
\newblock{Proc. Amer. Math. Soc. {\bf 143} (2015), no. 9, 3839--3846.}

\bibitem{Ke83}
\textsc{D. Kershaw},
\newblock{\it Some extensions of W. Gautschi's inequalities for the gamma function.}
\newblock Math. Comp. {\bf 41} (1983), no.~164, 607--611.

\bibitem{Ko14}
\textsc{H. K\"onig},
\newblock{\it On the best constants in the Khintchine inequality for Steinhaus variables.}
\newblock Israel J. Math. {\bf 203} (2014), no. 1, 23--57.

\bibitem{KK13}
\textsc{H. K\"onig and A. Koldobsky},
\newblock{\it On the maximal measure of sections of the $n$-cube.}
\newblock Geometric analysis, mathematical relativity, and nonlinear partial differential equations, 123--155, Contemp. Math., 599, Amer. Math. Soc., Providence, RI, 2013.

\bibitem{KK19}
\bysame
\newblock{\it On the maximal perimeter of sections of the cube.}
\newblock Adv. Math. {\bf 346} (2019), 773--804.

\bibitem{KK01}
\textsc{H. K\"onig and S. Kwapie\'n},
\newblock{\it Best Khintchine type inequalities for sums of independent, rotationally invariant random vectors.}
\newblock Positivity {\bf 5} (2001), no. 2, 115--152. 

\bibitem{Kr14}
\textsc{I. Krasikov},
\newblock{\it Approximations for the Bessel and Airy functions with an explicit error term.}
\newblock LMS J. Comput. Math. {\bf 17} (2014), no.~1, 209--225.

\bibitem{KR50} 
\textsc{M. G. Krein and M. A. Rutman},
\newblock {\it Linear operators leaving invariant a cone in a Banach space.}
\newblock Amer. Math. Soc. Translation 1950, (1950). no. 26, 128 pp. 

\bibitem{La84}
\textsc{A. Laforgia},
\newblock{\it Further inequalities for the gamma function.}
\newblock Math. Comp. {\bf 42} (1984), no. 166, 597--600.

\bibitem{LM83}
\textsc{A. Laforgia and M. E. Muldoon},
\newblock{\it Inequalities and approximations for zeros of Bessel functions of small order.}
\newblock SIAM J. Math. Anal. {\bf 14} (1983), no. 2, 383--388.

\bibitem{La00}
\textsc{L. J. Landau},
\newblock{\it Bessel functions: monotonicity and bounds.}
\newblock J. London Math. Soc. (2) {\bf 61} (2000), no.~1, 197--215.

\bibitem{Ma66}  
\textsc{L. N. Mann}, 
\newblock {\it Gaps in the dimensions of transformation groups.}
\newblock Illinois J. Math. {\bf 10} (1966), 532--546. 

\bibitem{Mo17}
\textsc{O. Mordhorst},
\newblock{\it The optimal constants in Khintchine's inequality for the case $2<p<3$.}
\newblock Colloq. Math. {\bf 147} (2017), no. 2, 203--216. 

\bibitem{DLMF}
\newblock{\it NIST Digital Library of Mathematical Functions.}
\newblock \url{http://dlmf.nist.gov/}, Release 1.0.23 of 2019-06-15,
 F.~W.~J. Olver, A.~B. {Olde Daalhuis}, D.~W. Lozier, B.~I. Schneider,
	R.~F. Boisvert, C.~W. Clark, B.~R. Miller and B.~V. Saunders, eds.
	
\bibitem{NP00a}
\textsc{P. Natalini and B. Palumbo},
\newblock{\it Inequalities for the incomplete gamma function.}
\newblock Math. Inequal. Appl. {\bf{3}} (2000), no. 1, 69--77.

\bibitem{NP00}
\textsc{F. L. Nazarov and A. N. Podkorytov},
\newblock{\it Ball, Haagerup, and distribution functions.}
\newblock{Complex analysis, operators, and related topics, 247--267, Oper. Theory Adv. Appl., 113, Birkhäuser, Basel, 2000.}

\bibitem{OP00}
\textsc{K. Oleszkiewicz and A. Pe\l{}czy\'nski},
\newblock{\it Polydisc slicing in $\Co^n$.}
\newblock Studia Math. {\bf 142} (2000), no. 3, 281--294.

\bibitem{OSQ19}
\textsc{D. Oliveira e Silva and R. Quilodr\'an},
\newblock{\it Smoothness of solutions of a restricted convolution equation on spheres.}
Preprint at arXiv:1909.10220 [math.CA].
To appear in Forum Math.\@ Sigma.

\bibitem{OST17}
\textsc{D. Oliveira e Silva and C. Thiele},
\newblock{\it Estimates for certain integrals of products of six Bessel functions.}
\newblock Rev. Mat. Iberoam. {\bf 33} (2017), no.~4, 1423--1462. 

\bibitem{OSTZK18}
\textsc{D. Oliveira e Silva, C. Thiele and P. Zorin-Kranich},
\newblock {\it Band-Limited Maximizers for a Fourier Extension Inequality on the Circle.}
\newblock Experimental Mathematics (2019),  DOI: 10.1080/10586458.2019.1596847.

\bibitem{Pe}
\textsc{K. Pearson},
\newblock{\it Mathematical contributions to the theory of evolution. IX.  A mathematical theory of random migration.}
\newblock{Drapers' Company Research Memoirs, Biometric Series, No. 3, 1906.}
 
\bibitem{PDKUK}
\textsc{R. Piessens, E. de Doncker-Kapenga, C.W. Uberhuber and D.K. Kahaner},
\newblock{ \texttt{QUADPACK}\textit{: A Subroutine Package for Automatic Integration}.}
\newblock Berlin: Springer-Verlag, 1983, ISBN 0387125531. 
 
\bibitem{Sa85}
\textsc{J. Sawa},
\newblock{\it The best constant in the Khintchine inequality for complex Steinhaus variables, the case $p=1$.}
\newblock Studia Math. {\bf 81} (1985), no. 1, 107--126.
 
\bibitem{Sh16}
\textsc{S. Shao},
\newblock {\it On existence of extremizers for the Tomas--Stein inequality for $\sph{1}$.}
\newblock J. Funct. Anal. {\bf 270} (2016), 3996--4038.

\bibitem{Sn60}
\textsc{I. N. Sneddon},
\newblock{\it On some infinite series involving the zeros of Bessel functions of the first kind.}
\newblock Proc. Glasgow Math. Assoc. {\bf 4} (1960), 144--156 (1960).

\bibitem{St93} 
\textsc{E. M. Stein},
\newblock Harmonic Analysis: Real-Variable Methods, Orthogonality, and Oscillatory Integrals.
\newblock Princeton University Press, Princeton, NJ, 1993.

\bibitem{St20}
\textsc{H. Steinhaus,}
\newblock{\it Sur les distances des points dans les ensembles de mesure positive.}
\newblock Fund. Math. {\bf 1} (1920), 93--104.

\bibitem{St72} 
 \textsc{K. Stromberg}, 
 \newblock{\it An elementary proof of Steinhaus's theorem.}
 \newblock Proc. Amer. Math. Soc. {\bf 36} (1972), 308.
 
\bibitem{To75}
\textsc{P. Tomas}, 
\newblock {\it A restriction theorem for the Fourier transform.}
\newblock Bull. Amer. Math. Soc. {\bf 81} (1975), no.~2, 477--478.


\bibitem{vDC06a}
\textsc{J. Van Deun and R. Cools},
\newblock{\it Algorithm 858: computing infinite range integrals of an arbitrary product of Bessel functions.}
\newblock Associated computer program available online. ACM Trans. Math. Software 32 (2006), no. 4, 580--596.

\bibitem{vDC06b}
\bysame
\newblock{\it A Matlab implementation of an algorithm for computing integrals of products of Bessel functions.}
\newblock Mathematical software--ICMS 2006, 284--295, Lecture Notes in Comput. Sci., 4151, Springer, Berlin, 2006. 

\bibitem{vDC08}
\bysame
\newblock{\it Integrating products of Bessel functions with an additional exponential or rational factor.}
\newblock Associated computer program available online. Comput. Phys. Comm. 178 (2008), no. 8, 578--590.

\bibitem{Wa44}
\textsc{G. N. Watson}.
\newblock{A Treatise on the Theory of Bessel Functions.}
\newblock Cambridge University Press, Cambridge, England; The Macmillan Company, New York, 1944. 

\bibitem{We18} 
\textsc{D. Werner},
\newblock Funktionalanalysis.
\newblock 8. Auflage. Springer-Verlag, Berlin, 2018.

\bibitem{Zh17a}
\textsc{Y. Zhou},
\newblock{\it Wick rotations, Eichler integrals, and multi-loop Feynman Diagrams.}
\newblock  Commun. Number Theory Phys. {\bf 12} (2018), no.~1, 127--192.

\bibitem{Zh17b}
\bysame
\newblock{\it On Borwein's conjectures for planar uniform random walks.}
\newblock  J. Aust. Math. Soc. {\bf 107} (2019), no.~3, 392--411. 
\end{thebibliography}
\end{document}